\documentclass[12pt]{article}

\usepackage{graphicx}
\usepackage[utf8]{inputenc}
\usepackage[numbers]{natbib}
\usepackage{amsmath} 
\usepackage{bbm}
\usepackage{amssymb,amsthm}
\usepackage{mathrsfs}
\usepackage{indentfirst}
\usepackage{geometry}
\usepackage{url}
\usepackage{comment}
\usepackage{transparent}
\usepackage{enumitem}

\allowdisplaybreaks[4]
\numberwithin{equation}{section}
\newtheorem{assumption}{Assumption}[section]
\newtheorem{definition}{Definition}[section]
\newtheorem{theorem}[definition]{\textbf{Theorem}}
\newtheorem{lemma}[definition]{\textbf{Lemma}}
\newtheorem{proposition}[definition]{\textbf{Proposition}}
\newtheorem{corollary}[definition]{\textbf{Corollary}}
\newtheorem{remark}[definition]{\textbf{Remark}}
\usepackage[colorlinks=true, allcolors=blue]{hyperref}
\usepackage{xcolor}

\def\r{{\mathbb R}}
\def\h{{\mathbb H}}
\def\e{{\mathbb E}}
\def\p{{\mathbb P}}
\def\P{{\bf P}}
\def\z{{\mathbb Z}}
\def\N{{\mathbb N}}
\def\d{\, \mathrm{d}}
\def\bD{{\mathbb D}}
\def\bI{{\mathbb I}}
\def\cD{\mathcal D}
\def\cE{\mathcal E}
\def\cL{\mathcal L}
\def\cH{\mathcal H}
\def\cW{\mathcal W}
\def\cWhat{\widehat {\mathcal W}}
\def\cM{\mathcal M}
\def\cS{\mathcal S}
\def\cX{\mathcal X}
\def\fD{\mathfrak D}
\def\fR{\mathfrak R}
\def\fB{\mathfrak B}
\def\cY{\mathcal Y}

\def\cF{\mathcal F}
\def\cShat{\widehat {\mathcal S}}
\def\df{\mathfrak{d}}
\newcommand{\dd}{\ensuremath{\mathrm{d}}}
\newcommand{\ind}{\mathbbm{1}}

\setlist[enumerate]{
    topsep=6pt,
    itemsep=4pt,
    partopsep=0pt,
    parsep=4pt
}

\title{\Large{\textbf{ The Polynomially Self-repelling Motion }  }}
\author{Yaolin Yu\footnote{\scriptsize SMS, Fudan University, China, \texttt{ylyu23@m.fudan.edu.cn}}}
\date{}

\begin{document}

\maketitle

\begin{abstract}
Kosygina et al.~\cite{Elena23} showed that the rescaled polynomially self-repelling walk does not converge to any Brownian motion perturbed at extrema (BMPE). Motivated by this result, we introduce an explicit candidate for the scaling limit, the polynomially self-repelling motion (PSRM), and study its properties. The construction relies on a new family of continuous stochastic flows, called polynomially self-repelling flows, introduced in~\cite{Yu26}. A key ingredient in relating the limiting process of the rescaled walk to the PSRM is a joint Ray--Knight theorem for rescaled local time processes, extending the classical Ray--Knight theorems of B.~T\'oth~\cite{Toth96}.
\end{abstract}

{\bf Classification}. 60F17, 60H10, 60J55, 60K35. 

{\bf Keywords}. Ray--Knight theorems, self-interacting random walks, local time flow, stochastic differential equation, generalized P\'{o}lya urn.

\section{Introduction}
Let $\Omega=\{(\omega_i)_{i\ge 0}:\, \omega_0=0,\, |\omega_i-\omega_{i-1}|=1,\, i\ge1\}$ be the set of nearest-neighbor paths starting from the origin. Set $\ell_x^0=r_x^0=0$ for $x\in \z$ and for each $\omega\in \Omega$ define
$$r_x^i(\omega) := \sum_{j=1}^i \mathbbm{1}_{\{\{\omega_{j-1},\, \omega_j\}=\{x,\, x+1\}\}},
\qquad
\ell_x^i(\omega) := \sum_{j=1}^i \mathbbm{1}_{\{\{\omega_{j-1},\, \omega_j\}=\{x-1,\, x\}\}}, \quad i\in \N.$$
They represent the numbers of crossings of the undirected edges $\{x,x+1\}$ and $\{x-1,x\}$ up to time $i$, respectively. 
Let $w:\N_0\to(0,\infty)$ be a given weight function. The self-interacting random walk (SIRW) $X=(X_i)_{i\ge 0}$ on $\Omega$ is defined recursively by, for $i\ge 0$,
\begin{align*}
    P\left(X_{i+1} = X_i + 1 \mid X_1, \ldots, X_i\right)=&\; \frac{w\!\left(r_{X_i}^i\right)}{w\!\left(r_{X_i}^i\right) + w\!\left(\ell_{X_i}^i\right)},\\
    P\left(X_{i+1} = X_i - 1 \mid X_1, \ldots, X_i\right)=&\; \frac{w\!\left(\ell_{X_i}^i\right)}{w\!\left(r_{X_i}^i\right) + w\!\left(\ell_{X_i}^i\right)}.
\end{align*}
If $w$ is non-increasing (non-decreasing), the walk is self-repelling (self-attracting). These non-Markovian processes have been extensively studied in probability theory and statistical physics, especially their long-time behavior. Different growth regimes of the weight function $w$ at infinity lead to distinct classes of SIRWs~\cite{Davis96, pemantle88, toth94, toth95, Toth96, wrrw, toth08}. We refer to~\cite{APP83, DR92, Lawler13, SAW, OP83, Peliti86, toth_survey} for further background and related literature.

\medskip

The \emph{polynomially self-repelling walk} (PSRW) is a distinguished member of the class of SIRWs, corresponding to polynomially decaying interaction weights. More precisely, it is the SIRW whose weight function $w$ is non-increasing and satisfies
\begin{align}\label{eq:weight_func}
    \frac{1}{w(n)}=n^\alpha\left(1+\frac{2 B}{n}+O\left(\frac{1}{n^{2}}\right)\right) \quad \text { as } n \rightarrow \infty ,
\end{align}
where $\alpha>0$ and $B\in\r$ are fixed parameters.
This model goes back to T\'oth~\cite{Toth96} and is characterized by $w(n)\sim n^{-\alpha}$. We highlight below several results from his work:
\begin{itemize}[itemsep=0pt, left=4pt]
\item A generalized Ray--Knight theorem for local time processes stopped at inverse local times (\cite[Theorem~1B]{Toth96}, recalled as Theorem~\ref{thm:1B_toth} below), together with the weak convergence of the inverse local times (\cite[Corollary 1B]{Toth96}). 

\item A local limit theorem observed at an independent geometric time (\cite[Theorem~2B]{Toth96}) and the tightness of the rescaled sequence $X_{\lfloor nt\rfloor}/\sqrt{n}$ for fixed $t$. It was further noted that Brownian motions perturbed at extrema (BMPE) with suitable parameters~\cite{CPY98, Davis96} exhibit the same local time processes as those in \cite[Theorem~1B]{Toth96}.
\end{itemize}

\noindent
These results suggest that the diffusively rescaled process $(X_{\lfloor nt\rfloor}/\sqrt{n})_{t\ge0}$ should converge, as $n\to\infty$, to a nontrivial limit. Establishing such a functional convergence, however, was left open (see \cite[Remark~1.7]{Elena23}). In fact, Kosygina et al.~\cite{Elena23} showed that, although BMPEs reproduce the marginal distributions predicted by the Ray--Knight theorem, they cannot serve as the scaling limit of the PSRW, thereby revealing the limitations of the Ray--Knight approach. Motivated by this, we propose in this work an explicit candidate for the scaling limit of the PSRW and provide evidence that it captures the correct limiting structure. In concurrent work, Kosygina et al.~\cite{Elena26} prove functional convergence of the diffusively rescaled PSRW to the PSRM constructed here. Together, the two works resolve the scaling limit problem for the PSRW.

As a side remark, we consider the case $\alpha=0$ in~\eqref{eq:weight_func}, referred to as \emph{asymptotically free} by T\'oth~\cite{Toth96}. We outline in Remark~\ref{rem:AF} how our joint Ray--Knight approach could extend to this case using the urn estimates in Appendix~\ref{appsub:AF}, and explain why the BMPE arises as the limiting process.

\bigskip
\bigskip

Our first result is a joint Ray--Knight theorem for PSRWs, stating that the rescaled upcrossing profiles converge in joint distribution to a family of limiting processes, each of which can be viewed as a generalized two-sided BESQ-type process; see Definition~\ref{def:RK_curve}. The counterexample in \cite[Section~3]{Elena23} demonstrates that a single Ray--Knight theorem is insufficient for functional convergence of the underlying walk, suggesting that a genuine understanding of the joint structure of the associated local time processes is necessary. A further motivation comes from the work of Kosygina and Peterson~\cite{Elena25}, who established the functional convergence of rescaled ``true'' self-avoiding walks (TSAW) to the ``true'' self-repelling motion (TSRM), introduced by T\'oth and Werner~\cite{Toth98}, by means of a joint generalized Ray--Knight theorem (GRKT). Here TSAW corresponds to the SIRW with exponential weight function $w(n)=e^{-\beta n}$, for $\beta>0$. This approach offers a blueprint for analyzing the functional convergence of other types of SIRWs.

\medskip

Fix $\alpha>0$ and $B\in \r$. Let $(X_n)_{n\ge0}$ be a PSRW corresponding to a non-increasing weight function satisfying~\eqref{eq:weight_func}. For each $k\in \z$ and $i\in\N$, let $\cE^k(i)$ denote the number of upcrossings of the directed bond $(k,k+1)$ by time $i$, that is,
\begin{align}\label{def_cE}
    \cE^k(i):=\sum_{j=0}^{i-1} \ind_{\{X_j=k,\; X_{j+1}=k+1\}}.
\end{align}
Set $T^{\cE}_{k,-1}=0$, and for $m\ge0$ define
\begin{align}\label{def_T_k_m}
    T^{\cE}_{k,m}:=\inf\{i>0 : \, \cE^k(i)>m\}
\end{align}
as the time of the $(m+1)$-st upcrossing of $(k,k+1)$. Then,  $\cE^k\bigl(T^{\cE}_{k,m}\bigr)=m+1$ for all $m\ge0$. For $(a,r)\in\h:=\r_+\times\r$ and $x\in\r$, define the rescaled upcrossing profile by
\begin{equation}\label{eq:def_S_N}
    S_{r,x}^N(a) := N^{-1}\,\cE^{\lfloor Nx\rfloor}\left(T^{\cE}_{\lfloor Nr\rfloor,\, \lfloor Na\rfloor}\right).
\end{equation}
Note that $S_{r,r}^N(a)=N^{-1}(\lfloor N a\rfloor+1)\to a$ as $N\to \infty$.

\medskip

We now turn to the continuum objects arising as limits of the above discrete quantities. Write $r^+=\max\{r,0\}$ and $r^-=\min\{r,0\}$. Let $\cW$ be a standard space--time white noise on $\h$ (see, e.g., \cite{Walsh1986AnIT}), and let $\cW^*$ denote its image under the map $(a,r)\mapsto(a,-r)$.

\begin{definition}\label{def:RK_curve}
    For each $(a,r)\in\h$, we define a process $(\cS_{r,x}(a),\, x\in\r)$, referred to as the \emph{Continuum Ray--Knight Curve (CRKC)}, as follows.
\begin{itemize}[itemsep=0pt, left=4pt]
    \item
    The forward branch $\bigl(\cS_{r,x}(a),\, x\ge r\bigr)$ is the almost sure unique strong solution to
    \begin{align}\label{SDE:for}
    \cS_{r,x}(a) = a + \int_r^x \frac{2}{\cS_{r,u}(a)^{\alpha}}  \int_{\r_+} s^\alpha \mathbbm{1}_{\{s\le \cS_{r,u}(a)\}} \cW(\dd s, \dd u) + \frac{x-r}{2\alpha+1},
    \end{align}
    which is reflected at $0$ and absorbed upon hitting $0$ after time $r^+$.
    
    \item
    The backward branch $\bigl(\cS_{r,-y}(a),\, y\ge -r\bigr)$ is the almost sure unique strong solution to
    \begin{equation}\label{SDE:back}
    \cS_{r,-y}(a) = a - \int_{-r}^{y} \frac{2}{\cS_{r,-u}(a)^{\alpha}} \int_{\r_+} s^{\alpha}\mathbbm{1}_{\{s\le \cS_{r,-u}(a)\}} \cW^{*}(\dd s,\dd u) + \frac{y+r}{2\alpha+1},
    \end{equation}
    which is reflected at $0$ and absorbed upon hitting $0$ after time $-r^-$.
\end{itemize}
\end{definition}

By verifying the conditions of \cite[Equation~(2.1)]{Dawson12}, the SDEs~\eqref{SDE:for} and~\eqref{SDE:back} admit unique strong solutions, and pathwise uniqueness holds. Let $\mathrm{BESQ}^{\delta}$ denote the squared Bessel process of dimension $\delta$ (see \cite[Definition~(1.1), Chapter~XI]{RevYor} and \cite[(28)]{GJY03}). Each CRKC $\cS_{r,\cdot}(a)$ may be viewed as the concatenation of two copies of $(2\alpha+1)^{-1}\mathrm{BESQ}^{1}$, evolving in the forward and backward directions.
The primary contribution of this paper is the following joint convergence result, which identifies the CRKCs as the limit of the discrete rescaled upcrossing profiles.

\begin{theorem}[Joint Ray--Knight]\label{thm:main}
    For any $p\in \mathbb{N}$ and $(a_i,r_i)\in \h $, $1\le i\le p$, the joint law of the processes $(2S_{r_i,x}^N(a_i/2),\, x\in \r)$, $1\le i\le p$, as defined in \eqref{eq:def_S_N}, converges as $N\to \infty$ to the joint law of CRKCs $(\cS_{r_i,x}(a_i),\, x\in \r)$, $1\le i\le p$.
\end{theorem}

\bigskip
\medskip

The second part of the paper concerns the construction of the limiting process, the PSRM, following the strategy of~\cite{Toth98} for the TSRM. We give an informal overview of the construction and refer to Sections~\ref{sec:PSR_flow} and~\ref{sec:constr_PSRM} for the full and precise definitions. Recall the polynomially self-repelling flow with index $\alpha$ (PSR$^\alpha$ flow) from~\cite[Theorem~6.2]{Yu26}, where each flow line starting from $(a,r)\in\h$ is the unique strong solution to~\eqref{SDE:for}, subject to the same reflecting/absorbing rule as the forward branch in Definition~\ref{def:RK_curve}. A key structural feature of the PSR flow is its self-duality, established in~\cite[Theorem~6.6]{Yu26}. This property already appears in the discrete setting of coalescing simple random walks~\cite{Arratia} and was later studied for families of independent coalescing reflected/absorbed Brownian motions by T\'oth and Werner~\cite{Toth98}. Let $\cS$ denote the PSR flow and $\cS^*$ its dual flow defined by~\eqref{eq:def_dual_+}. By stitching together the forward flow lines of $\cS$ with the backward flow lines of $\cS^*$ at their common starting points, we obtain a family of CRKCs on $\h$, which defines the two-sided flow $\overline{\cS}$; see~\eqref{eq:cS_bar}.

Consider the mapping $\h\ni (a,r)\mapsto D(a,r)$, where $D(a,r)$ denotes the closed planar region bounded by the curve $\overline{\cS}_{r,\cdot}(a)$ and the $y$-axis. The continuity and non-crossing property of the curves in $\overline{\cS}$ imply that the family $\{D(a,r):(a,r)\in\h\}$ is totally ordered by inclusion. This ordering gives rise to an injective area functional $T(a,r):=|D(a,r)|$ on $\h$ whose range is dense in $\r_+$. The quantity $T(a,r)$ can be interpreted as an inverse local time, in the sense that it records the time at which the motion has accumulated local time $a$ at position $r$. The PSRM $X$ is then defined implicitly by $X_{T(a,r)}=r$ and extended to all $t\ge0$ by continuity.

\bigskip
\bigskip

A\"id\'ekon et al.~\cite{elie23} show that a stochastic process can be recovered as a measurable functional of its local time flow. Several recent works, including \cite{elie24}, \cite{ASW25}, and \cite{AWY24}, build links between stochastic processes and their associated stochastic flows, using this viewpoint to study diverse questions. In the same spirit, properties of the PSRM can be derived from the structure of the PSR flow from which it is constructed. We finish the presentation with a summary of several such properties established in the present work.

\medskip

The PSRM $X$ is symmetric in law and has continuous sample paths almost surely. It is also recurrent in the sense that it visits every spatial point infinitely often. It also satisfies a diffusive scaling invariance: for any $a>0$, $(X_{a t})_{t\ge 0}$ and $(a^{1/2}X_t)_{t\ge 0}$ have the same law. Moreover, the pair $(X_t, L_t(\cdot))_{t\ge 0}$ is Markovian. In other words, the future evolution of $X$ depends only on its current position $X_t$ and its local time profile $L_t(\cdot)$, reflecting the self-interacting nature of $X$. The process also admits local times in the following sense:
\vspace{-0.3em}
\begin{definition}\label{def:local_time}
    We say that a process $\cX$ admits local times if there exists a family $\{\cL_\cX(t,x),\, t\ge 0,\, x\in \r\}$ such that the occupation times formula
    \begin{align}\label{eq:oc_times_for}
        \int_0^t h(\cX_s)\dd s=\int_\r h(x)\cL_\cX(t,x)\dd x
    \end{align}
    holds for every nonnegative Borel function $h$ and all $t\ge 0$. We refer to $\cL_\cX(t,x)$ as the local time of $\cX$ at level $x$ and time $t$.
\end{definition}
\vspace{-0.3em}
\noindent Furthermore, the local times of $X$ admit a continuous version, denoted by $L_t(\cdot)$, which is given in terms of the PSR flow via~\eqref{def_T} and~\eqref{eq:def_Lt}. Precise statements and proofs of these properties can be found in Section~\ref{subsec:prop_PSRM}.

\medskip

We also make two further observations concerning the sample-path behavior of the PSRM; see Section~\ref{subsec:pro_more}. The PSRM has no points of increase: almost surely, there do not exist $t,\varepsilon>0$ such that $L(t,X_t)>0$, $X_s\le X_t$ for all $s\in(t-\varepsilon,t)$, and $X_u\ge X_t$ for all $u\in(t,t+\varepsilon)$. The absence of points of increase for Brownian motion has been proven from different perspectives (see, e.g., \cite{non_incre_2, non_incre_1, Yuv96}). Related questions have also been studied for the Brownian sheet \cite{DRM97} and for L\'evy processes \cite{doney96}. In addition, on every fixed open interval, the maximum of $X$ is attained at a unique time almost surely. Moreover, the set of times at which $X$ attains a local maximum is almost surely countable. The same statements hold for the minimum and local minima. We refer to \cite[Theorem~2.11]{BM10} for the Brownian motion case.

\paragraph{Notation} 
Let $\bD$ denote the set of dyadic numbers, i.e., numbers of the form $k2^{-n}$ for $k \in \z$ and $n \in \N$, and let $\bD_+$ be the set of positive dyadic numbers. Set $\r_+ := [0, \infty)$, $\r_- := (-\infty,0]$, and define the half-plane $\h := \r_+ \times \r$ and its interior $\h^* := (0,\infty)\times\r$. Let $\cD:=\bD^2\cap\h$ be the set of dyadic pairs in $\h$. We write $a\vee b:=\max\{a,b\}$ and $a\wedge b:=\min\{a,b\}$.

\paragraph{Structure of the paper}
The proof of Theorem~\ref{thm:main} is given in Section~\ref{sec:joint_RK}, based on a coupling between the rescaled upcrossing profile and generalized P\'{o}lya urns. In Section~\ref{sec:PSR_flow}, we introduce the PSR flow and its dual, and collect useful properties. The PSRM is then constructed from the PSR flow in Section~\ref{sec:constr_PSRM}, where we also investigate several of its properties. Section~\ref{sec:deco} contains the spatial Markov property and decomposition results, which clarify the structure of the PSR flow. In Section~\ref{sec:perfect}, we further develop the perfect flow property. Finally, Appendix~\ref{app:est} provides estimates on the coalescence time of two trajectories in the PSR flow, while Appendix~\ref{app:discrepancy} collects results on discrepancy urn processes. Appendix~\ref{app:tight} treats the tightness of the collection of rescaled upcrossing profiles.

\paragraph{Acknowledgments}
This paper, including all its novel ideas and arguments, was conceived and written by the author. AI tools were used only for language polishing and correcting typographical errors during the final stage of manuscript preparation. We are grateful to Elena Kosygina for introducing this problem to us. We also thank Laure Mar\^ech\'e, Thomas Mountford, and Jonathon Peterson for many valuable discussions. Finally, we thank Elie A\"id\'ekon for his continued support and encouragement.

\section{Joint Ray--Knight theorem}
\label{sec:joint_RK}
For the proof of Theorem~\ref{thm:main}, we use the generalized P\'{o}lya urn description of the increments of the numbers of upcrossings; see \cite[Section~4]{Elena23} and \cite[Section~4]{Toth96}.

Given positive sequences $\{b(i)\}_{i \ge 0}$ and $\{r(i)\}_{i \ge 0}$, the generalized P\'olya urn process $(\fB_n, \fR_n)_{n \ge 0}$ is a Markov chain started from $(0,0)$, with transition probabilities
\begin{align*}
	\P\bigl((\fB_{n+1}, \fR_{n+1})=(i+1,j) \mid (\fB_n, \fR_n)=(i,j)\bigr) &= \frac{b(i)}{b(i)+r(j)},\\
	\P\bigl((\fB_{n+1}, \fR_{n+1})=(i,j+1) \mid (\fB_n, \fR_n)=(i,j)\bigr) &= \frac{r(j)}{b(i)+r(j)}.
\end{align*}
Here $\fB_n$ and $\fR_n$ count the blue and red draws, respectively. An equivalent realization, known as Rubin's construction, can be found in~\cite{Davis90}. Define the discrepancy urn process $\fD=\{\fD_n\}_{n\ge 0}$ as the signed difference between the numbers of red and blue balls
$$\fD_{n}=\fR_{n}-\fB_{n}.$$
Set $\tau_0^\fB=\tau_0^\fR=0$ and, for $k\ge1$, define
\begin{align*}
	\tau_{k}^{\fB}:=\inf \left\{n>0: \fB_{n}=k \right\},\qquad \tau_{k}^{\fR}:=\inf \left\{n>0: \fR_{n}=k \right\}.
\end{align*}

At each site $\ell\in\z$, take an independent generalized P\'olya urn process
$(\fB_n^{(\ell)},\fR_n^{(\ell)})_{n\ge0}$ with weights
\begin{align}\label{weights_on_Z}
	(b^{(\ell)}(j), r^{(\ell)}(j)) =
	\begin{cases}
		(w(2j), w(2j+1)), & \ell < 0,\\
		(w(2j+1), w(2j)), & \ell > 0,\\
		(w(2j), w(2j)), & \ell = 0,
	\end{cases}
\end{align}
where $w$ satisfies~\eqref{eq:weight_func}, with $\alpha>0$ and $B\in\r$.
Write $\fD_n^{(\ell)}:=\fR_n^{(\ell)}-\fB_n^{(\ell)}$. Starting from the origin, construct the walk by using the next draw from the urn at its current site for each departure: a blue draw gives a left jump, and a red draw gives a right jump. These weights reproduce the transition probabilities of the PSRW, so this construction realizes the walk and all its upcrossing profiles on a common probability space.
When no ambiguity arises, write $T_{k,m}:=T^\cE_{k,m}$, as defined in~\eqref{def_T_k_m}. We now express the increments of the two Markov chains
$(\cE^x(T_{k,m}),\,x\ge k)$ and $(\cE^x(T_{k,m}),\,x\le k)$ in terms of the urn processes. Both chains start from $\cE^k(T_{k,m})=m+1$.

\medskip

Suppose first that $k\ge0$. By time $T_{k,m}$, the walk has crossed $(k,k+1)$ to the right $m+1$ times and to the left $m$ times. Thus the urn at $k+1$ has been used through its $m$-th blue draw. Each red draw at that site contributes one upcrossing of $(k+1,k+2)$, giving $\cE^{k+1}(T_{k,m})=\fR^{(k+1)}_{\tau_m^\fB}$. Subtracting $\cE^k(T_{k,m})=m+1$ yields $\fD^{(k+1)}_{\tau_m^\fB}-1$.
If $k<0$, the walk starts and ends on the right side of the bond $(k,k+1)$, so its left and right crossing counts at $T_{k,m}$ are both $m+1$. The urn at $k+1$ has therefore been used through its $(m+1)$-st blue draw, giving $\cE^{k+1}(T_{k,m})=\fR^{(k+1)}_{\tau_{m+1}^\fB}$. The increment is $\fD^{(k+1)}_{\tau_{m+1}^\fB}$. By the same reasoning, one can also derive the backward increment. For all $k\in\z$ and $m\ge0$,
\begin{align}
   \cE^{k+1}(T_{k,m})-\cE^k(T_{k,m}) &\stackrel{\text{Law}}{=} 
   \begin{cases}
      \fD^{(k+1)}_{\tau_{m}^\fB}-1, & k \ge 0, \\[2mm]
      \fD^{(k+1)}_{\tau_{m+1}^\fB},  & k < 0,
   \end{cases}\label{eq:for_dis} \\
   \cE^{k-1}(T_{k,m})-\cE^k(T_{k,m}) &\stackrel{\text{Law}}{=} 
   \begin{cases}
      -\fD^{(k)}_{\tau_{m+1}^\fR}+1, & k > 0, \\[2mm]
      -\fD^{(k)}_{\tau_{m+1}^\fR},  & k \le 0.
   \end{cases}\label{eq:back_dis}
\end{align}

\noindent
For $x>k$, conditionally on $\{\cE^x(T_{k,m})=y\}$, the urn at $x+1$ has been used through its $y$-th blue draw when $x\ge0$, and through its $(y+1)$-st blue draw when $x<0$. Thus,
\begin{align}
	\cE^{x+1}(T_{k,m})-\cE^x(T_{k,m})
	\stackrel{\mathrm{Law}}{=}
	\begin{cases}
		\fD^{(x+1)}_{\tau_y^\fB}, & x\ge0, \\[2mm]
		\fD^{(x+1)}_{\tau_{y+1}^\fB}+1, & x<0.
	\end{cases}\label{eq:for_po}
\end{align}

\noindent
For $x<k$, conditionally on $\{\cE^x(T_{k,m})=y\}$, the urn at $x$ has been used through its $y$-th red draw. Thus,
\begin{align}
	\cE^{x-1}(T_{k,m})-\cE^x(T_{k,m})
	\stackrel{\mathrm{Law}}{=}
	\begin{cases}
		-\fD^{(x)}_{\tau_y^\fR}+1, & x>0, \\[2mm]
		-\fD^{(x)}_{\tau_y^\fR}, & x\le0.
	\end{cases}\label{eq:back_po}
\end{align}

In the remainder of this section, we first establish the joint convergence of the associated driving fields using the urn estimates in Appendix~\ref{appsub:PSR}, and then deduce the joint Ray--Knight theorem. Tightness of the collection of rescaled upcrossing profiles is treated separately in Appendix~\ref{app:tight}.

\subsection{Convergence of the driving fields}\label{sec:mart_meas}

For the urn at site $\ell$, we abbreviate $\fD^{(\ell)}_{\tau_{m}^{*^{(\ell)}}}$ as $\fD^{(\ell)}_{\tau_{m}^{*}}$ for $* \in \{\fB,\fR\}$. Define
\begin{align}\label{def_eta}
	\eta_{m,\ell}^*:=\fD^{(\ell)}_{\tau_m^*}-\fD^{(\ell)}_{\tau_{m-1}^*},	\qquad m\ge1,\quad \ell\in\z.
\end{align}
Let $\mathcal A$ denote the collection of all rectangles of the form $(0,t]\times(c,d]$, where $t\in\bD_+$ and $c,d\in\bD$ with $c<d$. For $A=(0,t]\times(c,d]\in\mathcal A$, write $A^{(2)}=(c,d]$ and $NA=\{Nz:z\in A\}$. Define the cumulative random fields
\begin{align}
	\Phi_N^\fB(A)&:=\frac{1}{\sqrt{2}N}\sum_{(m,\ell)\in N A}\eta_{m,\ell}^\fB+ \frac{1}{\sqrt{2}N} \sum_{\substack{\ell \in NA^{(2)},\, \ell \le 0}} 1,\label{def_Phi_B}\\
	\Phi_N^\fR(A)&:= - \frac{1}{\sqrt{2}N}\sum_{(m,\ell)\in N A}\eta_{m,\ell}^\fR+ \frac{1}{\sqrt{2}N} \sum_{\substack{\ell \in NA^{(2)},\, \ell> 0}} 1.\label{def_Phi_R}
\end{align}

We regard $(\Phi_N^\fB,\Phi_N^\fR)$ as a random element of $\r^{\mathcal A}\times\r^{\mathcal A}$, equipped with the product topology. The following theorem identifies its joint limit.

\begin{theorem}\label{thm:meas}
	As $N\to\infty$,
	$$
	(\Phi_N^\fB,\Phi_N^\fR) \ \mathop{\Longrightarrow}^{(d)\ } \ (\cM^\fB(A),\cM^\fR(A))_{A\in \mathcal{A}}.
	$$
	The limiting fields are given, for $t>0$, by
	\begin{align}
		\cM^\fB((0,t],\dd x)&=\frac{\dd x}{2\sqrt{2}(2\alpha+1)}+\frac{1}{t^\alpha}\int_0^t u^\alpha \cW(\dd u, \dd x),\label{eq:mart_meas_B}\\
		\cM^\fR((0,t],\dd x)&=\frac{\dd x}{2\sqrt{2}(2\alpha+1)}-\frac{1}{t^\alpha}\int_0^t u^\alpha \cW(\dd u, \dd x), \label{eq:mart_meas_R}
	\end{align}
	where $\cW$ is a common standard space-time white noise on $\h$.
\end{theorem}

\begin{proof}
	We first prove the convergence of $\Phi_N^\fB$. Lemma~\ref{lm:4_1} implies that the contribution of the urn at the origin tends to zero in $L^2$. Away from the origin,  $\{\eta_{\cdot,\ell}^\fB\}_{\ell>0}$ and $\{\eta_{\cdot,\ell}^\fB\}_{\ell<0}$ are sequences of i.i.d. random variables. Hence, it suffices to prove, for each $n\ge1$, the joint convergence in distribution of
	$$
	\left(\Phi_N^\fB(A),\, A = (0,k 2^{-n}]\times (l 2^{-n}, (l+1) 2^{-n}]\right)_{k\ge 1, l\in \{-2,-1,0\}}.
	$$ 

	\noindent
	Let $0<k_1<k_2<\ldots<k_p$ be multiples of $2^{-n}$, and define $A_{i,l}=(0,k_i]\times (l 2^{-n},(l+1)2^{-n}]$. For any $c_{i,l}\in\r$, let
	$$
	X_N:=\sum_{l=-2}^0 \sum_{i=1}^p  c_{i,l} \Phi_N^\fB(A_{i,l}).
	$$
	It is therefore enough to prove that $X_N$ converges in distribution. For any $A_{i,l}$, we have
	$$
	\Phi_N^\fB(A_{i,l}) = \frac{1}{\sqrt{2}N}\sum_{\ell \in NA_{i,l}^{(2)}} \fD_{\tau_{N_i}^\fB}^{(\ell)} + \frac{\lfloor N(l+1)2^{-n} \rfloor - \lfloor N l 2^{-n} \rfloor}{\sqrt{2}N}\mathbbm{1}_{\{l<0\}} ,
	$$
	where $N_i=\lfloor N k_i\rfloor$. By Propositions~\ref{prop:exp_discre} and~\ref{prop:cov},
	\begin{align*}
	    \lim_{N\to\infty}  \e[X_N]=& 2^{-n}\frac{1}{2\sqrt{2}(2\alpha+1)}\sum_{l=-2}^0 \sum_{i=1}^p  c_{i,l},\\
	    \lim_{N\to\infty} \operatorname{Var}(X_N)=
	    2^{-n}& \frac{1}{2\alpha+1}\sum_{l=-2}^0\sum_{i,j=1}^p c_{i,l} c_{j,l} \frac{(k_i\wedge k_j)^{2\alpha+1}}{ k_i^{\alpha}k_j^\alpha}.
	\end{align*}
	We recognize the expectation and variance of $\sum_{l=-2}^0 \sum_{i=1}^p  c_{i,l} \cM^\fB(A_{i,l})$.
	
	\medskip 
	
	We next show that $X_N$ converges to a Gaussian distribution. By independence, it suffices to prove this for $\sum_{i=1}^p c_{i,l}\Phi_N^\fB(A_{i,l})$ for each fixed $l$. We give the argument for $l=0$. Set $c_i:=c_{i,0}$ and $a_N:=\lfloor N2^{-n}\rfloor$. Then
	$$
	\sum_{i=1}^p c_i\left(\Phi_N^\fB(A_{i,0})-\e[\Phi_N^\fB(A_{i,0})]\right)=\sum_{\ell=1}^{a_N}Y_{N,\ell},
	$$
	where
	$$
	Y_{N,\ell}:=\frac{1}{\sqrt{2}N}\sum_{i=1}^p c_i\left(\fD_{\tau_{N_i}^\fB}^{(\ell)}-\e\left[\fD_{\tau_{N_i}^\fB}^{(\ell)}\right]\right).
	$$
	
	\noindent 
	For each $N$, the variables $\{Y_{N,\ell}\}_{\ell=1}^{a_N}$ are i.i.d. with mean zero, forming a triangular array. 
	By Lemma~\ref{lm:4_1}, there exist constants $C,c>0$ such that, for any $t=N_i>0$,
	$$
	\P\left(\bigl|\fD_{\tau_t^\fB}^{(1)}\bigr|\ge x\right)\le C e^{\frac{-c x^2}{x\vee t}}, \qquad \forall x\ge 1.
	$$
	Integrating the tail bound and splitting at $x=t$, we obtain
	\begin{align*}
		\e\left[\left|\fD_{\tau_t^\fB}^{(1)}\right|^4\right]
		&=4\int_0^\infty x^3\,\P\left(\left|\fD_{\tau_t^\fB}^{(1)}\right|\ge x\right)\dd x\\
		&\le 1+4C\left(\int_0^t x^3e^{-cx^2/t}\dd x+\int_t^\infty x^3e^{-cx}\dd x\right)\\
		&\le 1+\frac{2C}{c^2}t^2+\frac{24C}{c^4}\le C_1(1+t^2),
	\end{align*}
	for some constant $C_1>0$.
	Since $p$, $c_i$ and $k_i$ are fixed, there exists a constant $C_2>0$, independent of $N$, such that
	$$
	\e[Y_{N,1}^4]\le C_2N^{-2}.
	$$
	Thus, for every $\varepsilon>0$,
	$$
	\sum_{\ell=1}^{a_N}\e\left[Y_{N,\ell}^2\mathbbm{1}_{\{|Y_{N,\ell}|>\varepsilon\}}\right]\le\varepsilon^{-2}a_N\e[Y_{N,1}^4]\longrightarrow0.
	$$
	The Lindeberg central limit theorem therefore gives a Gaussian limit for the $l=0$ sum. The other cases are analogous, and independence yields the Gaussian limit of $X_N$.
	
	\medskip
	
	It remains to verify that the limiting distribution of $\Phi_N^\fR$ is given by~\eqref{eq:mart_meas_R} and is driven by the same white noise $\cW$ as in the limit of $\Phi_N^\fB$. To this end, it suffices to prove
	\begin{align}\label{eq:exp_Delta}
	X_N^\Delta:=\sum_{l=-2}^0 \sum_{i=1}^p  c_{i,l} \left(\Phi_N^\fB(A_{i,l})+\Phi_N^\fR(A_{i,l})\right) \, \xrightarrow[N\to\infty]{(d)}\,
	\frac{2^{-n}}{\sqrt{2}(2\alpha+1)}\sum_{l=-2}^0 \sum_{i=1}^p  c_{i,l},
	\end{align}
	where $A_{i,l}$ and $c_{i,l}$ are as defined before. Since
	$$
	\Phi_N^\fB(A_{i,l}) +\Phi_N^\fR(A_{i,l}) = \frac{1}{\sqrt{2}N}\sum_{\ell\in NA_{i,l}^{(2)}} \left(\fD_{\tau_{N_i}^\fB}^{(\ell)} - \fD_{\tau_{N_i}^\fR}^{(\ell)}\right) + \frac{\lfloor N(l+1)2^{-n} \rfloor - \lfloor N l 2^{-n} \rfloor}{\sqrt{2}N},
	$$
	by Proposition~\ref{prop:exp_discre}, the expectation of $X_N^\Delta$ converges to the limit appearing in~\eqref{eq:exp_Delta}. Moreover, the variance of $X_N^\Delta$ is
	$$
	\frac{1}{2N^2}\sum_{l=-2}^0  \sum_{i,j=1}^p c_{i,l} c_{j,l} \sum_{\ell \in NA_{i,l}^{(2)}} \operatorname{Cov}\left(\fD_{\tau_{N_i}^\fB}^{(\ell)} - \fD_{\tau_{N_i}^\fR}^{(\ell)}, \fD_{\tau_{N_j}^\fB}^{(\ell)} - \fD_{\tau_{N_j}^\fR}^{(\ell)}\right).
	$$
	By Lemma~\ref{lm:4_7}, there exists a constant $C > 0$ such that, for any $t=N_i$,
	\begin{align*}
	    \e\left[\left( \fD_{\tau_{t}^\fB}^{(\ell)} - \fD_{\tau_{t}^\fR}^{(\ell)} \right)^2 \right]\le 2\e\left[\left( \fD_{\tau_{t}^\fB}^{(\ell)} - \fD_{2t}^{(\ell)} \right)^2 \right]+2\e\left[\left( \fD_{\tau_{t}^\fR}^{(\ell)} - \fD_{2t}^{(\ell)} \right)^2 \right]
	    \le C\sqrt{t}.
	\end{align*}
	The Cauchy--Schwarz inequality implies that each covariance in the preceding sum is $O(\sqrt N)$. Thus, the variance of $X_N^\Delta$ converges to zero, which implies~\eqref{eq:exp_Delta}.
\end{proof}

We conclude this subsection by reconstructing the common driving white noise from the limiting fields.

\begin{lemma}\label{lem:M_b_r}
	Let $(\tilde{\cM}^\fB,\tilde{\cM}^\fR)$ be a pair of random fields indexed by $\mathcal A$ with the same joint law as $(\cM^\fB,\cM^\fR)$ in Theorem~\ref{thm:meas}. Then there exists a common standard space-time white noise $\tilde{\cW}$ on $\h$ such that the fields admit extensions satisfying~\eqref{eq:mart_meas_B} and~\eqref{eq:mart_meas_R}, respectively, for every $t>0$, with $\cW$ replaced by $\tilde{\cW}$.
\end{lemma}

\begin{proof}
	Let $\mathcal R$ denote the semi-ring of rectangles $(t_1,t_2]\times(x,y]$ with dyadic endpoints, $0\le t_1<t_2$ and $x<y$. For $R=(t_1,t_2]\times(x,y]\in\mathcal R$, define
	\begin{equation*}
		\tilde{M}_0(R) := t_2^\alpha \tilde{\cM}^\fB((0,t_2]\times(x,y]) - t_1^\alpha \tilde{\cM}^\fB((0,t_1]\times(x,y]) - \frac{(t_2^\alpha-t_1^\alpha)(y-x)}{2\sqrt{2}(2\alpha+1)},
	\end{equation*}
	where the second term is taken to be zero when $t_1=0$.
	Since $\tilde{\cM}^\fB\stackrel{\mathrm{(d)}}{=}\cM^\fB$ on $\mathcal A$, the finite-dimensional distributions of $\{\tilde M_0(R):R\in\mathcal R\}$ coincide with those of
	$$
	\left\{
	\int_x^y\!\int_{t_1}^{t_2}s^\alpha\cW(\dd s,\dd u)
	:\, (t_1,t_2]\times(x,y]\in\mathcal R
	\right\}.
	$$
	
	Define a set function $m$ on $\mathcal R$ by $m((t_1,t_2]\times(x,y]):=\frac{1}{2\alpha+1}(y-x)(t_2^{2\alpha+1}-t_1^{2\alpha+1})$. Then $\tilde M_0$ is an elementary orthogonal stochastic measure with structural measure $m$; see \cite[p.~231]{GS04}. The measure $m$ extends to $\mathcal B(\h)$ as $m(\dd s,\dd u)=s^{2\alpha}\dd s\,\dd u$, with zero mass on $\{0\}\times\r$. By \cite[Chapter IV, Section 4, Theorem 2]{GS04}, $\tilde M_0$ admits an orthogonal extension $\tilde M$ such that $\e[\tilde M(R_1)\tilde M(R_2)]=m(R_1\cap R_2)$ for all $R_1,R_2\in\mathcal B(\h)$ with $m(R_i)<\infty$, $i=1,2$. This extension remains jointly Gaussian, since it is obtained by $L^2$ limits. For $t>0$, we extend $\tilde{\cM}^\fB$ by
	$$
	\tilde{\cM}^\fB((0,t]\times(x,y]):=t^{-\alpha}\tilde M((0,t]\times(x,y])+\frac{y-x}{2\sqrt{2}(2\alpha+1)}.
	$$

	Using the stochastic integral construction of~\cite[Chapter~2]{Walsh1986AnIT}, define
	\begin{equation*}
		\tilde{\cW}_{x,y}(A):=\int_x^y\!\int_A s^{-\alpha}\tilde M(\dd s,\dd u), \qquad A\in\mathcal B(\r_+),\quad \operatorname{Leb}(A)<\infty,
	\end{equation*}
	with the integrand set to zero at $s=0$. A direct computation of its covariance, together with joint Gaussianity, shows that $\tilde\cW$ is a standard white noise on $\h$. Consequently, $\tilde{\cM}^\fB((0,t],\dd x)$ admits the representation~\eqref{eq:mart_meas_B} driven by $\tilde{\cW}$. By the assumed joint law,
	$$
	\tilde{\cM}^\fR((0,t]\times(x,y])
	=\frac{y-x}{\sqrt{2}(2\alpha+1)}-\tilde{\cM}^\fB((0,t]\times(x,y])
	$$
	for all $(0,t]\times(x,y]\in\mathcal A$, almost surely. This identity gives~\eqref{eq:mart_meas_R} with the same $\tilde{\cW}$.
\end{proof}

\subsection{Weak Convergence to the SDE}

We now prove Theorem~\ref{thm:main} using the fields $\cM^\fB$ and $\cM^\fR$ defined in~\eqref{eq:mart_meas_B} and~\eqref{eq:mart_meas_R}, driven by the same white noise $\cW$. For each $(a,r)\in\h$, let $\cS_{r,\cdot}(a)$ denote the two-sided process governed, up to its respective killing times, by
\begin{align}
	\cS_{r,x}(a) =&\; a+\sqrt{2}\int_r^x \cM^\fB([0,\cS_{r,u}(a)],\dd u), &&\ r\le x\le\omega_{a,r}^+, \label{eq:poly_GRKT}\\
	\cS_{r,-x}(a)=&\; a+\sqrt{2}\int_{-r}^x \cM^{\fR,*}([0,\cS_{r,-u}(a)],\dd u), &&\ -r\le x\le-\omega_{a,r}^-, \label{eq:poly_GRKT_back}
\end{align}
where $\cM^{\fR,*}$ is defined as in~\eqref{eq:mart_meas_R}, with $\cW$ replaced by $\cW^*$, and
\begin{align*}
	\omega_{a,r}^{-}:=\sup\{x\le r\wedge0:\,\cS_{r,x}(a)=0\},\qquad
	\omega_{a,r}^{+}:=\inf\{x\ge r\vee0:\,\cS_{r,x}(a)=0\}.
\end{align*}
The branches are reflected at zero until their respective killing times and set to zero thereafter. In these equations, the field integrals are understood via~\eqref{eq:mart_meas_B} and~\eqref{eq:mart_meas_R}, with the white noise integrands set to zero at height zero.

Throughout this section, $\cS$ denotes the process defined above. The family $\{2\cS_{r,\cdot}(a/2):(a,r)\in\h\}$ has the same joint law as the CRKCs in Definition~\ref{def:RK_curve}. Thus, to prove Theorem~\ref{thm:main}, it suffices to establish, for any $p\in\N$ and $\{(a_i,r_i)\}_{i=1}^p\subset\h$,
\begin{align}\label{eq:thm:main}
	\left\{S^N_{r_i,\cdot}(a_i)\right\}_{i=1}^p
	\ \Longrightarrow\
	\left\{\cS_{r_i,\cdot}(a_i)\right\}_{i=1}^p \qquad \text{as } N\to\infty,
\end{align}
weakly in $D(\r)^p$, where $D(\r)$ is equipped with the Skorokhod $J_1$ topology. The case $p=1$ follows from Tóth's generalized Ray--Knight theorem~\cite{Toth96}, which we recall below.

\begin{theorem}[\cite{Toth96}, Theorem~1B]\label{thm:1B_toth}
Let $\alpha>0$, and let $w$ be a non-increasing weight function satisfying~\eqref{eq:weight_func}. Fix $(a,r)\in\h$, and let $S^N_{r,\cdot}(a)$ denote the rescaled upcrossing profile of the PSRW defined in \eqref{eq:def_S_N}. Then
$$
\left(\omega_{a,r}^{N,-},\, \omega_{a,r}^{N,+},\, S^N_{r,\cdot}(a)\right)
\ \Longrightarrow\
\left(\omega_{a,r}^{-},\, \omega_{a,r}^{+},\, \cS_{r,\cdot}(a)\right),
\qquad \text{as } N\to\infty,
$$
weakly in $\r_-\times\r_+\times D(\r)$.
\end{theorem}

\begin{remark}
	Let $L(x,n)$ denote the number of crossings of $\{x,x+1\}$ by time $n$. Then
	$$
	L(x,T_{k,m})
	=\begin{cases}
		2\cE^x(T_{k,m})-\mathbbm{1}_{[0,k]}(x), & k\ge0,\\
		2\cE^x(T_{k,m})+\mathbbm{1}_{[k+1,0)}(x), & k<0.
	\end{cases}
	$$
	For $m\ge0$, set $T^L_{k,m}:=\inf\{i>0:\, L(k,i)>m\}$ and define
	$$
	Y_{r,x}^N(a):=N^{-1}L\bigl(\lfloor Nx\rfloor,T^L_{\lfloor Nr\rfloor,\lfloor Na\rfloor}\bigr).
	$$
	The profiles $Y_{r,\cdot}^N(a)$ and $2S_{r,\cdot}^N(a/2)$ are asymptotically equivalent in probability in the Skorokhod $J_1$ topology. Thus Theorem~\ref{thm:main} yields the joint convergence of any finite collection of these edge-crossing profiles to the corresponding CRKCs of Definition~\ref{def:RK_curve}.
\end{remark}

We now derive a representation of the rescaled upcrossing profiles $S^N$ defined in~\eqref{eq:def_S_N} in terms of $\eta$ and $\Phi$. To simplify the notation, for $y\in\r$, write $y_N:=\lfloor Ny\rfloor$, and set $\delta_k:=\mathbbm{1}_{\{k<0\}}$ and $\gamma_k:=\mathbbm{1}_{\{k>0\}}$. Define
\begin{align*}
    \omega_{a,r}^{N,-}:=\inf\{x\le r\wedge0:\, S_{r,x}^N(a)>0\},\qquad
    \omega_{a,r}^{N,+}:=\sup\{x\ge r\vee0:\, S_{r,x}^N(a)>0\}.
\end{align*}

For each $(a,r)\in\h$ and $N\in\N$, define
\begin{align*}
    R_k^N&:=\left(0,S^N_{r,\frac{k}{N}}(a)+\frac{\delta_k}{N}\right]\times\left(\frac{k}{N},\frac{k+1}{N}\right],
    && k\ge r_N,\\
    L_k^N&:=\left(0,S^N_{r,\frac{k+1}{N}}(a)\right] \times\left(\frac{k}{N},\frac{k+1}{N}\right],
    && k<r_N.
\end{align*}
These rectangles need not have dyadic endpoints. We extend $\Phi_N^*$ to rectangles of the form $(0,t]\times(c,d]$, with arbitrary real endpoints $t>0$ and $c<d$, using the same formulas as in~\eqref{def_Phi_B} and~\eqref{def_Phi_R}, and set $\Phi_N^*(\varnothing)=0$. By definition, $S^N_{r,x}(a)=N^{-1}(a_N+1)$ whenever $x_N=r_N$. Under the common urn coupling,
\begin{align}
    S^N_{r,x}(a)&=N^{-1}a_N-N^{-1} \eta^\fB_{a_N+1+\delta_{r_N},\,r_N+1}+\sqrt2\sum_{k=r_N}^{x_N-1}\Phi_N^\fB(R_k^N),&&x_N>r_N,\label{eq:proce_+1}\\
    S^N_{r,x}(a)&=N^{-1}(a_N+1)+\sqrt2\sum_{k=x_N}^{r_N-1}\Phi_N^\fR(L_k^N),&&x_N<r_N.\label{eq:proce_-1}
\end{align}

To verify these identities, fix $(a,r)\in\h$ and $N\in\N$, and let $\iota^N:=T^\cE_{r_N,a_N}$. If $x_N\ge r_N+1$, then, by~\eqref{eq:for_dis} and~\eqref{eq:for_po},
\begin{align*}
    S^N_{r,x}(a) = N^{-1}(a_N+1) + N^{-1}\Big(\fD_{\tau_{a_N+\delta_{r_N}}^\fB}^{(r_N+1)} - \mathbbm{1}_{\{r_N\ge0\}}\Big) + N^{-1} \sum_{k=r_N+1}^{x_N-1} \Big(\fD^{(k+1)}_{\tau_{\cE^{k}(\iota^N)+\delta_k}^\fB} + \delta_k\Big).
\end{align*}
Since $1-\mathbbm{1}_{\{r_N\ge0\}}=\delta_{r_N}$ and $\cE^{r_N}(\iota^N)=a_N+1$, the contribution corresponding to $k=r_N$ can be incorporated into the summation by using $\eta$ defined in~\eqref{def_eta}. This yields
\begin{align*}
    S^N_{r,x}(a) = N^{-1}a_N - N^{-1}\eta^\fB_{a_N+1+\delta_{r_N},\, r_N+1} + N^{-1} \sum_{k=r_N}^{x_N-1} \Bigg[\sum_{m=1}^{\cE^{k}(\iota^N)+\delta_k} \eta_{m,k+1}^\fB + \delta_k\Bigg].
\end{align*}
By~\eqref{def_Phi_B}, each summand equals $\sqrt{2}N\Phi_N^\fB(R_k^N)$, which proves~\eqref{eq:proce_+1}.

For $x_N<r_N$,~\eqref{eq:back_dis} and~\eqref{eq:back_po} similarly give
\begin{align*}
    S^N_{r,x}(a)&=N^{-1}(a_N+1)+N^{-1}\sum_{j=x_N+1}^{r_N}\left[\cE^{j-1}(\iota^N)-\cE^j(\iota^N)\right]\\
    &=N^{-1}(a_N+1)+N^{-1}\sum_{j=x_N+1}^{r_N}\left[-\sum_{m=1}^{\cE^j(\iota^N)}\eta^\fR_{m,j}+\gamma_j\right].
\end{align*}
Reindexing with $j=k+1$ and using~\eqref{def_Phi_R} yields~\eqref{eq:proce_-1}. The identities also cover empty rectangles: $R_k^N=\varnothing$ implies $\delta_k=0$, and $L_k^N=\varnothing$ implies $\gamma_{k+1}=0$.

\bigskip

Since each upcrossing count is non-decreasing in time, $S^N_{r_1,y}(a_1)<S^N_{r_2,y}(a_2)$ implies $T_{\lfloor Nr_1\rfloor,\lfloor Na_1\rfloor}<T_{\lfloor Nr_2\rfloor,\lfloor Na_2\rfloor}$. Hence, for every $z\in\r$,
$$
\cE^{\lfloor Nz\rfloor}(T_{\lfloor Nr_1\rfloor,\lfloor Na_1\rfloor})\le\cE^{\lfloor Nz\rfloor}(T_{\lfloor Nr_2\rfloor,\lfloor Na_2\rfloor}).
$$
This gives the non-crossing property, which will be used repeatedly in the sequel.

\begin{lemma}\label{lem:noncrossing}
	For any $(a_i,r_i)\in \h$, $i=1,2$, and any $y,z\in \r$,
	\begin{align*}
		S^N_{r_1,y}(a_1)< S^N_{r_2,y}(a_2) \ \Longrightarrow\ S^N_{r_1,z}(a_1)\le S^N_{r_2,z}(a_2).
	\end{align*}
\end{lemma}

To prove~\eqref{eq:thm:main}, we first establish the joint convergence of a single upcrossing profile $S^N$ and the driving fields $\Phi_N^\fB,\Phi_N^\fR$. The representations~\eqref{eq:proce_+1} and~\eqref{eq:proce_-1} are used to identify the limiting profile through SDEs driven by the limiting fields, while strong uniqueness determines the joint law of finitely many profiles driven by the same white noise.
We write $S_{r,\cdot}^{N,\pm}(a)$ and $\cS_{r,\cdot}^{\pm}(a)$ for the forward and backward branches of $S_{r,\cdot}^N(a)$ and $\cS_{r,\cdot}(a)$, respectively. For $*\in\{\fB,\fR\}$ and $I\subset\r$, we denote by $\Phi_N^*|_I$ the restriction of the field $\Phi_N^*$ to the domain $\r_+\times I$. Formally,
$$
\Phi_N^*|_I := \{\Phi_N^*(A) : A \in \mathcal{A} \text{ and } A^{(2)} \subset I\}.
$$
By the constructions~\eqref{def_Phi_B} and~\eqref{def_Phi_R}, for any intervals $I_1,\dots,I_k\subset\r$ whose interiors are mutually disjoint (e.g., $(-\infty, 0]$ and $[0, \infty)$), the pairs $(\Phi_N^\fB|_{I_j},\Phi_N^\fR|_{I_j})$, $1\le j\le k$, depend on disjoint sets of P\'olya urns and are therefore independent. Similarly, we denote by $\cM^*|_{\mathcal A}$ and $\cM^*|_I$ the restrictions of $\cM^*$ to $\mathcal A$ and $\{A\in\mathcal A:A^{(2)}\subset I\}$, respectively.

\smallskip

The next proposition, whose proof is postponed to the end of this section, gives the joint convergence of the driving field and finitely many one-sided profiles, for forward branches started in $\r_+$ and backward branches started in $\r_-$. In these cases, the killing times simplify to
\begin{align*}
    \omega_{a,r}^{+}=\inf\{x\ge r:\,\cS_{r,x}(a)=0\},
    \qquad -\omega_{a,r}^{-}=\inf\{x\ge -r:\,\cS_{r,-x}(a)=0\}.
\end{align*}

\begin{proposition}\label{prop:joint_for_RK}
For any $p\in \N$ and $(a_i,r_i)\in \r_+^2$, $1\le i\le p$, as $N\to \infty$,
\begin{align*}
\left(\Phi_N^\fB|_{[0,\infty)}, \left\{S_{r_i,\cdot}^{N,+}(a_i)\right\}_{i=1}^p\right)
\ \mathop{\Longrightarrow}^{(d)\ } \
\left(\cM^\fB|_{[0,\infty)}, \left\{\cS_{r_i,\cdot}^+(a_i)\right\}_{i=1}^p\right),
\end{align*}
where each $\cS_{r_i,\cdot}^+(a_i)$ is the unique strong solution to~\eqref{eq:poly_GRKT}, killed at $\omega_{a_i,r_i}^{+}$.

Likewise, for any $p\in \N$ and $(a_i,r_i)\in \r_+\times \r_- $, $1\le i\le p$, as $N\to \infty$,
\begin{align*}
\left(\Phi_N^\fR|_{(-\infty,0]}, \left\{S_{r_i,\cdot}^{N,-}(a_i)\right\}_{i=1}^p\right)
\ \mathop{\Longrightarrow}^{(d)\ } \
\left(\cM^\fR|_{(-\infty,0]}, \left\{\cS_{r_i,\cdot}^-(a_i)\right\}_{i=1}^p\right),
\end{align*}
where each $\cS^-_{r_i,-\cdot}(a_i)$ is the unique strong solution to~\eqref{eq:poly_GRKT_back}, killed at $-\omega_{a_i,r_i}^{-}$.
\end{proposition}

The previous proposition covers only absorption, whereas $\cS_{r,\cdot}(a)$ is reflected at zero between $0$ and $r$. We use the duality argument of~\cite{Yu26}, recalled in Section~\ref{sec:basic}, to handle this case. The following result is a direct consequence of Proposition~\ref{prop:self_dual}.

\begin{proposition}\label{cor:self-dual}
	Suppose that $\{(\cS_{r,x}(a),\,x\ge r)\}_{(a,r)\in\bD_+^2}$ is a family of processes such that each $(\cS_{r,x}(a),\,x\ge r)$ is the unique strong solution to~\eqref{eq:poly_GRKT}, driven by the common field $\cM^\fB$ and killed at $\omega_{a,r}^{+}$. Fix $(b,y)\in\r_+^2$, and let $(\cS_{y,x}(b),\,0\le x\le y)$ be a continuous process with $\cS_{y,y}(b)=b$ such that, almost surely,
	$$
	\cS_{y,x}(b)=\inf\{a\in\bD_+:\,\cS_{x,y}(a)>b\}, \qquad \forall x\in\bD_+\cap[0,y].
	$$
	Then $(\cS_{y,-x}(b),\,-y\le x\le0)$ is the unique strong solution to~\eqref{eq:poly_GRKT_back}, reflected at zero and driven by the corresponding field $\cM^{\fR,*}$.
\end{proposition}

\smallskip
	
\begin{proof}[Proof of Theorem~\ref{thm:main}]
    Fix $p$ and $(a_i,r_i)\in \h$, $1\le i\le p$. We claim that, in order to prove~\eqref{eq:thm:main}, it suffices to establish, for every fixed $(a,r)\in\h$,
    \begin{equation}\label{eq:joint_single_suff}
        \bigl(\Phi_N^\fB, \Phi_N^\fR, S^N_{r,\cdot}(a)\bigr) \ \mathop{\Longrightarrow}^{(d)\ }\ \bigl(\cM^\fB|_{\mathcal A}, \cM^\fR|_{\mathcal A}, \cS_{r,\cdot}(a)\bigr).
    \end{equation}
    By Theorems~\ref{thm:meas} and~\ref{thm:1B_toth}, the sequence $\bigl(\Phi_N^\fB,\Phi_N^\fR,S^N_{r_1,\cdot}(a_1),\ldots,S^N_{r_p,\cdot}(a_p)\bigr)_{N\ge1}$ is tight in the Polish space $\r^{\mathcal A}\times\r^{\mathcal A}\times D(\r)^p$. Let $(\tilde{\cM}^\fB,\tilde{\cM}^\fR,Y_1,\ldots,Y_p)$ denote a subsequential limit. By Theorem~\ref{thm:meas}, the pair $(\tilde{\cM}^\fB,\tilde{\cM}^\fR)$ has the same joint law on $\mathcal A$ as $(\cM^\fB,\cM^\fR)$. Lemma~\ref{lem:M_b_r} therefore gives a common standard space-time white noise $\tilde{\cW}$ and extensions of the fields such that, for every $t>0$,
   \begin{align}\label{eq:til_cM}
		\tilde{\cM}^\fB((0,t],\dd x)=\frac{\dd x}{2\sqrt{2}(2\alpha+1)}+\frac{1}{t^\alpha}\int_0^t u^\alpha \tilde{\cW}(\dd u, \dd x), 
	\end{align}
	and
	$$
	\tilde{\cM}^\fR((0,t],\dd x)=\frac{\dd x}{2\sqrt{2}(2\alpha+1)}-\frac{1}{t^\alpha}\int_0^t u^\alpha \tilde{\cW}(\dd u, \dd x).
	$$
    We use the same notation for these extensions. Define $\tilde{\cM}^{\fR,*}$ by
    \begin{align}\label{eq:til_cM_R}
        \tilde{\cM}^{\fR,*}((0,t],\dd x) =\frac{\dd x}{2\sqrt{2}(2\alpha+1)}-\frac{1}{t^\alpha}\int_0^t u^\alpha \tilde{\cW}^*(\dd u, \dd x),
    \end{align}
    where $\tilde{\cW}^*$ is the image of $\tilde{\cW}$ under the map $(a,r)\mapsto (a,-r)$.
    For each $1\le i\le p$,~\eqref{eq:joint_single_suff} identifies the joint law of $(\tilde{\cM}^\fB,\tilde{\cM}^\fR,Y_i)$. Thus the forward and backward branches of $Y_i$ solve~\eqref{eq:poly_GRKT} and~\eqref{eq:poly_GRKT_back}, respectively, driven by $\tilde{\cM}^\fB$ and $\tilde{\cM}^{\fR,*}$, with the prescribed reflecting and absorbing boundary rules. By strong uniqueness, each $Y_i$ is a measurable functional of the white noise $\tilde{\cW}$. Hence the joint law of $\{Y_i\}_{i=1}^p$ coincides with that of $\{\cS_{r_i,\cdot}(a_i)\}_{i=1}^p$. This identifies every subsequential limit and proves~\eqref{eq:thm:main}.

    \bigskip

	Fix $(a_0,r_0)\in\r_+^2$; the case $(a_0,r_0)\in\r_+\times\r_-$ follows by symmetry. If $r_0=0$, Proposition~\ref{prop:joint_for_RK}, together with the independence of the urns on the two half-lines and the joint field convergence in Theorem~\ref{thm:meas}, yields~\eqref{eq:joint_single_suff}. We therefore assume $r_0>0$ and set $I_1:=[r_0,\infty)$, $I_2:=[0,r_0]$ and $I_3:=(-\infty,0]$.
	For each $i=1,2,3$, write $S^N_{r_0,\cdot}(a_0)|_{I_i}$ for the restriction of the path to $I_i$. We first establish
    \begin{equation}\label{eq:joint_single}
        \bigl( \Phi_N^\fB|_{I_i}, S^N_{r_0,\cdot}(a_0)|_{I_i} \bigr) \ \mathop{\Longrightarrow}^{(d)\ }\ \bigl(\cM^\fB|_{I_i}, \cS_{r_0,\cdot}(a_0)|_{I_i}\bigr).
    \end{equation}
    It follows from Proposition~\ref{prop:joint_for_RK} that~\eqref{eq:joint_single} holds for $i=1$.

    \medskip
    
    Consider an arbitrary convergent subsequence of the tight sequence
    $$
    \left\{\bigl(\Phi_N^\fB,S^N_{r_0,\cdot}(a_0)|_{I_2},
    \{(S^{N,+}_{r,\cdot}(a),\omega_{a,r}^{N,+})\}_{a,r\in\bD_+}\bigr)\right\}_{N\ge1},
    $$
    and denote its limit by
    $\bigl(\tilde{\cM}^\fB|_{\mathcal A},Y|_{I_2},
    \{(\cY^+_{r,\cdot}(a),\omega_{a,r}^+)\}_{a,r\in\bD_+}\bigr)$. By Skorokhod's representation theorem (see, e.g., \cite[Theorem~6.7]{Bil99}), we may assume that this subsequence converges almost surely on a common probability space. Extend $\tilde{\cM}^\fB$ by the construction in Lemma~\ref{lem:M_b_r}. By Proposition~\ref{prop:joint_for_RK}, each $\cY^+_{r,\cdot}(a)$ is the unique strong solution to~\eqref{eq:poly_GRKT}, driven by $\tilde{\cM}^\fB$ and killed at $\omega_{a,r}^+$. By Theorem~\ref{thm:1B_toth}, $Y$ is continuous and $Y(r_0)=a_0$.
    
    We next show that, almost surely, for every $x\in [0,r_0]\cap \bD_+$,
    \begin{align}\label{eq:claim_dual}
    	Y(x)=\inf\{b\in \bD_+:\,\cY^+_{x,r_0}(b)>a_0\}.
    \end{align}
    The identity is immediate when $x=r_0$, so it remains to consider $x\in(0,r_0)\cap\bD_+$.
    Suppose $b\in \bD_+$ is such that $\cY^+_{x,r_0}(b)>a_0$. Since $\bigl(S^N_{r_0,x}(a_0), S^N_{x,r_0}(b)\bigr)$ converges almost surely to $\bigl(Y(x), \cY^+_{x,r_0}(b)\bigr)$, there exists $N_0$ such that for all $N\ge N_0$, $S^N_{x,r_0}(b)>\frac{\lfloor N a_0 \rfloor+1}{N}=S^N_{r_0,r_0}(a_0)$. Lemma~\ref{lem:noncrossing} implies that, for all $N\ge N_0$, $\frac{\lfloor N b \rfloor+1}{N}=S^N_{x,x}(b)\ge S^N_{r_0,x}(a_0)$. Letting $N\to\infty$, we obtain $b\ge Y(x)$. Taking the infimum over such $b$ yields that $Y(x)$ is no larger than the right-hand side of~\eqref{eq:claim_dual}. On the other hand, suppose $b\in \bD_+$ is such that $\cY^+_{x,r_0}(b)\le a_0$. 
    \begin{itemize}[itemsep=0pt, left=4pt]
        \item If $a_0>0$, note that $\cY^+_{x,\cdot}(b)$ has the same law as a rescaled BESQ$^1$ process up to the killing time $\omega_{b,x}^{+}$, and hence $\mathbb{P}(\exists b\in \bD_+:\,\cY^+_{x,r_0}(b)=a_0)=0$. We can therefore restrict our attention to those $b$ satisfying $\cY^+_{x,r_0}(b)<a_0$. Again, by almost sure convergence, there exists $N_0$ such that for all $N\ge N_0$, $S^N_{x,r_0}(b)<\frac{\lfloor N a_0 \rfloor+1}{N}=S^N_{r_0,r_0}(a_0)$.

        \item If $a_0=0$, then $\cY^+_{x,r_0}(b)=0$, so $\omega_{b,x}^+\le r_0$. Since the killing time has no atom at $r_0$, almost surely $\omega_{b,x}^+<r_0$ whenever this occurs, simultaneously for all such dyadic $x,b$. By the joint almost sure convergence of the killing times, $\omega_{b,x}^{N,+}<r_0$ for all sufficiently large $N$. Hence
        $S^N_{x,r_0}(b)=0<\frac1N=S^N_{r_0,r_0}(0)$.
    \end{itemize}
    In both cases, Lemma~\ref{lem:noncrossing} gives $\frac{\lfloor N b \rfloor+1}{N}=S^N_{x,x}(b)\le S^N_{r_0,x}(a_0)$ for all $N\ge N_0$. Passing to the limit as $N\to\infty$ yields $b\le Y(x)$. Since $\bD_+\ni b\mapsto \cY^+_{x,r_0}(b)$ is monotone, $Y(x)\ge \inf\{b\in \bD_+:\, \cY^+_{x,r_0}(b)>a_0\}$. This finishes the proof of~\eqref{eq:claim_dual}.
    
    Since all the assumptions of Proposition~\ref{cor:self-dual} have been verified, we may apply the proposition to conclude that $(Y(-x),\, -r_0\le x\le 0)$ is the unique strong solution to~\eqref{eq:poly_GRKT_back} driven by $\tilde{\cM}^{\fR,*}$ with reflection at zero, where $\tilde{\cM}^{\fR,*}$ is given by~\eqref{eq:til_cM_R}. Therefore,~\eqref{eq:joint_single} holds for $i=2$. In particular, $S^N_{r_0,0}(a_0)$ converges in distribution to $\cS_{r_0,0}(a_0)$.
    
    \medskip
    
    We next prove~\eqref{eq:joint_single} for $I_3$. By tightness, consider an arbitrary weakly convergent subsequence of
    $$
    \left\{\bigl(\Phi_N^\fR|_{I_3},S^N_{r_0,\cdot}(a_0)|_{I_3},
    \{S^{N,-}_{0,\cdot}(a)\}_{a\in\bD_+}\bigr)\right\}_{N\ge1}.
    $$
    By Skorokhod's representation theorem, we may assume almost sure convergence on a common probability space to $\bigl(\tilde{\cM}^\fR|_{I_3},Y|_{I_3},\{\cY^-_{0,\cdot}(a)\}_{a\in\bD_+}\bigr)$. By Proposition~\ref{prop:joint_for_RK}, each $\cY^-_{0,-\cdot}(a)$ is the unique strong solution to~\eqref{eq:poly_GRKT_back}, driven by $\tilde{\cM}^{\fR,*}$ and killed at $-\omega_{a,0}^-$. Also, $Y$ is continuous. For any $a\in\bD_+$ with $a>Y(0)$, almost sure convergence gives $S^{N,-}_{0,0}(a)>S^N_{r_0,0}(a_0)$ for all sufficiently large $N$. Lemma~\ref{lem:noncrossing} and passage to the limit yield $\cY^-_{0,x}(a)\ge Y(x)$ for all $x\le0$. Similarly, $a<Y(0)$ implies $\cY^-_{0,x}(a)\le Y(x)$ for all $x\le0$. Since $Y(0)$ has the atomless law of $\cS_{r_0,0}(a_0)$, it is positive almost surely. Monotonicity in $a$ therefore gives
    \begin{align}\label{eq:sqee}
    	\lim_{a\uparrow Y(0)}\cY^-_{0,x}(a)\le Y(x)\le\lim_{a\downarrow Y(0)}\cY^-_{0,x}(a),
    \end{align}
    where both limits are taken along $\bD_+$. For fixed $x<0$,
    $$
    \mathcal K_x:=\{b>0:\,\lim_{\bD_+\ni a\uparrow b}\cY^-_{0,x}(a)\ne\lim_{\bD_+\ni a\downarrow b}\cY^-_{0,x}(a)\}
    $$
    is $\tilde{\cM}^\fR|_{I_3}$-measurable and countable by monotonicity. Since $S^N_{r_0,0}(a_0)$ and $\Phi_N^\fR|_{I_3}$ depend on disjoint sets of urns, their limits $Y(0)$ and $\tilde{\cM}^\fR|_{I_3}$ are independent. Since $Y(0)$ has an atomless law, $\p(Y(0)\in\mathcal K_x)=0$. Thus~\eqref{eq:sqee} gives, almost surely,
    $$
    Y(x)=\lim_{\bD_+\ni a\downarrow Y(0)}\cY^-_{0,x}(a)=:Y'(x).
    $$
	By independence, we apply Corollary~\ref{cor:cov_b} conditionally on $Y(0)$, with the corresponding normalization and spatial reversal. Hence $(Y'(-x),\,x\ge0)$ is the unique strong solution to~\eqref{eq:poly_GRKT_back}, driven by $\tilde{\cM}^{\fR,*}$, started from $Y(0)$ and killed at $-\omega_{Y(0),0}^-$. Taking a countable dense set of $x<0$ and using continuity shows that $Y=Y'$ on $I_3$ almost surely. This proves~\eqref{eq:joint_single} with $\fB$ replaced by $\fR$ for $i=3$.
    
    Define the deterministic field $\mathcal H$ on $\mathcal A$ by $\mathcal H((0,t]\times(x,y]):=(y-x)/(\sqrt{2}(2\alpha+1))$. By~\eqref{eq:exp_Delta}, $\Phi_N^\fB+\Phi_N^\fR\to\mathcal H$ in probability in $\r^{\mathcal A}$, while~\eqref{eq:mart_meas_B}--\eqref{eq:mart_meas_R} give $\cM^\fB+\cM^\fR=\mathcal H$ on $\mathcal A$. Slutsky's theorem therefore upgrades the joint convergence with $\Phi_N^\fB$ on $I_1,I_2$ and with $\Phi_N^\fR$ on $I_3$ to
    \begin{align}\label{eq:joint_single_ext}
    	\bigl(\Phi_N^\fB|_{I_i},\Phi_N^\fR|_{I_i},S^N_{r_0,\cdot}(a_0)|_{I_i}\bigr) \ \mathop{\Longrightarrow}^{(d)\ }\ \bigl(\cM^\fB|_{I_i},\cM^\fR|_{I_i},\cS_{r_0,\cdot}(a_0)|_{I_i}\bigr), \qquad i=1,2,3.
    \end{align}
    
    \medskip
	
	It remains to establish~\eqref{eq:joint_single_suff}. By tightness, consider an arbitrary subsequential limit $(\tilde{\cM}^\fB|_{\mathcal A},\tilde{\cM}^\fR|_{\mathcal A},Y)$ of $(\Phi_N^\fB,\Phi_N^\fR,S^N_{r_0,\cdot}(a_0))$. By Theorem~\ref{thm:meas} and Lemma~\ref{lem:M_b_r}, the limiting fields admit extensions driven by a common white noise $\tilde{\cW}$. Moreover, $Y$ is continuous by Theorem~\ref{thm:1B_toth}, so~\eqref{eq:joint_single_ext} identifies the joint law of the restricted fields and $Y|_{I_i}$ for each $i=1,2,3$. Consequently, $Y|_{I_1}$ solves~\eqref{eq:poly_GRKT}, driven by $\tilde{\cM}^\fB$ and killed at its first hitting time of zero. The reversed branch $(Y(-x),-r_0\le x\le0)$ solves~\eqref{eq:poly_GRKT_back}, driven by $\tilde{\cM}^{\fR,*}$ and reflected at zero, while $(Y(-x),x\ge0)$ solves the same SDE from $Y(0)$ and is killed at its first hitting time of zero. Concatenating the latter two solutions at $x=0$ gives the backward branch with the prescribed boundary rule. Strong uniqueness therefore identifies the joint law of $(\tilde{\cM}^\fB|_{\mathcal A},\tilde{\cM}^\fR|_{\mathcal A},Y)$ with that of $(\cM^\fB|_{\mathcal A},\cM^\fR|_{\mathcal A},\cS_{r_0,\cdot}(a_0))$. This proves~\eqref{eq:joint_single_suff} and completes the proof.
\end{proof}

\smallskip

\begin{proof}[Proof of Proposition~\ref{prop:joint_for_RK}]
	Arguing as in the proof of Theorem~\ref{thm:main}, it suffices to establish, for each fixed $(a,r)\in\h$,
	\begin{equation}\label{eq:joint_single_I+}
		\bigl(\Phi_N^\fB|_{[r,\infty)},S^{N,+}_{r,\cdot}(a)\bigr) \ \mathop{\Longrightarrow}^{(d)\ }\ \bigl(\cM^\fB|_{[r,\infty)},\cS^+_{r,\cdot}(a)\bigr),
	\end{equation}
	for $r\in\r_+$, and
	\begin{equation}\label{eq:joint_single_I-}
		\bigl(\Phi_N^\fR|_{(-\infty,r]},S^{N,-}_{r,\cdot}(a)\bigr) \ \mathop{\Longrightarrow}^{(d)\ }\ \bigl(\cM^\fR|_{(-\infty,r]},\cS^-_{r,\cdot}(a)\bigr),
	\end{equation}
	for $r\in\r_-$. If $a=0$, the corresponding one-sided branches are killed at their starting points, so both assertions follow from Theorems~\ref{thm:meas} and~\ref{thm:1B_toth}. Henceforth assume $a>0$. We first consider $r\in\r_+$.
	
	By tightness, consider an arbitrary convergent subsequence of $(\Phi_N^\fB,S^{N,+}_{r,\cdot}(a))$, still indexed by $N$. By Skorokhod's representation theorem, we may assume almost sure convergence on a common probability space to $(\tilde{\cM}^\fB|_{\mathcal A},Y)$. The construction in Lemma~\ref{lem:M_b_r} reconstructs a white noise $\tilde{\cW}$ from $\tilde{\cM}^\fB|_{\mathcal A}$ and extends the field according to~\eqref{eq:til_cM}.
	
	For $x\ge r$, let $\mathcal F_x$ be the completed sigma-field generated by $\{Y(u):r\le u\le x\}$ and the restriction of $\tilde{\cW}$ to $\r_+\times(-\infty,x]$. The white noise $\tilde{\cW}$ remains a white noise relative to this filtration. Indeed, for every $\varepsilon>0$ and all sufficiently large $N$, the profile up to $x$ and the field on $(-\infty,x]$ depend on urns disjoint from those determining the field on $(x+\varepsilon,\infty)$. Passing this independence to the joint limit, then letting $\varepsilon\downarrow0$, gives independence of the future noise from $\mathcal F_x$. Since the reconstruction of $\tilde{\cW}$ preserves spatial restrictions, the same independence holds for the reconstructed noise.
	
	By Theorem~\ref{thm:1B_toth}, $Y$ has the same law as $\cS^+_{r,\cdot}(a)$. In particular, $Y$ is continuous, starts from $a$ and is absorbed at zero. Set $\tau^Y:=\inf\{x\ge r:Y(x)=0\}$. To prove~\eqref{eq:joint_single_I+}, it therefore suffices to show that, almost surely,
	\begin{align*}
		Y(x)=a+\sqrt{2}\int_r^x\tilde{\cM}^\fB((0,Y(u)],\dd u), \qquad r\le x<\tau^Y.
	\end{align*}
	
	\medskip
	
	For each $n\ge1$, set $r^n:=2^{-n}\lfloor2^nr\rfloor$ and $x_k^n:=r^n+k2^{-n}$ for $k\ge0$. Define $Y^{N,n}(r^n):=a$ and
	\begin{align}\label{eq:def_Y_N_n}
		Y^{N,n}(x_{k+1}^n):=Y^{N,n}(x_k^n)+\sqrt{2}\,\Phi_N^\fB\Bigl(\bigl(0,2^{-n}\lfloor2^nS^{N,+}_{r,x_k^n}(a)\rfloor\bigr]\times(x_k^n,x_{k+1}^n]\Bigr),
	\end{align}
	for $k\ge0$, where, in this recursion only, $S^{N,+}_{r,r^n}(a)$ is replaced by $a$. Extend $Y^{N,n}$ as a right-continuous step function, constant on each $[x_k^n,x_{k+1}^n)$. Since $Y$ has the law of a rescaled BESQ$^1$ process absorbed at zero,
	\begin{align}\label{eq:prob_zero_Y}
		\p\bigl(Y(x_k^n)\in2^{-n}\z,\ x_k^n<\tau^Y\bigr)=0, \qquad n,k\ge1.
	\end{align}
	Almost sure convergence to the continuous limit $Y$ gives $S^{N,+}_{r,x_k^n}(a)\to Y(x_k^n)$ for every $n,k\ge1$. Hence, on $\{x_k^n<\tau^Y\}$, the quantities $\lfloor2^nS^{N,+}_{r,x_k^n}(a)\rfloor$ and $\lfloor2^nY(x_k^n)\rfloor$ agree for all sufficiently large $N$, almost surely for every $n,k\ge1$.
	
	Set $Y_0^n:=2^{-n}\lfloor2^na\rfloor$ and $Y_k^n:=2^{-n}\lfloor2^nY(x_k^n)\rfloor$ for $k\ge1$, and define
	$$
	Y^n(x):=a+\sqrt{2}\sum_{\substack{k\ge0:\,x_{k+1}^n\le x}}\tilde{\cM}^\fB\bigl((0,Y_k^n]\times(x_k^n,x_{k+1}^n]\bigr), \qquad x\ge r^n,
	$$
	with zero summand when $Y_k^n=0$. Then $Y^n(r)=a$, and joint field convergence and the stabilization above give, almost surely, $Y^{N,n}(x)\to Y^n(x)$ for all $n\ge1$ and $r\le x<\tau^Y$.
	
	Set $\sigma(z,s):=2(\frac{s}{z})^\alpha\mathbbm{1}_{\{0<s\le z\}}$ for $z>0$, $\sigma(0,s):=0$, and $c_\alpha:=\frac{1}{2(2\alpha+1)}$. Define
	$$
	Y^\infty_x:=a+c_\alpha(x\wedge\tau^Y-r)+\frac1{\sqrt{2}}\int_r^{x\wedge\tau^Y}\!\int_0^\infty\sigma(Y(u),s)\tilde{\cW}(\dd s,\dd u), \qquad x\ge r.
	$$
	Direct integration gives, for $z,w\ge0$,
	\begin{equation}\label{eq:sigma_L2_bounds}
		\int_0^\infty\sigma(z,s)^2\,\dd s=\frac{4z}{2\alpha+1}, \qquad \int_0^\infty|\sigma(z,s)-\sigma(w,s)|^2\,\dd s\le4|z-w|.
	\end{equation}
	Fix $T>r$ and $0<\eta<a<K$, and set $\tau_{\eta,K}:=\inf\{u\ge r:Y(u)\notin(\eta,K)\}$. Write $\bar Y^n(u):=Y_k^n$ on $(x_k^n,x_{k+1}^n]$, with $\bar Y^n(r^n):=Y_0^n$. For $2^{-n}<\eta/2$, both $Y(u)$ and $\bar Y^n(u)$ lie in $[\eta/2,K]$ before $\tau_{\eta,K}$. For fixed $x\in(r,T]$, put $x_n:=2^{-n}\lfloor 2^n x\rfloor$. By Itô's isometry and~\eqref{eq:sigma_L2_bounds}, applied to the integrals stopped at $\tau_{\eta,K}$,
	$$
	\e\!\left[\mathbbm{1}_{\{x<\tau_{\eta,K}\}}|Y^\infty_x-Y^n(x)|^2\right]\le C\int_r^x\e\!\left[\mathbbm{1}_{\{u\le\tau_{\eta,K}\}}|Y(u)-\bar Y^n(u)|\right]\dd u+C2^{-n}.
	$$
	Here the last term accounts for $(r^n,r]$ and $(x_n,x]$, each of length at most $2^{-n}$. The integral tends to zero by continuity and bounded convergence. Letting $\eta\downarrow0$ and $K\uparrow\infty$, we obtain $Y^n(x)\to Y^\infty_x$ in probability on $\{x<\tau^Y\}$. It remains to prove the same convergence with $Y^\infty_x$ replaced by $Y(x)$.
	
	\bigskip
	
	By~\eqref{eq:proce_+1}, for $x_N>r_N$,
	\begin{equation}\label{eq:proce_++}
		S^{N,+}_{r,x}(a)=\frac{a_N}{N}-\frac{\eta^\fB_{a_N+1,r_N+1}}{N}+\sqrt{2}\sum_{k=r_N}^{x_N-1}\Phi_N^\fB\!\left(\left(0,S^N_{r,\frac{k}{N}}(a)\right]\times\left(\frac{k}{N},\frac{k+1}{N}\right]\right).
	\end{equation}
	Write $\mu(m):=\e[\fD^+_{\tau_m^\fB}]$. Proposition~\ref{prop:exp_discre} gives $\mu(m)\to c_\alpha$. Corollary~3.4 and Lemma~3.6 of~\cite{KMMP26} imply that, for any fixed $0<c<C_0<\infty$,
	\begin{equation}\label{eq:urn_increment_var}
		\operatorname{Var}\!\left(\fD^+_{\tau_m^\fB}-\fD^+_{\tau_l^\fB}\right)\le C|m-l|+C\sqrt N\log^2(2+N), \qquad cN\le l,m\le C_0N.
	\end{equation}
	Fix $x\in(r,T]$ and $2^{-n}<\eta/4$, with $n$ large enough that $x_n\ge r$. Set $\ell_N:=\lfloor Nx_n\rfloor$, $m_{r_N+1}:=a_N+1$ and $m_j:=NS^{N,+}_{r,(j-1)/N}(a)$ for $j\ge r_N+2$. For $r_N<j\le\ell_N$, let $k(j)$ satisfy $x_{k(j)}^n<j/N\le x_{k(j)+1}^n$ and put
	$$
	q_j:=\begin{cases}
		\lfloor NY_0^n\rfloor,&k(j)=0,\\
		\left\lfloor N2^{-n}\left\lfloor 2^n S^{N,+}_{r,x_{k(j)}^n}(a)\right\rfloor\right\rfloor,&k(j)\ge1.
	\end{cases}
	$$
	Let $\mathcal G_j$ be generated by the urns at sites at most $j$, and set
	$$
	\rho_N:=\inf\{j>r_N:m_j/N\notin[\eta/2,2K]\}.
	$$
	The heights $m_j,q_j$ and the event $\{j<\rho_N\}$ are $\mathcal G_{j-1}$-measurable. For $j<\rho_N$, both heights lie in $[\eta N/8,2KN]$ for all sufficiently large $N$.
	For deterministic rectangles, set $\bar\Phi_N^\fB(A):=\Phi_N^\fB(A)-\e[\Phi_N^\fB(A)]$, and interpret a signed height interval $(s,t]$ as the difference of the cumulative values at $t$ and $s$. Independence of the urns shows that
	$$
	\bar M^{N,n}(x):=\sqrt{2}\sum_{r_N<j\le\ell_N,\, j<\rho_N} \bar\Phi_N^\fB\!\left(\bigl(q_j/N,m_j/N\bigr]\times\bigl((j-1)/N,j/N\bigr]\right)
	$$
	is a sum of martingale differences. On $\{\rho_N>x_N\}$, write
	$$
	S^{N,+}_{r,x}(a)-Y^{N,n}(x)=\bar M^{N,n}(x)+R^{N,n}(x),
	$$
	and set $R^{N,n}(x)=0$ on the complement. The remainder consists of the initial correction in~\eqref{eq:proce_++}, the conditional means $N^{-1}(\mu(m_j)-\mu(q_j))$, and the two endpoint sums over $(r^n,r]$ and $(x_n,x]$. Put $\varepsilon_N:=\sup_{m\ge\eta N/8}|\mu(m)-c_\alpha|$. Stopping the final endpoint sum at $\rho_N$, Lemma~\ref{lm:4_1} and orthogonality of the centered urn increments give
	$$
	\e|R^{N,n}(x)|^2\le C(2^{-n}+N^{-1}+\varepsilon_N^2).
	$$
	Indeed, the initial correction has second moment $O(N^{-1})$, each endpoint sum contains at most $N2^{-n}+1$ terms with conditional variance at most $C/N$, and the sum of conditional means is bounded by $C\varepsilon_N$ for the interior terms and $C(2^{-n}+N^{-1})$ at the endpoints.
	
	Let $\Omega_N(h)$ be the modulus of continuity of $S^{N,+}_{r,\cdot}(a)$ on $[r,T]$. Since
	$$
	\frac{|m_j-q_j|}{N}\le 2^{-n}+\frac2N+\Omega_N(2^{-n}+N^{-1}),
	$$
	and the left-hand side is at most $4K$ before $\rho_N$,~\eqref{eq:urn_increment_var} gives
	$$
	\e|\bar M^{N,n}(x)|^2\le C\left(2^{-n}+N^{-1}+\e[\Omega_N(2^{-n}+N^{-1})\wedge4K]+N^{-1/2}\log^2(2+N)\right).
	$$
	Uniform convergence to the continuous limit $Y$ and bounded convergence give
	$$
	\lim_{n\to\infty}\limsup_{N\to\infty}\e[\Omega_N(2^{-n}+N^{-1})\wedge4K]=0.
	$$
	On $\{x<\tau_{\eta,K}\}$, we have $\rho_N>x_N$ for all sufficiently large $N$. Passing first to the limit in $N$ and then in $n$ therefore gives, for every $\varepsilon>0$,
	$$
	\p\bigl(|Y(x)-Y^n(x)|>\varepsilon,\ x<\tau_{\eta,K}\bigr)\longrightarrow0.
	$$
	Together with the convergence of $Y^n(x)$ to $Y^\infty_x$ in probability on $\{x<\tau^Y\}$, this gives $Y(x)=Y^\infty_x$ almost surely on $\{x<\tau_{\eta,K}\}$. Letting $\eta\downarrow0$ and $K\uparrow\infty$ yields the same identity on $\{x<\tau^Y\}$. A countable dense set of $x$ and continuity give the SDE on $[r,\tau^Y)$. Absorption and strong uniqueness now prove~\eqref{eq:joint_single_I+}.
	
	\medskip
	
	For $r\le0$, reverse the spatial direction and use~\eqref{eq:proce_-1}. Exchanging colors identifies $(-\fD^-_{\tau_m^\fR})_{m\ge0}$ in law with $(\fD^+_{\tau_m^\fB})_{m\ge0}$, so the same estimates apply to $\Phi_N^\fR$. When $r=0$, the origin urn contributes a term vanishing in $L^2$ by Lemma~\ref{lm:4_1}. The reversed-mesh approximations then identify the limiting branch as the solution to~\eqref{eq:poly_GRKT_back}, driven by $\tilde{\cM}^{\fR,*}$ and absorbed at zero. This proves~\eqref{eq:joint_single_I-} and completes the proof.
\end{proof}

\begin{remark}[Asymptotically free case]\label{rem:AF}
	Propositions~\ref{prop:AF_expe} and~\ref{prop:cov_AF} suggest that our joint Ray--Knight approach could also recover the BMPE limit in the asymptotically free case. Let $w$ be positive and monotone, satisfying~\eqref{eq:weight_func} with $\alpha=0$, and let $\gamma<1$ be given by~\eqref{def:gamma}. The urn estimates suggest that the limiting profiles satisfy
	\begin{align*}
		\dd\cS_{r,x}(a)&=\sqrt{2}\,\cW([0,\cS_{r,x}(a)],\dd x)+\bigl(1-\gamma\;\big|_0\;\gamma\bigr)(x)\d x, && x\ge r,\\
		\dd\cS_{r,-x}(a)&=-\sqrt{2}\,\cW^*([0,\cS_{r,-x}(a)],\dd x)+\bigl(1-\gamma\;\big|_0\;\gamma\bigr)(x)\d x, && x\ge-r,
	\end{align*}
	up to their killing times, with $\cS_{r,r}(a)=a$. Under the normalization $2\cS_{r,\cdot}(a/2)$, the profiles have squared Bessel dimension $2(1-\gamma)$ between $r$ and $0$, and $2\gamma$ outside, where they are absorbed at zero. This agrees with~\cite[Theorem~1A]{Toth96}.
	At inverse local times at the origin and at first hitting times, these profiles have the Ray--Knight laws of the BMPE
	$$
	W_t=B_t+\gamma\sup_{s\le t}W_s+\gamma\inf_{s\le t}W_s,\qquad W_0=0,
	$$
	where $B$ is standard Brownian motion; see~\cite[Theorems~3.1 and~3.4]{CPY98}. Extending the joint Ray--Knight argument and reconstructing the process from its local times would thus provide an alternative proof of the BMPE limit. 
\end{remark}

\section{The PSR flow}
\label{sec:PSR_flow}
In a recent work~\cite{Yu26}, the polynomially self-repelling flow with index $\alpha>0$, hereafter referred to as the PSR$^\alpha$ flow, was introduced. In this section, we recall its definition and basic properties, and present further properties used in the proof of Theorem~\ref{thm:main} and in the sequel.

\subsection{Basic facts about the PSR flow}\label{sec:basic}

Let $\cW$ be a standard space-time white noise on $\h$. Existence and uniqueness of the flow defined below follow~from \cite[Theorem~6.2]{Yu26}, where uniqueness means that any other flow driven by the same white noise and satisfying the same regularity properties coincides with $\cS$ almost surely.

\begin{definition}\label{def:PSR}
	Let $\alpha>0$. A PSR$^\alpha$ flow driven by $\cW$ is a collection $\cS=\{(\cS_{r,x}(a),\,x\ge r)\}_{(a,r)\in\h}$ of continuous processes satisfying the following conditions:
    \begin{enumerate}[label=\textnormal{(\arabic*)}]
        \item for each $(a,r)\in \h$, the process $(\cS_{r,x}(a),\, x \ge r)$ is almost surely the unique strong solution to the SDE~\eqref{SDE:for} and is absorbed at $\tau_{(a,r)}:=\inf \{x\ge r\vee 0:\, \cS_{r,x}(a)=0\}$.
        
        \item Almost surely,
        \begin{enumerate}[label=\textnormal{(R\arabic*)}, nosep]
            \item for all $(a,r) \in \h$, $\cS_{r,r}(a) = a$, \label{item:R1}
            \item for all $r \le x$, $a \mapsto \cS_{r,x}(a)$ is c\`adl\`ag, \label{item:R2}
            \item for any $(a_1,r_1), (a_2,r_2) \in \h$ and any $z \ge y \ge r_1\vee r_2$,
            $$
            \cS_{r_1,y}(a_1) < \cS_{r_2,y}(a_2) 
            \ \Longrightarrow\ 
            \cS_{r_1,z}(a_1) \le \cS_{r_2,z}(a_2).
            $$ \label{item:R3}
        \end{enumerate}
    \end{enumerate}
\end{definition}

Let $\cS$ denote the PSR$^\alpha$ flow driven by $\cW$. Define, for any $x\in \r$,
$$M(x):=\{\cS_{r,x}(a):\, r<x,\,(a,r)\in\cD\}.$$

\begin{lemma}\label{lm:basci_prop}
	The following properties hold almost surely:
	\begin{enumerate}[label=\textup{(\roman*)}]
		\item For all $x\in \r$, the set $M(x)$ is dense in $ \r_+$.
		
		\item For all $(a,r)\in \h$ and $x\ge r$,
		\begin{align}\label{eq:cS_I_+}
			\cS_{r,x}(a)
			= \inf \left\{ \cS_{r_n,x}(a_n):\, r_n< r,\ \cS_{r_n,r}(a_n)>a,\ (a_n,r_n)\in \cD \right\}.
		\end{align}
	\end{enumerate}
\end{lemma}

\begin{proof}
	Assertion~(i) is~\cite[Lemma~6.5(i)]{Yu26}, while~(ii) follows by the argument of~\cite[Lemma~5.2(ii)]{Yu26}.
\end{proof}
 
\bigskip

We next recall the self-duality and non-crossing properties established in~\cite{Yu26}. The \emph{dual} (or \emph{backward}) \emph{PSR flow}
$\cS^*=\{(\cS^*_{r,x}(a),\,x\le r)\}_{(a,r)\in\h}$ is defined by
\begin{align}\label{eq:def_dual_+}
	\cS^*_{r,x}(b):=\inf\{a\ge0:\,\cS_{x,r}(a)>b\}, \qquad x\le r,\ b\ge0.
\end{align}
Concatenating the forward and backward flow lines at their common starting point gives the \emph{two-sided PSR flow}. More precisely, for each $(a,r)\in\h$, define
\begin{align}\label{eq:cS_bar}
\overline{\cS}_{r,x}(a):= 
\begin{cases}
    \cS_{r,x}(a), & \text{if }\  x \ge r, \\
    \cS^*_{r,x}(a), & \text{if }\  x \le r,
\end{cases}
\end{align}
and write $\overline{\cS}=\bigl\{\overline{\cS}_{r,\cdot}(a)\bigr\}_{(a,r)\in\h}$. 

\begin{proposition}[Self-duality, {\cite[Theorem~6.6]{Yu26}}]\label{prop:self_dual}
The families
$$
\{(\cS_{r,x}(a),\,x\ge r)\}_{(a,r)\in\h}
\quad\text{and}\quad
\{(\cS^*_{-r,-x}(a),\,x\ge r)\}_{(a,r)\in\h}
$$
are identical in law. Furthermore, for each $(a,r)\in \h$, the two-sided PSR flow line $\overline{\cS}_{r,\cdot}(a)$ is a CRKC in the sense of Definition~\ref{def:RK_curve}. In particular,
$$
\{(\overline{\cS}_{r,x}(a),\, x\in \r)\}_{(a,r)\in \h}
\quad\text{and}\quad
\{(\overline{\cS}_{-r,-x}(a),\, x\in \r)\}_{(a,r)\in \h}
$$
are identical in law.
\end{proposition}

\begin{remark}[An alternative proof of Proposition~\ref{prop:self_dual}: a sketch]
	Let $\{(\cS^{\mathrm b}_{-r,-x}(a),\,x\ge r)\}_{(a,r)\in\h}$ be the PSR$^\alpha$ flow driven by $-\cW^*$. Concatenating the forward flow lines of $\cS$ with the backward flow lines of $\cS^{\mathrm b}$ yields the CRKCs of Definition~\ref{def:RK_curve}. By Theorem~\ref{thm:main} and Lemma~\ref{lem:noncrossing}, these curves almost surely do not cross for any starting points in $\cD$. Lemma~\ref{lm:basci_prop}(ii) and its time-reversed counterpart then allow us to apply the argument of \cite[Theorem~1.4]{Yu26} to obtain, almost surely,
	$$
	\cS^{\mathrm b}_{r,x}(b)=\inf\{a\ge0:\cS_{x,r}(a)>b\},\qquad (b,r)\in\h,\ x\le r.
	$$
	Thus $\cS^{\mathrm b}=\cS^*$, proving self-duality.
\end{remark}

\begin{proposition}[Non-crossing, {\cite[Proposition~6.7]{Yu26}}]\label{prop:non-crossing}
Almost surely, for any $(a_i,r_i)\in\h$, $i=1,2$, and any $y,z\in\r$,
$$
\overline\cS_{r_1,y}(a_1)<\overline\cS_{r_2,y}(a_2)\ \Longrightarrow\  \overline\cS_{r_1,z}(a_1)\le\overline\cS_{r_2,z}(a_2).
$$
\end{proposition}

We end this subsection with a consequence of right-continuity used in the proof of Theorem~\ref{thm:main}, which follows from property~\ref{item:R2} of the PSR flow.

\begin{corollary}\label{cor:cov_b}
	Fix $(a,r)\in\h$ and a sequence $(a_n)_{n\ge1}$ with $a_n\downarrow a$. Define $Y(x):=\lim_{n\to \infty}\cS_{r,x}(a_n)$ for $x\ge r$. Then, almost surely, $(Y(x),\,x\ge r)$ is the unique strong solution to~\eqref{SDE:for}, started from $a$, and is absorbed at $\tau_{(a,r)}$.
\end{corollary}

\subsection{Coalescence properties}

In this subsection, we investigate the coalescence properties of the PSR flow. A key ingredient is the \emph{coming-down-from-infinity} property in Proposition~\ref{prop:coming_down}: infinitely many flow lines started from bounded positions coalesce rapidly after any positive time, leaving only finitely many distinct paths. Our proof follows the argument for an infinite family of independent coalescing reflected/absorbed Brownian motions (FICRAB) in~\cite[Section~8.2]{Toth98}; see also~\cite{co_down_2, Kingman, co_down_1} for related notions.

We first recall the flow property. For each fixed $(a,r)\in\h$ and $x\ge r$, almost surely,
\begin{align}\label{eq:flow_prop}
	\cS_{r,y}(a)=\cS_{x,y}\circ\cS_{r,x}(a),\qquad y\ge x.
\end{align}
If $x\ge\tau_{(a,r)}$, this is immediate; otherwise, it follows from~\cite[Lemma~4.5(i)]{Yu26}. The next proposition establishes the \emph{perfect flow property} at positive intermediate heights: on a common event of probability one,~\eqref{eq:flow_prop} holds for all $(a,r)\in\h$ and $y\ge x\ge r$ such that $\cS_{r,x}(a)>0$. Its proof is deferred to Section~\ref{subsec:perf}.

\begin{proposition}[Perfect flow property]\label{prop:perfect}
	Almost surely, for all $(a,r)\in\h$ and $y\ge x\ge r$ such that $\cS_{r,x}(a)>0$,
	\begin{align}\label{eq:perf}
		\cS_{r, y}(a) = \cS_{x, y} \circ \cS_{r, x}(a).
	\end{align}
\end{proposition}

\medskip

For any $x\le y$, define
$$
M(x,y):=\{\cS_{r,y}(a):\, (a,r)\in \h,\, r<x\}.
$$
By Lemma~\ref{lm:basci_prop}(i), almost surely, $M(x,x)$ is dense in $\r_+$ for every $x\in\r$. Enumerate $\cD$ as $\{(a_n,r_n):n\in\N\}$ and write $\cS^n=\cS_{r_n,\cdot}(a_n)$. We call this countable family the \emph{skeleton} of the PSR flow. Since the dual flow is finite-valued, \eqref{eq:def_dual_+} implies that $a\mapsto\cS_{x,y}(a)$ is unbounded for every $x<y$. Consequently, almost surely,
$$
\sup\{\cS^n(y):\,n\in\N,\ r_n<x\}=\infty,\qquad x\le y.
$$
In particular, $M(x,y)$ is unbounded for every $x\le y$.

\begin{proposition}\label{prop:coming_down}
    Almost surely, for every $x < y$, the set $M(x,y)$ is locally finite.
\end{proposition}

Let $\cShat$ be the restriction to initial points in $\h$ of the non-killed flow on $\r^2$ described in~\cite[Remark~6.3]{Yu26}, whose existence follows from \cite[Theorem~4.3]{Yu26}. We take this flow to be driven by the same white noise $\cW$; its trajectories with initial points in $\h$ solve~\eqref{SDE:for} without absorption. By pathwise uniqueness, almost surely, for every $(a,r)\in\cD$, the trajectories $\cS_{r,\cdot}(a)$ and $\cShat_{r,\cdot}(a)$ coincide up to $\tau_{(a,r)}$, after which the former remains zero. For $x\le y$, define
$$
M'(x,y):=\{\cShat_{r,y}(a):\, (a,r)\in \cD,\, r<x\}.
$$
It follows from~\eqref{eq:cS_I_+} that, almost surely,
$$
M(x,y)\subseteq\overline{M'(x,y)\cup\{0\}},\qquad x\le y.
$$
Thus the following lemma implies Proposition~\ref{prop:coming_down}.

\begin{lemma}
	Almost surely, for every $x<y$, the set $M'(x,y)$ is locally finite.
\end{lemma}

\begin{proof}
	We follow the proof of~\cite[Proposition~2.2(ii)]{Toth98}, using Corollary~\ref{app:coal} in place of the meeting-time estimate. Fix dyadic numbers $x<y$ and an integer $K\ge1$. For $m\ge1$, let
	$$
	N_m:=\#\{\cShat_{x,y}(j2^{-m}):j=0,\ldots,K2^m\}.
	$$
	By monotonicity and Corollary~\ref{app:coal},
	$$
	\e(N_m)=1+\sum_{j=0}^{K2^m-1}\p\!\left(\cShat_{x,y}(j2^{-m})<\cShat_{x,y}((j+1)2^{-m})\right)\le1+\frac{(2\alpha+1)K}{2(y-x)}.
	$$
	Since the dyadic meshes are nested, monotone convergence implies that $\{\cShat_{x,y}(a):a\in\bD\cap[0,K]\}$ is almost surely finite. Taking a countable intersection over $K$, Step~1 of the proof of~\cite[Proposition~2.2(ii)]{Toth98} yields local finiteness of $M'(x,y)$ via a truncation argument. Step~2 of the same proof extends this result to arbitrary $x<y$.
\end{proof}

For $a>0$ and $r\le x$, define the left limit $\cS_{r,x}(a-):=\lim_{a'\uparrow a}\cS_{r,x}(a')$. A point $(a,r)\in\h^*$ is called a \emph{bifurcation point} if $\cS_{r,x}(a)>\cS_{r,x}(a-)$ for some $x>r$. Proposition~\ref{prop:coming_down} yields the following structural properties.

\begin{corollary}\label{p:properties}
	Almost surely,
	\begin{enumerate}[label=\textup{(\roman*)}]
		\item for any $r<x$ and $a\ge 0$, $\cS_{r,x}(a')=\cS_{r,x}(a)$ for all $a'>a$ close enough to $a$;
        
		\item for any $r<x$ and $a>0$, $\cS_{r,x}(a')=\cS_{r,x}(a-)$ for all $a'<a$ close enough to $a$;
        
        \item for any $r\in \r$, the set of bifurcation points at level $r$ is countable.
	\end{enumerate}
\end{corollary}

\begin{corollary}\label{cor:dual}
    Almost surely, for any $r<x$ and $a,b\ge 0$, 
    $$
    \cS^*_{x,r}(b)>a \ \Longleftrightarrow \ b\ge \cS_{r,x}(a).
    $$
\end{corollary}

\begin{proof}
	The equivalence follows from the definition~\eqref{eq:def_dual_+} and Corollary~\ref{p:properties}(i).
\end{proof}

For each fixed pair of starting points in $\h$, pathwise uniqueness implies that the corresponding flow lines coalesce upon meeting, almost surely. The next lemma strengthens this to simultaneous coalescence for all pairs whose flow lines meet strictly after both starting times.

\begin{lemma}[Coalescing property]\label{lm:coal}
	Almost surely, for all $(a_1,r_1),(a_2,r_2)\in\h$, if $\cS_{r_1,x}(a_1)=\cS_{r_2,x}(a_2)$ for some $x>r_1\vee r_2$, then $\cS_{r_1,y}(a_1)=\cS_{r_2,y}(a_2)$ for all $y\ge x$.
\end{lemma}

\begin{proof}
	By~\eqref{eq:cS_I_+} and Proposition~\ref{prop:coming_down} applied to $M(r_1\vee r_2,x)$, we may choose $n_1,n_2\in\N$ such that $r_{n_i}<r_i$, $\cS^{n_i}(r_i)>a_i$ for $i=1,2$, and
	\begin{align*}
		\cS^{n_1}(x)=\cS_{r_1,x}(a_1)=\cS_{r_2,x}(a_2)=\cS^{n_2}(x).
	\end{align*}
	The two skeleton paths coalesce at $x$, so it suffices to show that $\cS_{r_i,y}(a_i)=\cS^{n_i}(y)$ for all $y\ge x$ and $i=1,2$. We prove this for $i=1$. Let $(a',r')\in\cD$ satisfy $r'<r_1$ and $\cS_{r',r_1}(a')>a_1$. If $\cS_{r',r_1}(a')\le \cS^{n_1}(r_1)$, then, by Property~\ref{item:R3} and the fact that $\cS_{r',\cdot}(a')$ and $\cS^{n_1}$ coalesce once they meet, $\cS_{r',x}(a')=\cS^{n_1}(x)$, and hence $\cS_{r',y}(a')= \cS^{n_1}(y)$ for all $y\ge x$. Otherwise, Property~\ref{item:R3} yields $\cS_{r',y}(a')\ge \cS^{n_1}(y)$ for all $y\ge x$. Therefore, the desired conclusion follows from~\eqref{eq:cS_I_+}.
\end{proof}

By Corollary~\ref{p:properties} and Lemma~\ref{lm:coal}, the trajectories $\cS_{r,\cdot}(a-)$ and $\cS_{r,\cdot}(a)$ can bifurcate only at their common starting time $r$.

\medskip

We next introduce the \emph{instantaneous coalescence} property: almost surely, every flow line beyond its starting point is contained in the union of the skeleton paths.
\begin{align*}
	\bigcup_{n \in \N} \textup{graph} (\cS^n|_{[r_n, \infty)})=\{(x, \cS^n(x)):\, n\in \N, \, x\ge r_n\}.
\end{align*}

\begin{proposition}\label{prop:skeleton}
	Almost surely, for every $(a,r) \in \h$ and $\varepsilon > 0$, there are points  $(a_n,r_n)\in \cD$ arbitrarily close to $(a,r)$ such that $r_n < r$, $\cS^n(r) \in (a , a+ \varepsilon)$, and $\cS_{r,x}(a)=\cS^n(x)$ for all $x\ge r+\varepsilon$.
\end{proposition}

\begin{proof}
	Almost surely, for any $(a,r)\in\h$ and $\varepsilon>0$, by~\eqref{eq:cS_I_+} and Proposition~\ref{prop:coming_down} applied to $M(r,r+\varepsilon)$, there exists $k=k(a,r,\varepsilon)$ such that $r_k<r$, $\cS^k(r)>a$ and $\cS_{r,r+\varepsilon}(a)=\cS^k(r+\varepsilon)$. By Lemma~\ref{lm:coal},
	\begin{equation*}
		\cS_{r,y}(a)=\cS^k(y),\qquad \forall y\ge r+\varepsilon.
	\end{equation*}
	Moreover, Lemma~\ref{lm:basci_prop}(i) allows us to choose $k$ such that $a^k:=\cS^k(r)\in(a,a+\varepsilon)$.
	
	For arbitrarily small $\delta\in(0,\tfrac{a^k-a}{2})$, Lemma~\ref{lm:basci_prop}(i) gives an index $n=n(\delta)$ such that $r_n<r$ and $\cS^n(r)\in(a,a+\delta)$. By property~\ref{item:R3},
	$$
	\cS_{r,y}(a)\le\cS^n(y)\le\cS^k(y)=\cS_{r,y}(a),\qquad \forall y\ge r+\varepsilon.
	$$
	By density of $\cD$ in $\h$ and continuity of $\cS^n$ and $\cS^k$, we can choose $(a',r')\in\cD$ such that
	\begin{align*}
		\|(a',r')-(a,r)\|<2\delta,\qquad r>r'>\max(r_k,r_n),\qquad \cS^n(r')<a'<\cS^k(r').
	\end{align*}
	Property~\ref{item:R3} then implies
	$$
	\cS_{r',r}(a')\in(a,a+\varepsilon),\qquad \cS_{r',y}(a')=\cS_{r,y}(a),\quad \forall y\ge r+\varepsilon.
	$$
	Since $\delta$ can be chosen arbitrarily small, this proves the claim.
\end{proof}

\subsection{Properties of the double PSR flow}

We classify points $z=(a,r)\in \h$ according to the local behavior of the PSR flow and its dual $(\cS,\cS^*)$, which we call the \emph{double PSR flow}. Let $\gamma,\gamma' \in \cS$ be two flow lines entering $z$, that is, starting before time $r$ and satisfying $\gamma(r)=\gamma'(r)=a$. We write $\gamma =^{z}_{\mathrm{in}} \gamma'$ if $\gamma|_{[r-\varepsilon, r]} \equiv \gamma'|_{[r-\varepsilon, r]}$ for some $\varepsilon>0$, and denote by $m_{\mathrm{in}}(z)$ the number of resulting equivalence classes of the flow lines entering $z$. Similarly, a flow line leaves $z$ if it starts no later than time $r$ and satisfies $\gamma(r)=a$, and we write $\gamma =^{z}_{\mathrm{out}} \gamma'$ if $\gamma|_{[r,r+\varepsilon]} \equiv \gamma'|_{[r,r+\varepsilon]}$ for some $\varepsilon>0$. Let $m_{\mathrm{out}}(z)$ denote the number of equivalence classes of the flow lines leaving $z$ under $=^{z}_{\mathrm{out}}$. The quantities $m_{\mathrm{in}}^*(z)$ and $m_{\mathrm{out}}^*(z)$ are defined analogously for the dual flow $\cS^*$, with the direction of evolution reversed.

\begin{definition}
    Let $z\in \h$. The \textbf{type} and \textbf{dual type} of $z$ are defined as the pairs $$
	\bigl(m_{\mathrm{in}}(z), m_{\mathrm{out}}(z)\bigr)
	\quad\text{and}\quad
	\bigl(m_{\mathrm{in}}^*(z), m_{\mathrm{out}}^*(z)\bigr),
	$$
	respectively.
\end{definition}

For related classifications of coalescing systems and the Brownian web, see~\cite{Toth98, Fontes06, BWBN17}. On $\h^*$, the PSR flow admits fewer multiplicity types than the Brownian web. Indeed, Proposition~\ref{prop:perfect}, together with self-duality, implies that, almost surely, $m_{\mathrm{out}}(z)=m_{\mathrm{out}}^*(z)=1$ for every $z\in\h^*$. The precise classification is stated below, while its proof is deferred to the end of this section.

\begin{proposition}[Special points of the double PSR flow]\label{prop:special_points}
	Let $(\cS, \cS^*)$ be the double PSR flow. Then almost surely, every point $z\in \h^*$ belongs to one of the following five types according to $\left(m_{\mathrm{in}}(z), m_{\mathrm{out}}(z)\right) /\left(m^*_{\mathrm{in}}(z), m^*_{\mathrm{out}}(z)\right)$:
    \begin{align}\label{eq:types}
        (0,1) /(0,1),(0,1) /(1,1),(0,1) /(2,1),(1,1) /(0,1),(2,1) /(0,1).
    \end{align}
	Moreover, almost surely:
	\begin{enumerate}[label=\textup{(\alph*)}, left=0pt]
		\item The set of points of type $(0,1) /(0,1)$ has full Lebesgue measure in $\h^*$;

		\item The sets of points of type $(0,1) /(1,1)$ and $(1,1) /(0,1)$ have Hausdorff dimension $\frac{3}{2}$;

		\item The sets of points of type $(0,1) /(2,1)$ and $(2,1) /(0,1)$ are countable.
	\end{enumerate}
\end{proposition}

\smallskip

By the non-crossing property of $\overline{\cS}$ (Proposition~\ref{prop:non-crossing}), each flow line acts as a barrier separating the half-plane $\h$ into two regions. This decomposition underpins the spatial Markov property for the two-sided PSR flow, which is analogous to that in \cite[Proposition~5.1]{Toth98}. Fix $(a_0,r_0)\in\h$. If $a_0>0$, define the index sets
\begin{equation}\label{def:lamb}
\begin{aligned}
    \Lambda_- &:= \left\{(a,r)\in\h:\, \overline{\cS}_{r,r_0}(a)<a_0\right\},\\
    \Lambda_+ &:= \left\{(a,r)\in\h:\, \overline{\cS}_{r,r_0}(a)>a_0\right\}.
\end{aligned}
\end{equation}
If $a_0=0$, define instead
\begin{equation}\label{def:lamb_0}
\begin{aligned}
    \Lambda_- &:= \left\{(a,r)\in\h:\, \overline{\cS}_{r,r_0}(a)=0\right\}\setminus\{(a_0,r_0)\},\\
    \Lambda_+ &:= \left\{(a,r)\in\h:\, \overline{\cS}_{r,r_0}(a)>0\right\}.
\end{aligned}
\end{equation}
These index sets induce two non-intersecting subfamilies of $\overline{\cS}$ by
\begin{align}
	\label{def:barS}\overline \cS_- := \left\{ \overline \cS_{r, \cdot}(a) : \, (a,r)\in \Lambda_- \right\},\qquad  \overline \cS_+ := \left\{ \overline \cS_{r, \cdot}(a) : \, (a,r)\in \Lambda_+ \right\}.
\end{align}
Moreover, we claim that, almost surely, $\overline \cS_-\uplus \overline \cS_+\uplus \{\overline \cS_{r_0, \cdot}(a_0)\}=\overline \cS$. For $a_0=0$, this is immediate. For $a_0>0$, it suffices to show that $\{(a,r)\in \h:\, \overline \cS_{r, r_0}(a) = a_0\}=\{(a_0,r_0)\}$. Indeed, by Proposition~\ref{prop:skeleton} and property~\ref{item:R1}, almost surely,
\begin{align*}
	\{(a,r)\in \h:\, \overline \cS_{r, r_0}(a) = a_0\}\setminus \{(a_0,r_0)\}
	= \{(a,r)\in \cD:\, \overline \cS_{r, r_0}(a) = a_0,\, r\neq r_0\}=\varnothing.
\end{align*}

\begin{proposition}[Spatial Markov property]
	\label{prop:spatial_markov}
	Fix $(a_0,r_0)\in\h$, and let $\overline \cS_-$ and $\overline \cS_+$ be defined as in~\eqref{def:barS}. Then, $\overline{\cS}_-$ and $\overline{\cS}_+$ are conditionally independent given $\overline{\cS}_{r_0,\cdot}(a_0)$.
\end{proposition}

\begin{proposition}\label{prop:no_31pts}
    Almost surely, there do not exist a point $(a,x)\in \h$ and three PSR flow lines $\cS^i(\cdot)=\cS_{r_i,\cdot}(a_i)$, $i=1,2,3$, starting before $x$ such that, for some $\varepsilon>0$, the paths $\cS^1,\cS^2,\cS^3$ are pairwise distinct on $[x-\varepsilon,x)$ and satisfy $\cS^i(x)=a$ for $i=1,2,3$.
\end{proposition}

The proofs of Propositions~\ref{prop:spatial_markov} and~\ref{prop:no_31pts} are given in Section~\ref{sec:deco}.

\smallskip

\begin{proof}[Proof of Proposition~\ref{prop:special_points}]
	Let $z=(a,r)\in\h^*$. By Proposition~\ref{prop:no_31pts}, $m_{\mathrm{in}}(z)\le2$, and Proposition~\ref{prop:perfect} implies that every trajectory entering $z$ coincides with $\cS_{r,\cdot}(a)$ after time $r$, so $m_{\mathrm{out}}(z)=1$. By self-duality, the same conclusions hold for the dual flow.
	
    Suppose that both $m_{\mathrm{in}}(z)$ and $m^*_{\mathrm{in}}(z)$ are nonzero. Then there exist $(a_1,r_1),(a_2,r_2)\in\h$ with $r_1<r<r_2$ such that $\cS_{r_1,r}(a_1) = \cS^*_{r_2,r}(a_2) = a$. By Corollary~\ref{cor:dual}, $\cS^*_{r_2,r}(a_2)=a$ implies that $\cS_{r,r_2}(a)>a_2$. Proposition~\ref{prop:perfect} yields
    \begin{align}\label{eq:for_a}
        \cS_{r_1,r_2}(a_1)=\cS_{r,r_2}\circ \cS_{r_1,r}(a_1)
        =\cS_{r,r_2}(a)>a_2.
    \end{align}
    By Proposition~\ref{prop:self_dual}, the same argument applied to the dual flow yields
    \begin{align}\label{eq:back_a}
        \cS^*_{r_2,r_1}(a_2)=\cS^*_{r,r_1}\circ \cS^*_{r_2,r}(a_2)
        =\cS^*_{r,r_1}(a)>a_1.
    \end{align}
    Together, \eqref{eq:for_a} and \eqref{eq:back_a} contradict Proposition~\ref{prop:non-crossing}. Hence, almost surely, there is no point $z\in\h^*$ at which $m_{\mathrm{in}}(z)$ and $m^*_{\mathrm{in}}(z)$ are nonzero. Consequently, only the five types in~\eqref{eq:types} can occur.

    By Proposition~\ref{prop:skeleton} and self-duality, any point $z\in\h^*$ with $m_{\mathrm{in}}(z)=2$ or $m^*_{\mathrm{in}}(z)=2$ must be a coalescence point of two forward or two backward skeleton paths. Since each skeleton is countable and each pair of paths can coalesce at most once, the set of such points is countable, proving (c). Applying Proposition~\ref{prop:skeleton} again, the set of points of type $(1,1)/(0,1)$ is
    $$
    \bigcup_{n\in\N}\{(\cS^n(x),x):x>r_n,\, \cS^n(x)>0\},
    $$
    with the forward coalescence points removed. On every non-degenerate compact interval where a skeleton path stays strictly positive, its graph has Hausdorff dimension $3/2$, by the argument in~\cite[Appendix~B]{AWY24} applied to its rescaled BESQ marginal law. A countable union of such graph segments, with a countable set removed, still has Hausdorff dimension $3/2$. By self-duality, the same holds for points of type $(0,1)/(1,1)$, proving (b).
    Finally, since sets of Hausdorff dimension $3/2<2$ and countable sets have zero two-dimensional Lebesgue measure, the set of points of type $(0,1)/(0,1)$ has full Lebesgue measure in $\h^*$. This proves (a).
\end{proof}

\begin{remark}
    In this section, we have shown that the PSR flow shares many properties with the BESQ flow (see, e.g., \cite[Section 2.3]{AWY24} and \cite[Section 2]{elie23}). We trace this similarity to the coalescence-time estimates established in Corollary~\ref{app:coal}. While properties of the BESQ flow are typically studied through its embedding in PRBM, where Ray--Knight theorems identify it as a local time flow, our analysis of the PSR flow proceeds directly from its intrinsic flow structure. The same approach can also be used to derive analogous properties of BESQ flows directly from their flow structure. 
\end{remark}

\section{Construction and properties of the PSRM}
\label{sec:constr_PSRM}
We construct the PSRM from the two-sided PSR flow $\overline{\cS}$ of~\eqref{eq:cS_bar}, following~\cite{Toth98}.  While stochastic flows are often realized as local time flows, we consider the inverse problem of constructing a process from a given flow, here the PSR flow.

\subsection{Construction of the PSRM}

For $(a,r)\in\h$, set
\begin{align*}
	D(a,r):=\{(b,s)\in\h:b\le\overline{\cS}_{r,s}(a)\}.
\end{align*}
By Proposition~\ref{prop:non-crossing} and Corollary~\ref{cor:dual}, these regions are totally ordered, and distinct starting points give distinct profiles. Continuity then makes each strict difference have positive area. By Lemma~\ref{lm:basci_prop}(i), every profile is bounded above by a two-sided skeleton profile, which has compact support. Hence the area map
\begin{align}\label{def_T}
	T(a,r):=|D(a,r)|=\int_\r\overline{\cS}_{r,x}(a)\,\dd x
\end{align}
is finite and injective. More precisely, non-crossing and Corollary~\ref{cor:dual} give the right-continuous counterpart of~\cite[Proposition~3.1]{Toth98}:
\begin{align}\label{eq:area_order}
	T(b,s)<T(a,r)\quad\Longleftrightarrow\quad b<\overline{\cS}_{r,s}(a),\qquad (a,r),\ (b,s)\in\h.
\end{align}
For each $r$, the map $a\mapsto T(a,r)$ is strictly increasing, right-continuous and unbounded. For $a>0$, write
\begin{align*}
	T^-(a,r):=\lim_{b\uparrow a}T(b,r).
\end{align*}

\begin{lemma}\label{lm:T_sigl1}
	Almost surely, $T(\cD)$ is dense in $\r_+$.
\end{lemma}

\begin{proof}
	The argument in~\cite[Lemma~3.2]{Toth98} gives density of $T(\h)$. Between any two distinct profiles there is a nonempty open region containing a dyadic point; by~\eqref{eq:area_order}, its area lies strictly between their areas. This gives density of $T(\cD)$.
\end{proof}

For $t\ge0$, define
\begin{align}\label{eq:def_P_t}
	P_t:=\bigcap_{\varepsilon>0}\overline{\{(a,r)\in\h:|T(a,r)-t|<\varepsilon\}}.
\end{align}
The proofs of~\cite[Lemmas~3.3--3.4]{Toth98} carry over by Propositions~\ref{prop:coming_down} and~\ref{prop:skeleton} and their dual counterparts. Hence, almost surely, $P_t$ is a singleton for every $t\ge0$.

\begin{definition}\label{def:PSRM}
	For each $t\ge 0$, let the singleton $P_t$ be denoted by $\{(H_t, X_t)\}$. We then refer to the process $(X_t,\, t\ge 0)$ as the polynomially self-repelling motion (PSRM).
\end{definition}

\begin{remark}\label{rmk:def_L}
	For $t\ge0$ and $y\in\r$, define
	\begin{align}\label{eq:def_Lt}
		L_t(y):=\inf\{a>0:T(a,y)>t\}.
	\end{align}
	Since $T(\cdot,y)$ is strictly increasing and unbounded, $L_\cdot(y)$ is well-defined, continuous and nondecreasing. By~\eqref{eq:area_order}, $L_{T(a,r)}(y)=\overline{\cS}_{r,y}(a)$. Thus, Lemma~\ref{lm:T_sigl1} gives spatial continuity at a dense set of times. Monotonicity and pointwise continuity in time then give joint continuity of $(t,y)\mapsto L_t(y)$ by comparison with profiles at nearby times in this dense set. At $t=0$, the same argument uses $L_0\equiv0$ as the lower bound.
	
	For any $t\ge0$, choose $(a_n,r_n)\to(H_t,X_t)$ such that $T(a_n,r_n)\to t$. Joint continuity gives $a_n=L_{T(a_n,r_n)}(r_n)\to L_t(X_t)$, and hence
	\begin{align}\label{eq:L_t_H_t}
		H_t=L_t(X_t),\qquad t\ge0.
	\end{align}
	The definition of $L$ and right-continuity of $T(\cdot,y)$ also imply, for every $t\ge0$,
	\begin{align}\label{eq:t_between}
		T^-(H_t,X_t)\le t\le T(H_t,X_t)&\quad\text{if }H_t>0,\nonumber\\
		0\le t\le T(0,X_t)&\quad\text{if }H_t=0.
	\end{align}
	Theorem~\ref{thm:occupation} identifies $L$ as the occupation density of $X$.
\end{remark}

\subsection{Properties of PSRM}\label{subsec:prop_PSRM}

We first record some basic properties of the PSRM.

\begin{proposition}
	\begin{enumerate}[label=\textup{(\roman*)}]
		\item The processes $(X_{t}, H_{t})$ and $(-X_{t}, H_{t})$ are identical in law.
		\item Almost surely, for any $x \in \r$, the set $\{t \in \r_{+} : X_t = x\}$ is unbounded.
		\item Almost surely, $t \mapsto (X_{t}, H_{t})$ is continuous on $[0, \infty)$ and $(X_{0}, H_{0}) = (0, 0)$.
		\item For any $a > 0$, $t \mapsto (X_{a t}, H_{a t})$ and $t \mapsto (a^{1/2} X_{t}, a^{1/2} H_{t})$ are identical in law.
	\end{enumerate}
\end{proposition}

\begin{proof}
	Part (i) follows from Proposition~\ref{prop:self_dual}. For (ii), use $X_{T(a,r)}=r$ and $T(a,r)\to\infty$ as $a\to\infty$. The proof of~\cite[Proposition~3.5(i)]{Toth98} gives (iii). For (iv), fix $c>0$ and define
	$$
	\overline{\cS}^{(c)}_{r,x}(b):=c^{-1}\overline{\cS}_{cr,cx}(cb).
	$$
	By scaling invariance, $\overline{\cS}^{(c)}$ has the same law as $\overline{\cS}$. Its area map is $T^{(c)}(b,r)=c^{-2}T(cb,cr)$, so the corresponding motion is $(c^{-1}X_{c^2t},c^{-1}H_{c^2t})_{t\ge0}$. Taking $c=a^{1/2}$ yields (iv).
\end{proof}

The following proposition is a direct consequence of the definition of $L$ in Remark~\ref{rmk:def_L}.

\begin{proposition}\label{prop:ray-knight}
	Almost surely,
	\begin{enumerate}[label=\textup{(\roman*)}]
		\item for every $(a,r)\in\h$ and $x\in\r$, $L_{T(a,r)}(x)=\overline{\cS}_{r,x}(a)$;
		\item for every $(a,r)\in\h^*$ and $x\in\r$, $L_{T^-(a,r)}(x)=\overline{\cS}_{r,x}(a-):=\lim_{b\uparrow a}\overline{\cS}_{r,x}(b)$.
	\end{enumerate}
\end{proposition}

We next identify $L$ as the occupation density of $X$, following the argument of~\cite[Proposition~4.1(iii) and Theorem~4.2(i)]{Toth98}; we refer the reader there for details.

\begin{theorem}[Occupation-time density property]\label{thm:occupation}
	Almost surely, for any measurable, bounded, real-valued function $h$ and any $t\ge 0$,
	\begin{align}\label{eq:occu}
		\int_0^t h(X_s) \dd s=\int_\r h(x)L_t(x) \dd x.
	\end{align}
	In particular, $L_t(\cdot)$ is the Radon--Nikodym derivative of the occupation time measure $\mu_t$ with respect to Lebesgue measure, where $\mu_t$ is defined by
	$$
	\mu_t(A):=\int_0^t \mathbbm{1}_{\{X_s\in A\}}\,\dd s,
	\qquad A\subset \r \text{ Borel}.
	$$
\end{theorem}

\bigskip

We now show that the joint laws of the inverse local-time profiles characterize the PSRM. Let $\cX$ be a continuous process admitting local times $\cL_\cX$ in the sense of Definition~\ref{def:local_time}, and define
$$
\tau_a^r(\cX):=\inf\{t\ge0:\cL_\cX(t,r)>a\},\qquad (a,r)\in\h.
$$
For the PSRM, right-continuity of $T(\cdot,r)$ and~\eqref{eq:def_Lt} give
$$
T(a,r)=\inf\{t\ge0:L_t(r)>a\}.
$$

\begin{assumption}\label{assump_H}
	For any $p\in \N$ and any $(a_i,r_i)\in \h$, $i=1,\ldots,p$, the family
	$$
	\bigl(\cL_\cX(\tau_{a_i}^{r_i}(\cX),x),\, x\in\r\bigr)_{i=1}^p
	$$
	has the same joint law as 
	$$
	\bigl(\overline{\cS}_{r_i,x}(a_i),\, x\in\r\bigr)_{i=1}^p.
	$$
\end{assumption}

By Proposition~\ref{prop:ray-knight}(i), the PSRM satisfies Assumption~\ref{assump_H}. 

\begin{proposition}
    Let $\cX$ and $\cX'$ denote two continuous processes admitting local times. Suppose that both satisfy Assumption~\ref{assump_H}. Then, $\cX \overset{d}{=} \cX'$.
\end{proposition}

\begin{proof}
	Set $A_\cX(t):=\int_0^t\arctan(\cX_s)\,\dd s$. For every $(a,r)\in\cD$, the occupation formula gives
	\begin{align*}
		\tau_a^r(\cX)&=\int_\r\cL_\cX(\tau_a^r(\cX),x)\,\dd x,\\
		A_\cX(\tau_a^r(\cX))&=\int_\r\arctan(x)\cL_\cX(\tau_a^r(\cX),x)\,\dd x.
	\end{align*}
	Thus Assumption~\ref{assump_H} determines the joint law of $\bigl(\tau_a^r(\cX),A_\cX(\tau_a^r(\cX))\bigr)_{(a,r)\in\cD}$. Moreover, these inverse times have the same joint law as $(T(a,r))_{(a,r)\in\cD}$ and hence are almost surely dense by Lemma~\ref{lm:T_sigl1}. Continuous extension therefore determines $A_\cX$ measurably from this family. Finally, continuity of $\cX$ gives
	$$
	\arctan(\cX_t)=\lim_{m\to\infty}m\bigl(A_\cX(t+m^{-1})-A_\cX(t)\bigr),\qquad t\ge0.
	$$
	Since $\arctan$ is injective, this recovers $\cX$ measurably from $A_\cX$. Its law is therefore determined by Assumption~\ref{assump_H}, proving the claim.
\end{proof}

\subsection{Markov property of $(X_t, L_t(\cdot))$}

As in~\cite[Section~5.1]{Toth98}, we deduce the Markov property of $(X_t,L_t(\cdot))$ from the spatial Markov property of the flow.

\begin{theorem}[Markov property]
	For every fixed $t\ge0$, the $\sigma$-fields $\sigma(X_s:0\le s\le t)$ and $\sigma(X_s:s\ge t)$ are conditionally independent given $(X_t,L_t(\cdot))$. In particular, $(X_t,L_t(\cdot))_{t\ge0}$ is a Markov process.
\end{theorem}

\begin{proof}
	Fix $(a_0,r_0)\in\h^*$, and set $\lambda=\overline{\cS}_{r_0,\cdot}(a_0)$ and $\tau=T(a_0,r_0)$. By Proposition~\ref{prop:ray-knight} and the construction of $X$,
	$$
	(X_\tau,L_\tau)=(r_0,\lambda),\qquad a_0=\lambda(r_0),\qquad \tau=\int_\r\lambda(x)\,\dd x.
	$$
	Recall the index sets $\Lambda_-$ and $\Lambda_+$ from~\eqref{def:lamb}. Their almost-sure partition of $\h\setminus\{(a_0,r_0)\}$, together with~\eqref{eq:area_order}, gives
	$$
	\Lambda_-=\{(a,r):T(a,r)<\tau\},\qquad \Lambda_+=\{(a,r):T(a,r)>\tau\}.
	$$
	Thus, for $0\le s<\tau<u$, the construction~\eqref{eq:def_P_t} yields
	\begin{align*}
		\{(H_s,X_s)\}&=\bigcap_{\varepsilon>0}\overline{\{(a,r)\in\Lambda_-:|T(a,r)-s|<\varepsilon\}},\\
		\{(H_u,X_u)\}&=\bigcap_{\varepsilon>0}\overline{\{(a,r)\in\Lambda_+:|T(a,r)-u|<\varepsilon\}}.
	\end{align*}
	Lemma~\ref{lm:T_sigl1} and continuity imply that the stopped past $(X_{s\wedge\tau})_{s\ge0}$ and the shifted future $(X_{\tau+s})_{s\ge0}$ are measurable functions of $(\lambda,\overline{\cS}_-)$ and $(\lambda,\overline{\cS}_+)$, respectively. By Proposition~\ref{prop:spatial_markov}, they are conditionally independent given $\lambda$.
	
	To pass to deterministic times, note that~\eqref{eq:area_order} gives $|\{T<v\}|=v$ for $v\in T(\h)$, and Lemma~\ref{lm:T_sigl1} extends this identity to every $v\ge0$. Thus, for every nonnegative Borel function $f$ on $\r_+$,
	\begin{align}\label{eq:area_iden}
		\int_{\h^*} f(T(a,r))\,\dd a\dd r=\int_0^\infty f(t)\dd t.
	\end{align}
	Let $F,G$ be bounded Borel functions on $C([0,\infty))$, and put
	$$
	F_t:=F\bigl((X_{s\wedge t})_{s\ge0}\bigr),
	\qquad
	G_t:=G\bigl((X_{t+s})_{s\ge0}\bigr).
	$$
	Since $a_0=\lambda(r_0)$, we may choose jointly measurable functions $\varphi,\psi$ such that
	$$
	\e[F_\tau\mid\lambda]=\varphi(r_0,\lambda),
	\qquad
	\e[G_\tau\mid\lambda]=\psi(r_0,\lambda).
	$$
	For $b>0$ and any bounded Borel function $h$,
	\begin{align*}
		&\int_0^b \e\bigl[h(t,X_t,L_t)F_tG_t\bigr]\,\dd t\\
		&=\int_{\h^*}\e\bigl[\ind_{\{\tau<b\}}h(\tau,r_0,\lambda)F_\tau G_\tau\bigr]\,\dd a_0\dd r_0\\
		&=\int_{\h^*}\e\bigl[\ind_{\{\tau<b\}}h(\tau,r_0,\lambda)\varphi(r_0,\lambda)\psi(r_0,\lambda)\bigr]\,\dd a_0\dd r_0\\
		&=\int_0^b \e\bigl[h(t,X_t,L_t)\varphi(X_t,L_t)\psi(X_t,L_t)\bigr]\,\dd t.
	\end{align*}
	The first and last equalities follow from~\eqref{eq:area_iden}. The middle equality uses the conditional independence proved above, since $\tau=\int_\r\lambda(x)\,\dd x$ is determined by $\lambda$. It follows that
	$$
	\e[F_tG_t\mid X_t,L_t]
	=\e[F_t\mid X_t,L_t]\,\e[G_t\mid X_t,L_t]
	$$
	for Lebesgue-almost every $t>0$. 
	Scaling and the occupation formula yield
	$$
	\bigl(X_{c^2s},L_{c^2s}(\cdot)\bigr)_{s\ge0}
	\overset{d}{=}
	\bigl(cX_s,cL_s(\cdot/c)\bigr)_{s\ge0},
	\qquad c>0.
	$$
	Rescaling from any one such positive time gives the result for every $t>0$.
\end{proof}

\subsection{Trajectory properties}\label{subsec:pro_more}

We now describe the local extrema and points of increase and decrease of the PSRM.

\begin{proposition}
	Let $(X_t)_{t\ge0}$ be the PSRM. Then the following statements hold.
	\begin{enumerate}[label=\textup{(\roman*)}]
		\item Almost surely, $X$ has no flat plateaus, that is, $X$ is not constant on any non-trivial interval.
		\item For every fixed $0\le s<t<\infty$, almost surely, the maximum and minimum of $X$ on $(s,t)$ are each attained at a unique time. Moreover, almost surely, the sets of local maximum and local minimum points are countable.
		\item Almost surely, $X$ has no points of increase (resp. decrease) at which $L_t(X_t)>0$, that is, there do not exist $t,\varepsilon>0$ such that $L_t(X_t)>0$, $X_s\le X_t$ (resp. $X_s\ge X_t$) for all $s\in(t-\varepsilon,t)$, and $X_u\ge X_t$ (resp. $X_u\le X_t$) for all $u\in(t,t+\varepsilon)$.
	\end{enumerate}
\end{proposition}

\begin{proof}
    \begin{enumerate}[label=(\roman*)]
        \item Suppose for contradiction that $X \equiv r$ on some interval $(s, t)$ with $0\le s < t$. Taking $h = \mathbbm{1}_{\{r\}}$ in~\eqref{eq:occu} yields 
        $$t - s = \int_s^t \mathbbm{1}_{\{r\}}(X_u) \dd u = \int_\r \mathbbm{1}_{\{r\}}(x) (L_t(x) - L_s(x)) \dd x = 0,$$
        which contradicts $t > s$.
        
        \item Fix an interval $(s,t)$. Since $T(\cdot,r)$ is increasing,~\eqref{eq:area_iden} and $(H_{T(a,r)},X_{T(a,r)})=(a,r)$ imply that, almost surely, for Lebesgue-almost every $u>0$,
        $$
        H_u>0,\qquad T^-(H_u,X_u)=u=T(H_u,X_u).
        $$
        Then $L_{u-\delta}(X_u)<H_u<L_{u+\delta}(X_u)$ for $0<\delta<u$. If $X$ stayed on one side of $X_u$ in a left or right neighborhood of $u$, the occupation formula and spatial continuity of $L$ would force the corresponding local-time increment to vanish. Thus $X$ takes values both above and below $X_u$ arbitrarily close to $u$ on either side. By Fubini, this holds almost surely for almost every deterministic $u>0$. Since scaling makes the failure probability the same for all $u>0$, the property holds almost surely at each fixed $u>0$. At zero, $T(a,0)\downarrow0$ as $a\downarrow0$ implies $L_\delta(0)>0$ for every $\delta>0$, and the same argument gives right-hand oscillation. Hence neither endpoint is an extremum on $[s,t]$, and continuity ensures that both extrema are attained in $(s,t)$.
        
        Let $x_0:=\max_{u\in(s,t)}X_u$. Suppose, for contradiction, that this maximum is attained at two distinct times $u_1<u_2$ in $(s,t)$. Since $X$ has no flat plateaus, there exist times $v_1\in(s,u_1)$, $v_2\in(u_1,u_2)$ and $v_3\in(u_2,t)$ such that $X_{v_i}<x_0$ for $i=1,2,3$. By continuity of $X$ and Lemma~\ref{lm:T_sigl1}, we may take $v_i=T(a_i,r_i)$ for some $(a_i,r_i)\in\h$, $i=1,2,3$. By Proposition~\ref{prop:ray-knight}(i), $L_{v_i}(\cdot)=\overline{\cS}_{r_i,\cdot}(a_i)$. Since $X_u\le x_0$ on $(s,t)$, applying~\eqref{eq:occu} over $[v_1,v_3]$ gives $\int_{x_0}^\infty (L_{v_3}(x)-L_{v_1}(x))\,\dd x=0$.
        Monotonicity in time and continuity in space imply that $L_{v_1}(x)=L_{v_2}(x)=L_{v_3}(x)$ for all $x\ge x_0$. Hence the three PSR flow lines $\cS_{r_i,\cdot}(a_i)$, $i=1,2,3$, coalesce by $x_0$. We claim that they first coalesce at $x_0$.
        
        Suppose that, for some $\varepsilon>0$ with $x_0-\varepsilon>r_1\vee r_2$, $\cS_{r_1,x_0-\varepsilon}(a_1)=\cS_{r_2,x_0-\varepsilon}(a_2)$. By Lemma~\ref{lm:coal}, $L_{v_1}(x)=L_{v_2}(x)$ for all $x\in[x_0-\varepsilon,x_0]$. Applying~\eqref{eq:occu} over $[v_1,v_2]$ with $h=\mathbbm{1}_{[x_0-\varepsilon,x_0]}$ gives
        $$
        \int_{v_1}^{v_2}\mathbbm{1}_{[x_0-\varepsilon,x_0]}(X_u)\,\dd u=0.
        $$
        However, $X_{u_1}=x_0$ and $u_1\in(v_1,v_2)$, so continuity and the bound $X\le x_0$ imply that $X$ spends a positive amount of time in $[x_0-\varepsilon,x_0]$, a contradiction. Applying the same argument to $v_2$ and $v_3$, we obtain
        $$
        \cS_{r_1,x}(a_1)<\cS_{r_2,x}(a_2)<\cS_{r_3,x}(a_3),
        \qquad r_1\vee r_2\vee r_3<x<x_0.
        $$
        Their common value at $x_0$ contradicts Proposition~\ref{prop:no_31pts}, proving uniqueness. The minimum is treated by symmetry. Finally, every local maximum is a maximum on some interval with rational endpoints containing it. Applying uniqueness simultaneously to these countably many intervals proves countability of the local maximum points, and likewise for local minima.

        \item By symmetry, it suffices to consider the increasing case. Suppose that $t>0$ and $0<\varepsilon<t$ satisfy $L_t(X_t)>0$, $X_s\le X_t$ for $s\in(t-\varepsilon,t)$ and $X_u\ge X_t$ for $u\in(t,t+\varepsilon)$. By (i), continuity and Lemma~\ref{lm:T_sigl1}, we can choose $v_i=T(a_i,r_i)$, $i=1,2$, such that
        $$
        t-\varepsilon<v_1<t<v_2<t+\varepsilon,\qquad r_1<X_t<r_2.
        $$
        Since $X_s\le X_t$ on $[v_1,t]$ and $X_s\ge X_t$ on $[t,v_2]$, the occupation formula~\eqref{eq:occu} gives
        $$
        \int_{X_t}^\infty \bigl(L_t(x)-L_{v_1}(x)\bigr)\dd x=0,\qquad
        \int_{-\infty}^{X_t}\bigl(L_{v_2}(x)-L_t(x)\bigr)\dd x=0.
        $$
        Since $L_t(\cdot)$ is continuous and $L_\cdot(x)$ is non-decreasing (see Remark~\ref{rmk:def_L}), we obtain
        $$L_{T(a_1,r_1)}(X_t)=L_t(X_t)=H_t=L_{T(a_2,r_2)}(X_t).$$
        By Proposition~\ref{prop:ray-knight}(i),
        $$
        H_t=\cS_{r_1,X_t}(a_1)=\cS^*_{r_2,X_t}(a_2).
        $$
        Thus $(H_t,X_t)\in\h^*$ has both a forward and a backward incoming line, contradicting Proposition~\ref{prop:special_points}.
    \end{enumerate}
\end{proof}

\section{Decomposition and Markov property}
\label{sec:deco}
In this section, we first establish a decomposition of the PSR flow along a reference flow line (Proposition~\ref{p:difference}) and use it to prove the spatial Markov property of the two-sided flow (Proposition~\ref{prop:spatial_markov}). We then prove the absence of triple coalescence (Proposition~\ref{prop:no_31pts}). 

Let $\cW$ be the white noise introduced in Section~\ref{sec:PSR_flow}, and let $\cW^*$ be its image under $(a,r)\mapsto(a,-r)$. For any interval $I\subset\r$, let $\cF_I$ and $\cF_I^*$ be the $\sigma$-fields generated by the restrictions of $\cW$ and $\cW^*$ to $\r_+\times I$, respectively. If $I\cap(-J)=\varnothing$, then $\cF_I$ and $\cF_J^*$ are independent. We use the natural filtrations $\cF=(\cF_x)_{x\in\r}$ and $\cF^*=(\cF_x^*)_{x\in\r}$, where $\cF_x:=\cF_{(-\infty,x]}$ and $\cF_x^*:=\cF^*_{(-\infty,x]}$, with the usual augmentation. Throughout, let $(\cS,\cS^*)$ be the double PSR flow. 

Fix $(a_0,r_0)\in\h$ and define the non-negative $\cF$-predictable process
\begin{align}\label{def:eq_S}
	S_x :=
	\begin{cases}
		0, & x<r_0,\\
		\cS_{r_0,x}(a_0), & x\ge r_0,
	\end{cases}
\end{align}
and the non-negative $\cF^*$-predictable process
\begin{align*}
	Y_x :=
	\begin{cases}
		0, & x<-r_0,\\
		\cS^*_{r_0,-x}(a_0), & x\ge-r_0.
	\end{cases}
\end{align*}
Set $\tau^S:=\tau_{(a_0,r_0)}$ and $\tau^Y:=\sup\bigl\{x\le r_0\wedge0:\,\cS^*_{r_0,x}(a_0)=0\bigr\}$. On their respective time domains, $S$ and $Y$ satisfy \eqref{SDE:for} and \eqref{SDE:back} up to absorption at $\tau^S$ and $-\tau^Y$, respectively, and remain zero thereafter.

\subsection{Decomposition of the PSR flow along a flow line}

Following \cite[Chapter~2]{Walsh1986AnIT}, we define the martingale measures $\cW_S^-$ and $\cW_S^+$ describing the noise below and above $S$. Let
$$
f_S^-(s,u):=
\begin{cases}
	\sqrt{2\alpha+1}\,(s/S_u)^\alpha\ind_{[0,S_u]}(s), & S_u>0,\\
	0, & S_u=0.
\end{cases}
$$
For any $x\le y$ and Borel set $A\subset\r_+$ of finite Lebesgue measure, set
\begin{align*}
	(\cW_S^-)_{x,y}(A)
	&:=\int_{A\times[x,y]} f_S^-(s,u)\,\cW(\dd s,\dd u),\\
	(\cW_S^+)_{x,y}(A)
	&:=\int_{\r_+\times[x,y]}\ind_A(s-S_u)\,\cW(\dd s,\dd u).
\end{align*}
Since $f_S^-$ is bounded and predictable, $\cW_S^-=f_S^-\cdot\cW$ is an orthogonal martingale measure. The quadratic covariations of $\cW_S^+$, together with the L\'evy characterization, show that $\cW_S^+$ is a white noise on $\h$ with respect to $\cF$.

For any $\cF$-predictable $h$ satisfying $\e[\int_{\r_+\times[x,y]}h^2(s,u)f_S^-(s,u)^2\,\dd s\dd u]<\infty$, we have
\begin{equation}\label{eq:fW}
	\int_{\r_+\times[x,y]}h(s,u)\cW_S^-(\dd s,\dd u)=\int_{\r_+\times[x,y]}h(s,u)f_S^-(s,u)\cW(\dd s,\dd u).
\end{equation}
Similarly, for any $\cF$-predictable $h$ satisfying $\e[\int_{\r_+\times[x,y]}h^2(s,u)\,\dd s\,\dd u]<\infty$,
\begin{equation}\label{eq:thetaSW}
	\int_{\r_+\times[x,y]}h(s,u)\,\cW_S^+(\dd s,\dd u)=\int_x^y\int_{S_u}^{\infty}h(s-S_u,u)\,\cW(\dd s,\dd u).
\end{equation}

For later use, we also employ the following shorthand notation for $\cW^-_S$ and $\cW^+_S$:
\begin{align}
 \label{eq:W-}   \cW^-_S(\dd s, \dd u) &:= \tfrac{\sqrt{2\alpha+1}}{S_u^\alpha} s^\alpha \cW(\dd s, \dd u), && s\le S_u, u\ge r_0, S_u>0, \\
\label{eq:W+}    \cW^+_S(\dd s, \dd u) &:= \cW(S_u+\dd s, \dd u), && s\ge 0, u\in \r.
\end{align}

\begin{proposition}\label{p:independence_W-&W^+}
The martingale measures $\cW_S^-$ and $\cW_S^+$, defined via \eqref{eq:W-} and \eqref{eq:W+}, are independent. In particular, $\cW_S^+$ is a white noise independent of the process $S$.
\end{proposition}

\begin{proof}
	Let $\cW_1$ and $\cW_2$ be independent white noises on $\h$. Let $Z$ be the unique strong solution to \eqref{SDE:for}, driven by $\cW_1$ and started from $a_0$ at $r_0$, absorbed at $\tau^Z:=\inf\{x\ge r_0\vee0:\,Z_x=0\}$, and extended by $Z_x=0$ for $x<r_0$. Define $\cW_1^-$ as in \eqref{eq:W-}, with $(\cW,S)$ replaced by $(\cW_1,Z)$. Since $\cW_1^-$ is supported on $\{(s,u):0\le s\le Z_u\}$, \eqref{eq:fW} gives
	\begin{equation}\label{sde:Y_1}
		Z_x=a_0+\frac{2}{\sqrt{2\alpha+1}}\int_{r_0}^x\cW_1^-(\r_+,\dd u)+\frac{x-r_0}{2\alpha+1},	\qquad r_0\le x\le\tau^Z.
	\end{equation}
	In particular, $Z$ is measurable with respect to $\cW_1^-$.
	
	For deterministic $g\in L^2(\h)$, define
	\begin{align}\label{eq:W'}
		\cWhat(g):=\frac{1}{\sqrt{2\alpha+1}}\iint_{\h}Z_u^\alpha s^{-\alpha}g(s,u)\cW_1^-(\dd s,\dd u)+\iint_{\h}g(s+Z_u,u)\cW_2(\dd s,\dd u),
	\end{align}
	where $Z_u^\alpha s^{-\alpha}$ is set to zero at $s=0$. A direct computation of the quadratic variation shows that $\cWhat$ is a white noise with respect to the joint filtration of $\cW_1$ and $\cW_2$. By density of simple functions, identity~\eqref{eq:W'} extends to predictable processes under suitable integrability assumptions.
	
	Let $\cWhat_Z^-$ and $\cWhat_Z^+$ be defined as in \eqref{eq:W-} and \eqref{eq:W+}, with $(\cW,S)$ replaced by $(\cWhat,Z)$. Applying \eqref{eq:fW} and \eqref{eq:thetaSW} to the two terms in \eqref{eq:W'} yields
	\begin{equation}\label{eq:W'S}
		\cWhat_Z^-=\cW_1^-,\qquad \cWhat_Z^+=\cW_2.
	\end{equation}
	Moreover, replacing $\cW_1$ by $\cWhat$ leaves the stochastic integral in the equation for $Z$ unchanged. Hence $Z$ also solves \eqref{SDE:for} driven by $\cWhat$, with the same initial condition and absorption rule. Thus, $(\cWhat,Z)\overset{\mathrm d}{=}(\cW,S)$. Together with \eqref{eq:W'S}, this gives
	$$
	(\cW_S^-,\cW_S^+,S)
	\overset{\mathrm d}{=}(\cW_1^-,\cW_2,Z).
	$$
	Since $(\cW_1^-,Z)$ is measurable with respect to $\cW_1$ and $\cW_1$ is independent of $\cW_2$, the required independence follows. The recovery formula \eqref{sde:Y_1}, applied to $(S,\cW_S^-)$, also shows that $S$ is measurable with respect to $\cW_S^-$.
\end{proof}

From the above arguments, the process $S$ is driven by $\cW_S^-$ and satisfies
\begin{align}\label{sde:S}
S_x=a_0+\frac{2}{\sqrt{2\alpha+1}} \int_{r_0}^x \cW_S^-([0,S_u], \dd u)+\frac{x-r_0}{2\alpha+1},\qquad  r_0\le x\le \tau^S.
\end{align}

With the construction and properties of the martingale measures in place, we are now ready to establish the decomposition theorem along the process $S=(S_x,\, x\in \r)$ defined in \eqref{def:eq_S}, which is driven by $\cW$.

\begin{proposition}\label{p:difference}
    Fix $(a_0,r_0)\in \h$. Let $\cS$ be the PSR flow defined in Definition~\ref{def:PSR}, driven by $\cW$. Consider the families of processes 
    $$\cS^+:=\{(\cS^+_{r,x}(a),\, x\ge r)\}_{a\ge 0,\, r\ge r_0} \quad\text{and}\quad 
    \cS^-:=\{(\cS^-_{r,x}(a),\, x\ge r)\}_{0\le a< S_r,\, r\ge r_0}$$
    defined by
    \begin{align*}
    \cS^+_{r,x}(a) &:=\cS_{r,x}(a+ S_r)- S_x, &&  a\ge 0,\,  x\ge r, \\
    \cS^-_{r,x}(a) &:=  \cS_{r,x}(a), && a< S_r,\, x\ge r.
    \end{align*}
    Then the following statements hold.
    \begin{enumerate}[label=\textup{(\roman*)}]
    \item The family $\cS^+$ is measurable with respect to $(\cW_S^+, S)$. In particular, conditionally on $S$, $\cS^+$ is independent of $\cW_S^-$.
    
    \item Both $S$ and $\cS^-$ are measurable with respect to $\cW_S^-$. Consequently, they are independent of $\cW_S^+$.
    \end{enumerate}
\end{proposition}

\begin{proof}
	For a non-negative continuous function $g$, write
	$$
	\psi_g(z,u):=\bigl[(\alpha+1)z+g_u^{\alpha+1}\bigr]^{1/(\alpha+1)},
	\qquad z\ge0,
	$$
	where $u$ belongs to the domain of $g$.
	
	\begin{enumerate}[label=\textup{(\roman*)}]
		\item Fix $r\ge r_0$ and $a>0$. Let $Z_x:=\cS_{r,x}(a+S_r)$, and set $F_x:=(\alpha+1)^{-1}(Z_x^{\alpha+1}-S_x^{\alpha+1})$.
		By It\^o's formula and~\eqref{eq:thetaSW},
		\begin{align}\label{eq:cS_+_SDE_F}
			F_x=F_r+2\int_r^x\!\int_0^{\psi_S(F_u,u)-S_u}(s+S_u)^\alpha\cW_S^+(\dd s,\dd u)+\int_r^x(\psi_S(F_u,u)^\alpha-S_u^\alpha)\dd u.
		\end{align}
		Here $F$ is absorbed at its first zero by Lemma~\ref{lm:coal}.
		
		Fix a path $g$. On $\{z,z'\ge\varepsilon,\ \psi_g(z,u),\psi_g(z',u)\le K\}$, the drift in~\eqref{eq:cS_+_SDE_F} is Lipschitz in $z$ with constant $\alpha((\alpha+1)\varepsilon)^{-1/(\alpha+1)}$, while the squared $L^2(\dd s)$-distance between the diffusion coefficients is bounded by
		$$
		\frac{4}{2\alpha+1}\bigl|\psi_g(z,u)^{2\alpha+1}-\psi_g(z',u)^{2\alpha+1}\bigr|\le 4K^\alpha|z-z'|.
		$$
		By~\eqref{eq:W'} and Proposition~\ref{p:independence_W-&W^+}, we may include $S$ in the initial $\sigma$-field while preserving the white-noise property of $\cW_S^+$ and the adaptedness of $F$. The Yamada--Watanabe argument gives pathwise uniqueness up to the first zero. Since $F$ is absorbed at $0$, it is a strong solution with respect to $(S,\cW_S^+)$. Hence $\cS^+_{r,x}(a)=\psi_S(F_x,x)-S_x$ is measurable with respect to $(S,\cW_S^+)$.
		
		We claim that, almost surely, for all $a\ge0$ and $x\ge r\ge r_0$,
		$$
		\cS^+_{r,x}(a)
		=\inf\left\{\cS^+_{r_n,x}(a_n):r_0\le r_n\le r,\ \cS^+_{r_n,r}(a_n)>a,\ a_n\in\bD_+,\ r_n\in\bD\cup\{r_0\}\right\}.
		$$
		This follows from the argument for~\eqref{eq:cS_I_+}, using Lemma~\ref{lm:basci_prop}(i), continuity and properties~\ref{item:R2}--\ref{item:R3}. Therefore, $\cS^+$ is measurable with respect to $(S,\cW_S^+)$, and the conditional independence follows from Proposition~\ref{p:independence_W-&W^+}.
		
		\item Let $\mathcal V$ be a white noise independent of $\cW$, and define $\cWhat$ by~\eqref{eq:W'} with $(Z,\cW_1^-,\cW_2)$ replaced by $(S,\cW_S^-,\mathcal V)$. Then $\cWhat$ is a white noise driving $S$ and agrees with $\cW$ below $S$. For fixed $r\ge r_0$ and $a\ge0$, the corresponding PSR flow line started from $a\wedge S_r$ stays below $S$ and therefore coincides with $\cS_{r,\cdot}(a\wedge S_r)$ by pathwise uniqueness. Given $\cW_S^-$, the two paths are independent, since one is a function of $(\cW_S^-,\mathcal V)$ and the other of $\cW$. Their equality implies that $\cS_{r,\cdot}(a\wedge S_r)$ is measurable with respect to $\cW_S^-$. Restricting to $\{a<S_r\}$ and applying the countable reconstruction argument of (i) to starting points below $S$ proves measurability of $\cS^-$. The independence of $(S,\cS^-)$ and $\cW_S^+$ follows from Proposition~\ref{p:independence_W-&W^+}.
	\end{enumerate}
\end{proof}

\subsection{Spatial Markov property of the flow}

Recall the two-sided PSR flow $\overline{\cS}$ defined by~\eqref{eq:cS_bar}. The spatial Markov property states that, conditionally on a curve in $\overline{\cS}$ started from a fixed point, the collections of curves ``above'' and ``below'' it are independent (see Figure~\ref{fig:decomp}).

Define $\cW_Y^{*,-}$ and $\cW_Y^{*,+}$ by the same integral definitions as $\cW_S^-$ and $\cW_S^+$, with $(S,\cW)$ replaced by $(Y,\cW^*)$. In particular, the lower weight is zero whenever $Y_u=0$. Set
$$
H_S:=\r_+\times[r_0,\infty),\qquad H_Y:=\r_+\times[-r_0,\infty),
$$
and define
$$
\mathscr W^\pm:=\sigma\!\left(\left.\cW_S^\pm\right|_{H_S}, \left.\cW_Y^{*,\pm}\right|_{H_Y}\right).
$$

\begin{proposition}\label{prop:ind_scr}
	The sigma-fields $\mathscr W^-$ and $\mathscr W^+$ are independent.
\end{proposition}

\begin{proof}
	Throughout this proof, $\cW_S^\pm$ and $\cW_Y^{*,\pm}$ denote the restrictions to $H_S$ and $H_Y$, respectively. The pairs $(\cW_S^-,\cW_S^+)$ and $(\cW_Y^{*,-},\cW_Y^{*,+})$ are measurable with respect to the independent $\sigma$-fields $\cF_{(r_0,\infty)}$ and $\cF^*_{(-r_0,\infty)}$, respectively, so the pairs are independent. Moreover, by Proposition~\ref{p:independence_W-&W^+}, $\cW^-_S$ and $\cW^+_S$ are independent. Similarly, $\cW^{*,-}_Y$ and $\cW^{*,+}_Y$ are independent. Therefore, for any measurable sets $A, B, C, D$,
	\begin{align*}
		&\p\big(\cW^-_S\in A,\cW^+_S\in B,
		\cW^{*,-}_Y\in C,\cW^{*,+}_Y\in D\big)\\
		&=\p\big(\cW^-_S\in A,\cW^+_S\in B\big)
		\p\big(\cW^{*,-}_Y\in C,\cW^{*,+}_Y\in D\big)\\
		&=\p(\cW^-_S\in A)\p(\cW^+_S\in B)
		\p(\cW^{*,-}_Y\in C)\p(\cW^{*,+}_Y\in D)\\
		&=\p(\cW^-_S\in A)\p(\cW^{*,-}_Y\in C)
		\p(\cW^+_S\in B)\p(\cW^{*,+}_Y\in D)\\
		&=\p\big(\cW^-_S\in A,\cW^{*,-}_Y\in C\big)
		\p\big(\cW^+_S\in B,\cW^{*,+}_Y\in D\big).
	\end{align*}
	This proves the independence of $\mathscr W^-$ and $\mathscr W^+$.
\end{proof}

\begin{figure}[htbp]
	\centering
	\scalebox{1}{ 
		\def\svgwidth{0.5\columnwidth}
\begingroup%
  \makeatletter%
  \providecommand\color[2][]{%
    \errmessage{(Inkscape) Color is used for the text in Inkscape, but the package 'color.sty' is not loaded}%
    \renewcommand\color[2][]{}%
  }%
  \providecommand\transparent[1]{%
    \errmessage{(Inkscape) Transparency is used (non-zero) for the text in Inkscape, but the package 'transparent.sty' is not loaded}%
    \renewcommand\transparent[1]{}%
  }%
  \providecommand\rotatebox[2]{#2}%
  \newcommand*\fsize{\dimexpr\f@size pt\relax}%
  \newcommand*\lineheight[1]{\fontsize{\fsize}{#1\fsize}\selectfont}%
  \ifx\svgwidth\undefined%
    \setlength{\unitlength}{203.20415497bp}%
    \ifx\svgscale\undefined%
      \relax%
    \else%
      \setlength{\unitlength}{\unitlength * \real{\svgscale}}%
    \fi%
  \else%
    \setlength{\unitlength}{\svgwidth}%
  \fi%
  \global\let\svgwidth\undefined%
  \global\let\svgscale\undefined%
  \makeatother%
  \begin{picture}(1,1.7030837)%
    \lineheight{1}%
    \setlength\tabcolsep{0pt}%
    \put(0,0){\includegraphics[width=\unitlength,page=1]{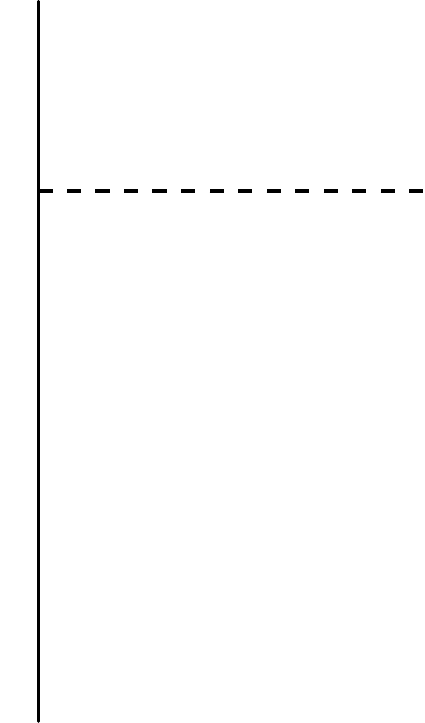}}%
    \put(0.02182134,1.22815191){\color[rgb]{0,0,0}\transparent{0.82599998}\makebox(0,0)[lt]{\lineheight{1.25}\smash{\begin{tabular}[t]{l}$0$\end{tabular}}}}%
    \put(0,0){\includegraphics[width=\unitlength,page=2]{figures/decomposition_proof.pdf}}%
    \put(-0.00032393,1.59749078){\color[rgb]{0,0,0}\transparent{0.82599998}\makebox(0,0)[lt]{\lineheight{1.25}\smash{\begin{tabular}[t]{l}$\tau^S$\end{tabular}}}}%
    \put(-0.00032393,0.35599589){\color[rgb]{0,0,0}\transparent{0.82599998}\makebox(0,0)[lt]{\lineheight{1.25}\smash{\begin{tabular}[t]{l}$\tau^Y$\end{tabular}}}}%
  \end{picture}%
\endgroup%

	}
	\caption{\textbf{Spatial decomposition.}
		Fix a point $(a_0,r_0)$, marked by a \textbf{\textcolor{blue}{blue dot}}. 
		From this point, the process $S$ evolves upward until absorption at zero at time $\tau^S$, while $Y$ evolves downward until absorption at time $\tau^Y$. The black solid curve represents the PSR flow line $(\overline\cS_{r_0,x}(a_0),\, x\in \r)$, which partitions the half-plane $\h$ into two disjoint regions: a gray-shaded region (left) and a white region (right). These regions correspond to $\mathscr{W}^-$ and $\mathscr{W}^+$, respectively. 
		Any trajectory initiated in one region remains confined to that region. 
		Sample trajectories $(\overline\cS_{r,x}(a),\, x\in \r)$, initiated from points $(a,r)$ (indicated by black dots), are shown as blue curves.}
	\label{fig:decomp}
\end{figure}

We are now ready to prove the spatial Markov property.

\begin{proof}[Proof of Proposition~\ref{prop:spatial_markov}]
	Write $\lambda(x):=\overline{\cS}_{r_0,x}(a_0)$ and $\mathscr G^+:=\sigma(\lambda)\vee\mathscr W^+$. By \eqref{sde:S} and its backward counterpart, $\lambda$ is $\mathscr W^-$-measurable. By Proposition~\ref{prop:ind_scr}, it suffices to show that $\overline{\cS}_-$ is $\mathscr W^-$-measurable and $\overline{\cS}_+$ is $\mathscr G^+$-measurable. Note that Proposition~\ref{p:difference} remains valid when the noises are restricted to $H_S$ and $H_Y$.
	
	Fix $(a,r)\in\h$ with $r>r_0$, and write $\gamma(x):=\overline{\cS}_{r,x}(a)$. By Corollary~\ref{cor:dual}, $(a,r)\in\Lambda_+$ if and only if $a\ge\lambda(r)$. On this event, Proposition~\ref{p:difference}(i) shows that $(\gamma(x),\,x\ge r)$ is $\mathscr G^+$-measurable. For $r_0\le x\le r$, since $c<\lambda(x)$ implies $\cS_{x,r}(c)\le\lambda(r)\le a$, we have
	$$
	\gamma(x)=\inf\{c\ge\lambda(x):\cS_{x,r}(c)>a\}.
	$$
	Thus this segment, and in particular $b:=\gamma(r_0)$, is $\mathscr G^+$-measurable. By~\eqref{eq:flow_prop} for $\cS^*$, $\gamma(x)=\cS^*_{r_0,x}(b)$ for $x<r_0$, so $\gamma(x)$ is also $\mathscr G^+$-measurable for $x<r_0$.
	
	Consider $a<\lambda(r)$. Proposition~\ref{p:difference}(ii) gives the $\mathscr W^-$-measurability of $(\gamma(x),\,x\ge r)$. For $r_0\le x\le r$, the same duality formula yields
	$$
	\gamma(x)=\inf\left(\{0\le c<\lambda(x):\cS_{x,r}(c)>a\}\cup\{\lambda(x)\}\right).
	$$
	Hence $b:=\gamma(r_0)\le a_0$ is $\mathscr W^-$-measurable. By the definition of $\Lambda_-$, if $(a,r)\in\Lambda_-$, then either $a_0>0$ and $b<a_0$, or $a_0=0$ and $b=0$. By the flow property, $\gamma(x)=\cS^*_{r_0,x}(b)$ for $x<r_0$.
	When $a_0>0$, the backward version of Proposition~\ref{p:difference}(ii) shows that this segment is $\mathscr W^-$-measurable. When $a_0=0$, we have $b=0$, so $\gamma(x)=\lambda(x)$ for $x<r_0$, which is also $\mathscr W^-$-measurable. 
	The case $r<r_0$ follows by reversing the two directions, and $r=r_0$ follows directly from the one-sided decompositions. Taking the infima over rational heights and using right-continuity construction, we conclude that $\overline{\cS}_-$ is $\mathscr W^-$-measurable and $\overline{\cS}_+$ is $\mathscr G^+$-measurable.
\end{proof}

\begin{proof}[Proof of Proposition~\ref{prop:no_31pts}]
	Suppose, by contradiction, that such a point $(a,x)$ and three flow lines $\cS^i$ exist, with $\cS^i$ distinct on $[x-\varepsilon,x)$ for some $\varepsilon>0$. By Proposition~\ref{prop:skeleton}, after decreasing $\varepsilon$ if necessary, the starting points may be taken in $\cD$ without changing the paths on $[x-\varepsilon,x)$. By the flow property~\eqref{eq:flow_prop} and Corollary~\ref{p:properties}(i), we may further assume that $r_1=r_2=r_3=:r\in\bD$ and $0<a_1<a_2<a_3$ with $a_i\in\bD$. Lemma~\ref{lm:coal} implies that $\cS^i(y)$, $i=1,2,3$, are almost surely pairwise distinct on $[r,x)$. Define 
	$$
	\tau^{i j }:= \inf\bigl\{y>r:\, \cS^i(y)=\cS^j(y)\bigr\}\quad \text{ for } i\neq j\in \{1,2,3\}.
	$$
	On this event, $x=\tau^{1 2}=\tau^{2 3}<\infty$. To prove Proposition~\ref{prop:no_31pts}, it therefore suffices to show that for any fixed $r\in\bD$ and $0<a_1<a_2<a_3$ with $a_i\in\bD$, 
	\begin{align}\label{eq:de_tau}
		\p(\tau^{12}=\tau^{23}<\infty)=0.
	\end{align}
	
	On $\{\tau^{12}=\tau^{23}<\infty,\, \cS^2(\tau^{23})=0\}$, define
	$$
	d_i:=\inf\{y\ge r:\cS^i(y)=0\},\qquad i=2,3.
	$$
	The three paths are strictly ordered before $\tau^{23}$, hence $d_2=d_3=\tau^{23}$. However, Corollary~\ref{cor:hit_boun_fi}, applied to $\cS^2$ and $\cS^3$ up to the first zero of $\cS^3$, gives $d_2\wedge\tau^{23}<d_3$ almost surely, a contradiction.
	
	It remains to exclude coalescence at $a>0$. 
	Fix a dyadic time $q>r$ and rational numbers $0<\ell<K$ and $T>q$. On $A:=\{\ell<\cS^1(q)<\cS^2(q)<\cS^3(q)<K\}$, set
	$$
	\rho:=T\wedge\inf\bigl\{y\ge q:
	\cS^1(y)=\cS^2(y)\ \text{or}\
	\cS^2(y)=\cS^3(y)\ \text{or}\
	\cS^1(y)\le\ell\ \text{or}\
	\cS^3(y)\ge K\bigr\},
	$$
	and set $\rho=q$ on $A^c$. Thus, on $A$, the three paths are strictly ordered in $(\ell,K)$ throughout $[q,\rho)$.
	For $i=1,2$, set
	$$
	F_i(x):=\frac{(\cS^{i+1}_x)^{\alpha+1}-(\cS^i_x)^{\alpha+1}}{\alpha+1}.
	$$
	As in the proof of Proposition~\ref{p:difference}, It\^o's formula gives, before $\rho$,
	$$
	\dd F_i=\beta_iF_i\,\dd x+\dd M_i,
	\qquad
	\dd\langle M_i\rangle_x=c_iF_i\,\dd x,
	\qquad
	\langle M_1,M_2\rangle=0,
	$$
	where $\beta_i=(F_i)^{-1}((\cS^{i+1})^\alpha-(\cS^i)^\alpha)$ and $c_i=4(F_i)^{-1}\int_{\cS^i}^{\cS^{i+1}}s^{2\alpha}\dd s$.
	On $A$, throughout $[q,\rho)$, there are constants $c,C>0$, depending only on $\alpha,\ell,K$, such that
	$$
	c\le c_i\le C,
	\qquad
	|\beta_i|\le C,
	\qquad
	|c_2-c_1|\le C(F_1+F_2).
	$$
	Put $D:=F_1+F_2$ and $R:=F_1/D$. It\^o's formula yields
	\begin{gather*}
		\dd R=\dd N+R(1-R)\left(\beta_1-\beta_2+\frac{c_2-c_1}{D}\right)\dd x,\\
		\dd\langle N\rangle_x=\frac{R(1-R)}{D}\bigl((1-R)c_1+Rc_2\bigr)\,\dd x,
	\end{gather*}
	where $N$ is a continuous local martingale. Set $h(p):=-p\log p-(1-p)\log(1-p)$. Then, $0\le h\le\log2$ and $h''(p)=-1/[p(1-p)]$. By It\^o's formula,
	\begin{align*}
		\dd h(R)
		=h'(R)\,\dd N+\left[
		R(1-R)\log\frac{1-R}{R}
		\Big(\beta_1-\beta_2+\frac{c_2-c_1}{D}\Big)
		-\frac{(1-R)c_1+Rc_2}{2D}\right]\dd x.
	\end{align*}
	Since $p(1-p)\log(\frac{1-p}{p})$ is bounded on $(0,1)$ and
	$$
	\left|\beta_1-\beta_2+\frac{c_2-c_1}{D}\right|\le3C,
	\qquad
	(1-R)c_1+Rc_2\ge c,
	$$
	the drift of $h(R)$ is at most $C-c/(2D)$, after increasing $C$ if necessary. Set
	$$
	\rho_n:=\rho\wedge\inf\{x\ge q:F_1(x)\wedge F_2(x)\le n^{-1}\}.
	$$
	On $A$, before $\rho_n$, both $R$ and $1-R$ are bounded away from zero, since $D\le K^{\alpha+1}/(\alpha+1)$. Thus the stopped stochastic integral in the equation for $h(R)$ is square-integrable and has conditional mean zero given $\cF_q$. Integrating the drift bound and taking conditional expectations gives, on $A$,
	\begin{align*}
		\frac c2\,\e\!\left[\int_q^{\rho_n}\frac{\dd x}{D_x}
		\,\middle|\,\cF_q\right]\le h(R_q)-\e[h(R_{\rho_n})\mid\cF_q] +C\,\e[\rho_n-q\mid\cF_q]\le \log2+C(T-q).
	\end{align*}
	Here $h$ is extended continuously to $[0,1]$ by $h(0)=h(1)=0$. Since $F_1,F_2$ are continuous and strictly positive on $[q,\rho)$, we have $\rho_n\uparrow\rho$. Conditional monotone convergence therefore yields, on $A$,
	$$
	\e\!\left[\int_q^\rho\frac{\dd x}{D_x}\,\middle|\,\cF_q\right]
	\le\frac{2}{c}\bigl(\log2+C(T-q)\bigr)<\infty.
	$$
	It follows that $\int_q^\rho D_x^{-1}\dd x<\infty$ almost surely. On the other hand,
	$$
	\dd\log D=\frac{\dd M_1+\dd M_2}{D}+ \left(\frac{\beta_1F_1+\beta_2F_2}{D}-\frac{c_1F_1+c_2F_2}{2D^2}\right) \dd x.
	$$
	The quadratic variation of the local martingale
	$$
	L_x:=\int_q^x\frac{\dd M_1+\dd M_2}{D},\qquad q\le x<\rho,
	$$
	satisfies
	$$
	\langle L\rangle_{\rho-}=\int_q^\rho\frac{c_1F_1+c_2F_2}{D^2}\,\dd x\le C\int_q^\rho\frac{\dd x}{D_x}<\infty.
	$$
	Thus $L_x$ has a finite limit as $x\uparrow\rho$. Moreover,
	$$
	\int_q^\rho	\left|\frac{\beta_1F_1+\beta_2F_2}{D}-\frac{c_1F_1+c_2F_2}{2D^2}\right|\dd x\le C(T-q)+\frac C2\int_q^\rho\frac{\dd x}{D_x}<\infty.
	$$
	Consequently, $\log D_x$ has a finite limit as $x\uparrow\rho$. By continuity of $D$, this rules out $D(\rho)=0$, which would imply $\log D_x\to-\infty$. If $\tau^{12}=\tau^{23}<\infty$ at positive height, continuity allows us to choose a dyadic $q$ sufficiently close before this time and rational $\ell,K,T$ such that $A$ holds and $\rho=\tau^{12}=\tau^{23}$. This proves~\eqref{eq:de_tau}.
\end{proof}

\section{Space-time transformation}
\label{sec:perfect}
Motivated by Perkins' disintegration theorem~\cite{perkins}, its flow version in~\cite{elie23}, and related works cited therein, we study a space--time transformation of the PSR flow using the framework developed in~\cite{elie23}. The transformed flow is mainly used to establish the perfect flow property in Proposition~\ref{prop:perfect}.

\subsection{Transformation of the PSR flow}

Fix $\alpha>0$ and $(a_0,r_0)\in\h^*$. Let $f$ be the unique strong solution to
\begin{align}\label{SDE:for_f}
	f(x) = a_0 + \int_{r_0}^x \frac{2}{f(u)^{\alpha}}  \int_{\r_+} s^\alpha \mathbbm{1}_{\{s\le f(u)\}} \cW(\dd s, \dd u) + \frac{x-r_0}{2\alpha+1},
\end{align}
for $r_0\le x<\df_f^+$, where $\df_f^+:=\inf\{x\ge r_0:f(x)=0\}$, and absorb $f$ at $0$ at this time. The process $(2\alpha+1)f$ is a killed $\mathrm{BESQ}^1$ process, so $\df_f^+<\infty$ almost surely. Define
\begin{align*}
	\eta_{f}(x):=\int_{r_0}^{x} \frac{\dd u}{f(u)}, \qquad r_0\le x< \df^+_{f}.
\end{align*}
By time reversal, $\bigl((2\alpha+1)f(\df_f^+-u),\,0\le u\le\df_f^+-r_0\bigr)$ has the law of a $\mathrm{BESQ}_0^3$ process run up to its last passage at level $(2\alpha+1)a_0$. By \cite[Exercise~3.20, Chapter~X]{RevYor}, it follows that $\eta_f(\df_f^+):=\lim_{x\uparrow\df_f^+}\eta_f(x)=+\infty$ almost surely. Since $f(x)>0$ on $[r_0, \df^+_f)$, the map $\eta_f$ is continuous and strictly increasing. Thus, the transformation $\eta_f$ defines a bijection from $[r_0,\df^+_f)$ onto $\r_+$. 

\begin{definition}[Transformation of the PSR flow]\label{def:trans_flow}
	Let $\bI_+:=[0,1)\times\r_+$, and let $\cS$ be the PSR flow driven by $\cW$ in Definition~\ref{def:PSR}. For $(v,r)\in\bI_+$ and $x\ge r$, define
	\begin{align*}
		\cY^+_{r, x}(v)
		:= \frac{\cS_{\eta_{f}^{-1}(r), \eta_{f}^{-1}(x)}\left(v f \circ \eta_{f}^{-1}(r)\right)}{f \circ \eta_{f}^{-1}(x)}.
	\end{align*}
    By property~\ref{item:R3} of $\cS$, $\cY^+$ takes values in $[0,1]$, and by Lemma~\ref{lm:coal}, it is absorbed at $1$ upon its first hit. Moreover, $(f(x),\, r_0\le x\le \df^+_{f})=(\cS_{r_0,x}(a_0),\, r_0\le x\le \df^+_{f})$.
\end{definition}

\begin{definition}[Transformation of white noise]\label{def:trans_noise}
	Let $(\mathscr E_u)_{u\in\r}$ be the usual augmented filtration generated by $\cW$. For each $t\ge0$, $\eta_f^{-1}(t)$ is an $(\mathscr E_u)$-stopping time. Set
	$$
	\widetilde{\mathscr E}_t:=\mathscr E_{\eta_f^{-1}(t)},\qquad t\ge0,
	$$
	and define
	\begin{align}\label{eq:def_tild_cW}
		\widetilde{\cW}(g):=\int_{r_0}^{\df_f^+}\int_0^{f(u)}\frac{1}{f(u)}g\!\left(\frac{s}{f(u)},\eta_f(u)\right)\cW(\dd s,\dd u),\qquad g\in L^2([0,1]\times\r_+).
	\end{align}
	This gives a white noise on $[0,1]\times\r_+$ with respect to $(\widetilde{\mathscr E}_t)$. By the It\^o isometry and approximation with elementary predictable integrands, identity~\eqref{eq:def_tild_cW} extends to all $(\widetilde{\mathscr E}_t)$-predictable functions $g$ satisfying $\e\!\left[\int_0^\infty\int_0^1g(v,t)^2\,\dd v\dd t\right]<\infty$.
\end{definition}

\begin{figure}[htbp]
	\centering
	\scalebox{1}{ 
		\def\svgwidth{0.8\columnwidth}
\begingroup%
  \makeatletter%
  \providecommand\color[2][]{%
    \errmessage{(Inkscape) Color is used for the text in Inkscape, but the package 'color.sty' is not loaded}%
    \renewcommand\color[2][]{}%
  }%
  \providecommand\transparent[1]{%
    \errmessage{(Inkscape) Transparency is used (non-zero) for the text in Inkscape, but the package 'transparent.sty' is not loaded}%
    \renewcommand\transparent[1]{}%
  }%
  \providecommand\rotatebox[2]{#2}%
  \newcommand*\fsize{\dimexpr\f@size pt\relax}%
  \newcommand*\lineheight[1]{\fontsize{\fsize}{#1\fsize}\selectfont}%
  \ifx\svgwidth\undefined%
    \setlength{\unitlength}{338.22452545bp}%
    \ifx\svgscale\undefined%
      \relax%
    \else%
      \setlength{\unitlength}{\unitlength * \real{\svgscale}}%
    \fi%
  \else%
    \setlength{\unitlength}{\svgwidth}%
  \fi%
  \global\let\svgwidth\undefined%
  \global\let\svgscale\undefined%
  \makeatother%
  \begin{picture}(1,0.4883588)%
    \lineheight{1}%
    \setlength\tabcolsep{0pt}%
    \put(0,0){\includegraphics[width=\unitlength,page=1]{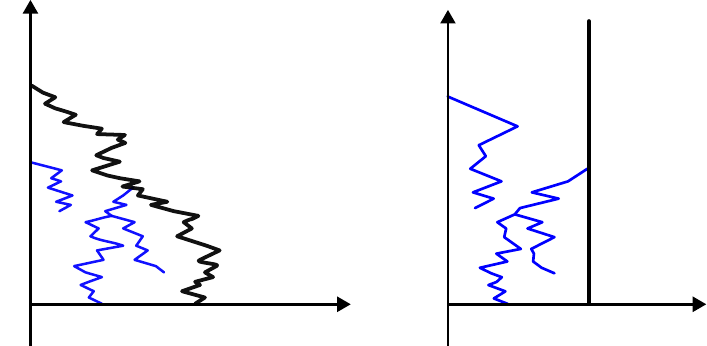}}%
    \put(0.44773258,0.45094967){\color[rgb]{0,0,0}\transparent{0.82599998}\makebox(0,0)[lt]{\lineheight{1.25}\smash{\begin{tabular}[t]{l}$\eta_{f}(\df_{f})=\infty$\end{tabular}}}}%
    \put(0.82547013,0.01810133){\color[rgb]{0,0,0}\transparent{0.82599998}\makebox(0,0)[lt]{\lineheight{1.25}\smash{\begin{tabular}[t]{l}$1$\end{tabular}}}}%
    \put(0.00878855,0.04803706){\color[rgb]{0,0,0}\transparent{0.82599998}\makebox(0,0)[lt]{\lineheight{1.25}\smash{\begin{tabular}[t]{l}$0$\end{tabular}}}}%
    \put(0.26585258,0.01810133){\color[rgb]{0,0,0}\transparent{0.82599998}\makebox(0,0)[lt]{\lineheight{1.25}\smash{\begin{tabular}[t]{l}$b$\end{tabular}}}}%
    \put(-0.00019459,0.36591864){\color[rgb]{0,0,0}\transparent{0.82599998}\makebox(0,0)[lt]{\lineheight{1.25}\smash{\begin{tabular}[t]{l}$\mathfrak{d}_{f}$\end{tabular}}}}%
    \put(0.60000264,0.04803706){\color[rgb]{0,0,0}\transparent{0.82599998}\makebox(0,0)[lt]{\lineheight{1.25}\smash{\begin{tabular}[t]{l}$0$\end{tabular}}}}%
  \end{picture}%
\endgroup%

	}
	\caption{\textbf{Transformation.}
		Left: the process $x\mapsto f(x)$ is shown in bold, started from $(a_0,r_0):=(b,0)$ and absorbed at time $\df^+_f$. The blue curves represent the PSR flow lines $\cS_{r,\cdot}(a)$ up to their first hitting time of $f$. 
		Right: under the space--time transformation, the curve $x\mapsto f(x)$ is mapped to the constant function $1$, and the corresponding flow lines become $\cY^+_{r,\cdot}(v)$.}
\end{figure}

\begin{theorem}\label{thm:perkin_disintegration}
	Let $\cY^+$ and $\widetilde{\cW}$ be as in Definitions~\ref{def:trans_flow} and~\ref{def:trans_noise}. For each $r\ge0$ and $v\in[0,1)$, the process $(\cY^+_{r,x}(v),\,x\ge r)$ is the unique strong solution, up to absorption, to
	\begin{align}\label{eq:SDE_perkin}
		\nonumber \cY^+_{r,x}(v)= v &+ 2\int_r^x \! \int_0^1 \left(\frac{s^\alpha}{\cY^+_{r,u}(v)^\alpha}\mathbbm{1}_{\{s\le \cY^+_{r,u}(v)\}}-\cY^+_{r,u}(v)s^\alpha \right) \widetilde{\cW}(\dd s, \dd u)\\
		& + \frac{1}{2\alpha+1} \int_r^x \left(1+3\cY^+_{r,u}(v)-4 \cY^+_{r,u}(v)^{\alpha+1} \right)\dd u.
	\end{align}
	It is absorbed at $1$, and absorption at $0$ occurs at the first hit with $x\ge\mathfrak{t}$, where
    \begin{align}\label{eq:frak_t}
         \mathfrak t:=
             \begin{cases}
             0,&r_0\ge0,\\
             \eta_f(0),&r_0<0<\df^+_f,\\
             +\infty,&r_0<\df^+_f\le0,
             \end{cases}
    \end{align}
    is a $\widetilde{\mathscr E}$-stopping time.
    Moreover, the family $\cY^+:=\{(\cY^+_{r,x}(v),\, x\ge r)\}_{(v,r)\in \bI_+}$ satisfies the following regularity properties: almost surely,
	\begin{enumerate}[label=\textup{(\roman*)}, nosep]
		\item for all $(v,r)\in \bI_+$, $\cY^+_{r, r}(v)=v$;
		\item for all $r \le x$, $[0,1)\ni v \mapsto \cY^+_{r, x}(v)$ is càdlàg;
		\item for any $(v_1,r_1), (v_2,r_2) \in \bI_+$ and any $z \ge y \ge r_1 \vee r_2$,
		$$
		\cY^+_{r_1,y}(v_1) < \cY^+_{r_2,y}(v_2) 
		\ \Longrightarrow\ 
		\cY^+_{r_1,z}(v_1) \le \cY^+_{r_2,z}(v_2).
		$$
	\end{enumerate}
\end{theorem}

\begin{proof}
	By Definition~\ref{def:trans_flow} and the properties of $\cS$, $\cY^+$ inherits the stated regularity properties and absorption rules. Moreover, for every $t\ge0$,
	$$
	\{\mathfrak t\le t\}=\{\eta_f^{-1}(t)\ge0\}\in\mathscr E_{\eta_f^{-1}(t)}=\widetilde{\mathscr E}_t,
	$$
	so $\mathfrak t$ is a $(\widetilde{\mathscr E}_t)$-stopping time.
	Fix $v\in[0,1)$. By~\eqref{SDE:for_f} and $\cS_{r_0,\cdot}(va_0)\le f(\cdot)$, we obtain $\dd \langle f\rangle_x=\frac{4}{2\alpha+1}f(x)\dd x$ and $\dd\langle\cS_{r_0,\cdot}(va_0),f\rangle_x=\frac{4}{2\alpha+1}\frac{\cS_{r_0,x}(va_0)^{\alpha+1}}{f(x)^\alpha}\dd x$. Define $\xi_x(v):=\frac{\cS_{r_0,x}(va_0)}{f(x)}=\cY^+_{0,\eta_f(x)}(v)$. We deduce from It\^o's formula that, 
	\begin{align}\label{cY_xi}
		\nonumber \dd \xi_x(v)=&\; f(x)^{-1} \dd \cS_{r_0,x}(va_0)+\cS_{r_0,x}(va_0) \dd \!\left(f(x)^{-1}\right)+\dd \langle \cS_{r_0,\cdot}(va_0), f^{-1} \rangle_x\\
		\nonumber =&\; \frac{1}{f(x)} \dd \cS_{r_0,x}(va_0)-\frac{\xi_x(v)}{f(x)} \dd f(x)+ \frac{4}{2\alpha+1} \frac{ \xi_x(v)-\xi_x(v)^{\alpha+1}}{f(x)} \dd x\\
		\nonumber =&\; \frac{2}{f(x)} \int_{s\ge 0} \! \left[\frac{s^\alpha}{\cS_{r_0,x}(va_0)^\alpha}\mathbbm{1}_{\{s\le \cS_{r_0,x}(va_0)\}} - \xi_x(v)\frac{s^\alpha}{f(x)^\alpha}\mathbbm{1}_{\{s\le f(x)\}}\right]\! \cW(\dd s, \dd x)\\
		&\; + \frac{1}{(2\alpha+1)f(x)} \left[1+3\xi_x(v)-4\xi_x(v)^{\alpha+1}\right]\! \dd x.
	\end{align}
	Applying~\eqref{eq:def_tild_cW} to the predictable integrand below, after localization, gives
	\begin{align}\label{cY_diffusion}
		\nonumber &\int_0^{\eta_f(x)}\! \int_0^1 \left[\frac{2 s^\alpha}{\cY^+_{0,u}(v)^\alpha} \mathbbm{1}_{\{s\le \cY^+_{0,u}(v)\}} - 2\cY^+_{0,u}(v) s^\alpha \right]\! \widetilde{\cW} (\dd s, \dd u)\\
		 =&\; \int_{r_0}^x\! \int_{0}^{f(u)} \frac{2}{f(u)} \left[\frac{s^\alpha}{\cS_{r_0,u}(va_0)^\alpha}\mathbbm{1}_{\{s\le \cS_{r_0,u}(va_0)\}}-\xi_u(v)\frac{s^\alpha}{f(u)^\alpha}\mathbbm{1}_{\{s\le f(u)\}} \right]\! \cW (\dd s, \dd u).
	\end{align}
	By a change of variables, 
	\begin{align}\label{cY_drift}
		\int_{r_0}^x \frac{1}{f(u)}\left[1+3\xi_u(v)-4\xi_u(v)^{\alpha+1}\right]\dd u
		=\int_0^{\eta_f(x)}\left[1+3\cY^+_{0,u}(v)-4\cY^+_{0,u}(v)^{\alpha+1}\right]\dd u.
	\end{align}
	Combining~\eqref{cY_xi}, \eqref{cY_diffusion}, and~\eqref{cY_drift}, it follows that $\cY^+_{0,\cdot}(v)$ satisfies~\eqref{eq:SDE_perkin} before absorption. For $r>0$, the same Itô's formula argument applies. This completes the proof.
\end{proof}

\begin{remark}
	The construction of $\cY^+$ follows the approach of \cite[Theorem~5.9]{elie23}, but the resulting flow is not independent of $f$. Set $F_u:=f\circ\eta_f^{-1}(u)$ for $u\ge0$, and, for fixed $r\ge0$ and $v\in(0,1)$, write $Y_u:=\cY^+_{r,u}(v)$ for $u\ge r$. By~\eqref{SDE:for_f}, \eqref{eq:def_tild_cW}, and It\^o's formula,
	$$
	\log\frac{F_u}{a_0}
	=2\int_0^u\int_0^1s^\alpha\widetilde{\cW}(\dd s,\dd t)-\frac{u}{2\alpha+1}.
	$$
	Together with~\eqref{eq:SDE_perkin}, this gives
	$$
	\dd\langle\log F,Y\rangle_u
	=\frac{4Y_u(Y_u^\alpha-1)}{2\alpha+1}\,\dd u.
	$$
	The covariation density is strictly negative while $Y_u\in(0,1)$, so $F$ and $Y$ are not independent. Since $F$ is determined by $f$, neither are $f$ and $\cY^+$. Moreover, the displayed representation of $\log F$ shows that $F$ is a nonconstant measurable function of $\widetilde{\cW}$, so $\widetilde{\cW}$ and $f$ are not independent either.
\end{remark}

For later use, write the coefficients in~\eqref{eq:SDE_perkin} as
\begin{align}\label{eq:coe}
	\sigma(x,s)=2\mathbbm{1}_{\{0\le s\le 1\}} \bigl( \tfrac{s^\alpha}{x^\alpha}\mathbbm{1}_{\{s\le x\}} - x s^\alpha\bigr),\qquad 
	b(x)=\tfrac{1}{2\alpha+1}(1+3x-4x^{\alpha+1}),
\end{align}
where $\sigma(0,\cdot)=\sigma(1,\cdot)=0$. The superscript $+$ in $\cY^+$ indicates that only flow lines starting at $r\ge 0$ are considered. The definition of $\cY^+$ extends naturally to $\bI:=[0,1)\times\r$, giving a collection $\cY$. More precisely, almost surely, $\cY$ satisfies regularity properties (i)--(iii) in Theorem~\ref{thm:perkin_disintegration} on $\bI$, and for each $(v,r)\in\bI$, $(\cY_{r,x}(v),\,x\ge r)$ is the unique strong solution to~\eqref{eq:SDE_perkin}, absorbed upon hitting $1$; at $0$, absorption occurs after time $\mathfrak{t}$. The existence and uniqueness of $\cY$ follow from arguments similar to those in \cite[Section~6.2]{Yu26}, together with (ii) of Theorem~\ref{thm:perkin_disintegration}. We leave the details to the reader. When $\alpha=0$ and $r_0\ge0$, $\cY^+$ is the restriction of a killed $\mathrm{Jacobi}(1,0)$ flow to nonnegative starting times. This degenerate case, however, is not relevant in the present setting.

\begin{corollary}\label{cor:hit_boun_fi}
	For any $a\in (0,a_0)$, define the first hitting times
	$$
	T_0:=\inf\{x\ge r_0:\, \cS_{r_0,x}(a)=0\},\qquad T_f:=\inf\{x\ge r_0:\, \cS_{r_0,x}(a)=f(x)\}
	$$
	Then $\p(T_0\wedge T_f<\df_f^+)=1$.
\end{corollary}

\begin{proof}
	For $v=a/a_0$, let $T_{0,1}:=\inf\{x\ge 0:\, \cY^+_{0,x}(v)\notin (0,1)\}$, where $\cY^+_{0,\cdot}(v)$ is given by the transformation in Definition~\ref{def:trans_flow}. Thus, it suffices to prove
	\begin{align}\label{eq:T_01fi}
		\p\bigl(T_{0,1}<\infty)=1.
	\end{align}
	The coefficients in~\eqref{eq:coe} satisfy
	\[
	\bar\sigma(x):=\int_\r\sigma(x,s)^2\,\dd s
	=\frac{4}{2\alpha+1}x(1+x-2x^{\alpha+1})>0,\qquad x\in(0,1),
	\]
	and $(1+|b|)/\bar\sigma$ is locally integrable on $(0,1)$. For $0<a<v<b<1$, let $T_{a,b}:=\inf\{x\ge0:\,\cY^+_{0,x}(v)\notin(a,b)\}$. By Equations~(5.55) and~(5.59) in \cite[Section~5.5]{kara98},
	\begin{align}\label{eq:M_ab}
		\e[T_{a,b}]=-\int_a^v(p(v)-p(y))m(\dd y) +\frac{p(v)-p(a)}{p(b)-p(a)}\int_a^b(p(b)-p(y))m(\dd y),
	\end{align}
	where the scale function and speed measure are
	$$
	p(x):=\int_c^x\exp\left\{-2\int_c^\xi\frac{b(z)}{\bar\sigma(z)}\,\dd z\right\}\dd\xi,\qquad
	m(\dd x):=\frac{2\,\dd x}{p'(x)\bar\sigma(x)},\qquad c\in(0,1).
	$$
	As $x\downarrow0$, we have $2b(x)/\bar\sigma(x)=(2x)^{-1}+O(1)$, whereas $2b(x)/\bar\sigma(x)\to(4\alpha+1)/(4\alpha+2)$ as $x\uparrow1$. Hence there are constants $C_0,C_1>0$ such that
	$$
	p'(x)\sim C_0x^{-1/2}\quad(x\downarrow0),\qquad
	p'(x)\rightarrow C_1\quad(x\uparrow1).
	$$
	In particular, $p(0):=p(0+)>-\infty$ and $p(1):=p(1-)<\infty$. Since $\bar\sigma(x)\sim4x/(2\alpha+1)$ as $x\downarrow0$ and $\bar\sigma(x)\sim4(1-x)$ as $x\uparrow1$, these estimates imply that $m(\dd x)$ is integrable near $0$ and $(p(1)-p(x))m(\dd x)$ is integrable near $1$. By continuity, $T_{a,b}\uparrow T_{0,1}$ as $a\downarrow0$ and $b\uparrow1$. Monotone convergence on the left-hand side of~\eqref{eq:M_ab} and dominated convergence on the right-hand side give
	\begin{align*}
		\e[T_{0,1}]=-\int_0^v(p(v)-p(y))m(\dd y)+ \frac{p(v)-p(0)}{p(1)-p(0)}\int_0^1(p(1)-p(y))m(\dd y)
		<\infty.
	\end{align*}
	This proves~\eqref{eq:T_01fi} and hence the corollary.
\end{proof}

\subsection{Perfect flow property}\label{subsec:perf}

This section proves the perfect flow property (Proposition~\ref{prop:perfect}) deferred from Section~\ref{sec:PSR_flow}. Let $\cS^*$ denote the dual PSR flow. For $(v,r)\in \bI_+$ and $0\le x\le r$, define
\begin{align}\label{eq:cY_00}
	\cY^{\#}_{r, x}(v)
	:= \frac{\cS^*_{\eta_{f}^{-1}(r), \eta_{f}^{-1}(x)}\left(v f \circ \eta_{f}^{-1}(r)\right)}{f \circ \eta_{f}^{-1}(x)}.
\end{align}
By Proposition~\ref{prop:non-crossing}, $\cY^\#$ takes values in $[0,1]$. Its trajectories are continuous, since the trajectories of $\cS^*$ and the inverse clock $\eta_f^{-1}$ are continuous and $f$ is strictly positive on $[r_0,\df_f^+)$. Recall the family $\cY^+$ studied in Theorem~\ref{thm:perkin_disintegration}.

\begin{lemma}\label{lm:equ_cY}
    Almost surely, for all $(v,r)\in\bI_+$ and $0\le x\le r$,
    $$
    \cY^\#_{r,x}(v)=\cY^*_{r,x}(v):=\inf\{b\in[0,1):\cY^+_{x,r}(b)>v\}\wedge1,
    $$
    where $\inf\varnothing=+\infty$.
\end{lemma}

\begin{proof}
    By Definition~\ref{def:trans_flow} and~\eqref{eq:def_dual_+},
    \begin{align*}
    \cY^*_{r,x}(v)
    &=\inf\{b>0:\, \cS_{\eta_{f}^{-1}(x), \eta_{f}^{-1}(r)}\left(b f \circ \eta_{f}^{-1}(x)\right)>v f \circ \eta_{f}^{-1}(r)\}\\
    &=\bigl[f \circ \eta_{f}^{-1}(x)\bigr]^{-1}\inf\{b>0:\, \cS_{\eta_{f}^{-1}(x), \eta_{f}^{-1}(r)}(b)>v f \circ \eta_{f}^{-1}(r)\}\\
    &=\bigl[f \circ \eta_{f}^{-1}(x)\bigr]^{-1} \cS^*_{\eta_{f}^{-1}(r), \eta_{f}^{-1}(x)}(v f \circ \eta_{f}^{-1}(r))=\cY^{\#}_{r, x}(v).
\end{align*}
This completes the proof.
\end{proof}

Let $\widetilde{\cW}^*$ denote the image of $\widetilde{\cW}$ under the map $(a,r)\mapsto(a,-r)$. 

\begin{proposition}\label{prop:transformed_separation}
	Almost surely, simultaneously for all $(v,r)\in\bI_+$,
	$$
	\cY^\#_{r,x}(v)<1,\qquad 0\le x\le r.
	$$
\end{proposition}

\begin{proof}
	Let $\widehat{\cY}$ be obtained from $\cY^+$ by absorption at the first hit of $0$. Then $\widehat{\cY}\le\cY^+$ and, by Theorem~\ref{thm:perkin_disintegration} and Corollary~\ref{p:properties}(i), $\widehat{\cY}$ is a flow with coefficients $(\sigma,b)$ from~\eqref{eq:coe}, absorbed at both endpoints.
	
	Let $\cY^{\mathrm b}$ be the backward flow driven by $\widetilde{\cW}^*$, with coefficients
	$$
	\sigma^*=-\sigma,\qquad b^*(z)=(2\alpha+1)^{-1}\big(1+z-4(\alpha+1)z^{\alpha+1}\big),
	$$
	and without absorption at either endpoint. Its existence follows from~\cite[Theorem~4.3]{Yu26}, after extending the noise independently to all times and the coefficients constantly outside $[0,1]$. Note that $b^*=-b+\tfrac12\bar\sigma'$.
	For the scale function $p$ associated with $(\bar\sigma,b^*)$, we have
	$$
	\frac{2b^*(z)}{\bar\sigma(z)}
	=\frac{1+z-4(\alpha+1)z^{\alpha+1}}{2(z+z^2-2z^{\alpha+2})}
	=\frac{1}{z-1}+O(1),\qquad z\uparrow1.
	$$
	Hence $p'(x)\asymp(1-x)^{-1}$ and $p(1-)=+\infty$. By~\cite[Proposition~5.22]{kara98}, $\cY^{\mathrm b}$ cannot reach $1$ from any fixed starting value in $(0,1)$. The absorption condition follows as in~\cite[Remark~5.6]{Yu26}, since $\cY^{\mathrm b}$ has no atom at $0$ at positive elapsed times, by a localized Lamperti--Girsanov argument, and cannot reach $1$. The argument of~\cite[Theorem~5.5]{Yu26}, restricted to the positive time axis, yields
	$$
	\cY^{\mathrm b}_{r,x}(v)
	=\inf\{a\in(0,1):\widehat{\cY}_{x,r}(a)>v\}\wedge1,
	\qquad 0<x\le r,\quad 0<v<1.
	$$
	Since $\widehat{\cY}\le\cY^+$, Lemma~\ref{lm:equ_cY} and continuity at $x=0$ imply that, for each fixed $(v,r)\in(0,1)\times(0,\infty)$, almost surely,
	$$
	\cY^\#_{r,x}(v)\le\cY^{\mathrm b}_{r,x}(v)<1,
	\qquad 0\le x\le r.
	$$
	
	On a common event of probability one, the preceding bound holds for all rational starting points. By continuity, right-continuity in the initial height and Lemma~\ref{lm:equ_cY}, every $(v,r)\in\bI_+$ admits a rational pair $(w,q)\in(0,1)\times(r,\infty)$ such that $\cY^\#_{q,r}(w)>v$. Non-crossing then gives
	$$
	\cY^\#_{r,x}(v)\le\cY^\#_{q,x}(w)<1,\qquad 0\le x\le r,
	$$
	proving the simultaneous assertion.
\end{proof}

For each fixed reference point $(a_0,r_0)\in\h^*$, almost surely,
\begin{align}\label{eq:hhh}
	\cS^*_{r,y}(a)<\cS_{r_0,y}(a_0),\qquad r_0\le y\le r<\df_f^+,\quad 0\le a<\cS_{r_0,r}(a_0),
\end{align}
simultaneously in $r$, $a$ and $y$. We are now ready to prove the perfect flow property.

\begin{proof}[Proof of Proposition~\ref{prop:perfect}]
	By property~\ref{item:R1},~\eqref{eq:perf} holds when $x=r$ or $y=x$. Suppose that $r<x<y$, $\cS_{r,x}(a)>0$, and $\cS_{r,y}(a)\neq \cS_{x,y}\circ\cS_{r,x}(a)$. For any $a'>\cS_{r,x}(a)$, property~\ref{item:R3} gives $\cS_{x,y}(a')\ge \cS_{r,y}(a)$. Letting $a'\downarrow\cS_{r,x}(a)$ and using property~\ref{item:R2}, we obtain $\cS_{x, y} \circ \cS_{r, x}(a)\ge \cS_{r,y}(a)$, and hence necessarily $\cS_{r,y}(a)<\cS_{x,y}\circ\cS_{r,x}(a)$. By Lemma~\ref{lm:coal}, the inequality remains strict for all $z\in (x,y]$. For any $(b,z)\in \cD$ such that
    $$
    b\in \big(\cS_{r,z}(a), \cS_{x,z}\circ\cS_{r,x}(a)\big), \quad z\in (x,y],
    $$
    property~\ref{item:R3} implies $\cS_{x,z}(a')\ge \cS_{x,z}\circ\cS_{r,x}(a)>b$ for any $a'>\cS_{r,x}(a)$ and $\cS_{x,z}(a')\le \cS_{r,z}(a)<b$ for any $a'<\cS_{r,x}(a)$. Thus, by~\eqref{eq:def_dual_+},
    $$
    \cS^*_{z,x}(b)=\inf\{c\ge0:\,\cS_{x,z}(c)>b\}=\cS_{r,x}(a)>0.
    $$
    Therefore,
    \begin{align}\label{pro_per}
    \nonumber &\p\Bigl(\exists\,a\ge 0,\,r<x<y:\, \cS_{r,x}(a)>0,\ \cS_{r,y}(a)\neq\cS_{x,y}\circ\cS_{r,x}(a)\Bigr)\\
    \nonumber &\le \p\Bigl(\exists\,(a,r)\in\h,\,(b,z)\in\cD,\,r<x<z:\, b>\cS_{r,z}(a),\,\cS^*_{z,x}(b)=\cS_{r,x}(a)>0\Bigr)\\
    \nonumber &= \p\Bigl(\exists\,(a,r),(b,z)\in\cD,\,r<x<z:\, b>\cS_{r,z}(a),\,\cS^*_{z,x}(b)=\cS_{r,x}(a)>0\Bigr)\\
    \nonumber &\le \sum_{\substack{(a,r),(b,z)\in\cD, z>r}}\p\Bigl(\exists\,x\in (r,z):\, b>\cS_{r,z}(a), \cS^*_{z,x}(b)=\cS_{r,x}(a)>0\Bigr)\\
    &\le \sum_{\substack{(a,r),(b,z)\in\cD, z>r}}\p\Bigl(\exists\,x\in (r,z):\, b<\cS_{r,z}(a), \cS^*_{z,x}(b)=\cS_{r,x}(a)>0\Bigr).
    \end{align}
    Here, the equality follows from Proposition~\ref{prop:skeleton}. The last inequality follows from Corollary~\ref{cor:dual} and Proposition~\ref{prop:self_dual}, with the reindexing $((a,r),(b,z))\mapsto((b,-z),(a,-r))$. For fixed $(a,r),(b,z)\in\cD$ with $z>r$, let $\mathcal J$ denote the event
    $$
    \Bigl\{\exists\,x\in(r,z):\,b<\cS_{r,z}(a),\,\cS^*_{z,x}(b)=\cS_{r,x}(a)>0\Bigr\}.
    $$
    
    On $\mathcal J$, fix such an $x$. By non-crossing, $\cS^*_{z,u}(b)\le\cS_{r,u}(a)$ for $u\in[r,z]$. Moreover, $\cS_{r,\cdot}(a)$ stays strictly positive on $[x,z]$: a zero at a nonnegative time would force $\cS_{r,z}(a)=0$, whereas a zero at a negative time would force $\cS^*_{z,x}(b)=0$ by backward absorption. Choose $r_1\in\bD\cap(r,x)$ sufficiently close to $x$ that $\cS_{r,\cdot}(a)>0$ on $[r_1,z]$, and set $a_1:=\cS_{r,r_1}(a)$. By~\eqref{eq:flow_prop} and Corollary~\ref{p:properties}(i), we may choose $a_2\in\bD_+$ with $a_2>a_1$ such that
    $$
    \cS_{r_1,x}(a_2)=\cS_{r_1,x}(a_1)=\cS_{r,x}(a)=\cS^*_{z,x}(b).
    $$
    
    Property~\ref{item:R3} gives $\cS_{r_1,u}(a_2)\ge\cS_{r,u}(a)>0$ for $u\in[r_1,z]$, and Lemma~\ref{lm:coal} gives $\cS_{r_1,z}(a_2)=\cS_{r,z}(a)>b$. Applying~\eqref{eq:hhh} with reference point $(a_2,r_1)$ now contradicts the displayed equality. Since~\eqref{eq:hhh} holds on a common event of probability one for all such dyadic reference points, $\p(\mathcal J)=0$. Together with~\eqref{pro_per}, this completes the proof. 
\end{proof}

\appendix

\section{Preliminary estimates}
\label{app:est}
We start with a general framework, of which the solution to the SDE~\eqref{SDE:for} is a specific instance. Let $\delta>0$ be a constant and let $\sigma:\r_+\times\r\to\r$ be Borel measurable. For every $a\ge0$, we assume that the SDE
\begin{align}\label{eq:SDE_gener}
	S_x(a)
	= a + \int_0^x \int_{\r}
	\sigma\bigl(S_u(a),s\bigr)\cW(\dd s,\dd u)
	+ \delta x
\end{align}
admits a continuous nonnegative strong solution on $[0,\infty)$, and that pathwise uniqueness holds. All these solutions are driven by the same space-time white noise $\cW$. We further assume that there exists a constant $L>0$ such that, for all $x,y\in\r_+$,
\begin{align}\label{ineq:ass_sig}
	\int_{\r}
	\bigl[\sigma(x,s)-\sigma(y,s)\bigr]^2\,\dd s
	\ge L|x-y|.
\end{align}

Fix $0\le a_1<a_2$. For $i=1,2$, write $S_x^i:=S_x(a_i)$.  We define the coalescence time of the two trajectories by
\begin{align}\label{eq:def_tau}
	\tau(a_1,a_2)
	:= \inf\bigl\{x>0:S_x(a_1)=S_x(a_2)\bigr\}.
\end{align}
By continuity and pathwise uniqueness, almost surely,
$S_x^1\le S_x^2$ for all $x\ge0$, and the two solutions coincide
after their first meeting.

Recall that the hitting time of $0$ for a ${\rm BESQ}^0$ process started from $a>0$ has the same law as $\frac{a}{2X}$, where $X\sim\mathrm{Exp}(1)$; see~\cite[Eq.~(15)]{GJY03} or~\cite[Exercise~1.23, Chapter~XI]{RevYor}. The following lemma gives a tail estimate for the coalescence time $\tau(a_1,a_2)$.

\begin{lemma}\label{lm:appen_tau}
    Fix $0\le a_1<a_2$, and let $\tau(a_1,a_2)$ be defined by~\eqref{eq:def_tau}. Under the above assumptions, for every $v>0$,
    \begin{align*}
        \p(\tau(a_1,a_2)>v) \le \frac{2(a_2-a_1)}{Lv}.
    \end{align*}
\end{lemma}

\begin{proof}
	Set $S_{a_1,a_2}(x):=S_x^2-S_x^1$ and $\tau:=\tau(a_1,a_2)$. Since the two solutions are ordered and their drift terms cancel, $S_{a_1,a_2}$ is a nonnegative continuous local martingale, absorbed at zero, with quadratic variation
	\begin{align}\label{eq:qua_var}
		\dd\langle S_{a_1,a_2}\rangle_x
		= \int_{\r}\bigl[\sigma(S_x^2,s)-\sigma(S_x^1,s)\bigr]^2\,\dd s \dd x \ge L S_{a_1,a_2}(x)\dd x.
	\end{align}
	
	\noindent
	Write $A_x:=\langle S_{a_1,a_2}\rangle_x$. By the Dambis--Dubins--Schwarz theorem, there exists a Brownian motion $\beta$ started at $a_2-a_1$ such that $S_{a_1,a_2}(x)=\beta(A_x)$ for all $x\ge0$. Let
	$$
	\tau^\beta:=\inf\{t>0:\beta(t)=0\}.
	$$
	For $x<\tau$, continuity of $A$ and positivity of $S_{a_1,a_2}$ on $[0,x]$ imply that $A_x<\tau^\beta$. Hence, by~\eqref{eq:qua_var} and a change of variables,
	\begin{align*}
		x\le \frac{1}{L}\int_0^x \frac{\dd A_u}{S_{a_1,a_2}(u)}
		= \frac{1}{L}\int_0^{A_x}\frac{\dd t}{\beta(t)}\le \frac{1}{L}\int_0^{\tau^\beta}\frac{\dd t}{\beta(t)}.
	\end{align*}
	The integral up to $\tau^\beta$ is defined by taking the limit from below. Letting $x\uparrow\tau$, with infinite values allowed, gives, for every $v>0$,
	\begin{align}\label{eq:A1}
		\p(\tau>v)
		\le \p\left(\frac{1}{L}\int_0^{\tau^\beta}\frac{\dd t}{\beta(t)}>v\right).
	\end{align}
	
	Next, let $Y$ be a ${\rm BESQ}^0$ process started at $a_2-a_1$, and let $\tau^Y:=\inf\{x>0:Y_x=0\}$. Since $\dd\langle Y\rangle_x=4Y_x\,\dd x$, a second application of the DDS theorem gives $Y_x=\beta'(\langle Y\rangle_x)$, where $\beta'$ is a Brownian motion started at $a_2-a_1$. Denote its first hitting time of zero by $\tau^{\beta'}$. Then $\langle Y\rangle_{\tau^Y}=\tau^{\beta'}$, and the same change of variables yields
	$$
	\int_0^{\tau^{\beta'}}\frac{\dd t}{\beta'(t)}
	= \int_0^{\tau^Y}\frac{\dd\langle Y\rangle_u}{Y_u}
	= 4\tau^Y.
	$$
	Since $\beta'$ has the same law as $\beta$ and $\tau^Y$ has the law of $\frac{a_2-a_1}{2X}$, where $X\sim\mathrm{Exp}(1)$, $\int_0^{\tau^\beta}\frac{\dd t}{\beta(t)}$ has the same law as $\frac{2(a_2-a_1)}{X}$. Combining this identity with~\eqref{eq:A1}, we obtain
	\begin{align*}
		\p(\tau>v)\le\p\left(X<\frac{2(a_2-a_1)}{Lv}\right)=1-\exp\left(-\frac{2(a_2-a_1)}{Lv}\right)\le \frac{2(a_2-a_1)}{Lv},
	\end{align*}
	as required.
\end{proof}

\begin{corollary}\label{app:coal}
	Fix $r\in\r$ and $0\le a_1<a_2$. Let $\cShat_{r,\cdot}(a_i)$, $i=1,2$, denote the non-killed solutions to~\eqref{SDE:for} started from $a_i$. Define their coalescence time by
	\begin{align*}
		\widehat{\tau}(r;a_1,a_2)
		:= \inf\bigl\{y>0:\cShat_{r,r+y}(a_1)=\cShat_{r,r+y}(a_2)\bigr\}.
	\end{align*}
	Then, for every $v>0$,
	\begin{align*}
		\p\bigl(\widehat{\tau}(r;a_1,a_2)>v\bigr)
		\le \frac{2\alpha+1}{2}\frac{a_2-a_1}{v}.
	\end{align*}
\end{corollary}

\begin{proof}
	Without loss of generality, we may assume $r=0$.  In~\eqref{eq:SDE_gener}, the SDE~\eqref{SDE:for} corresponds to $\sigma(x,s)=\frac{2}{x^\alpha}s^\alpha\mathbbm{1}_{\{0< s\le x\}}$ and $\delta=\frac{1}{2\alpha+1}$.
	A direct computation gives
	$$
	    \int_{\r} \bigl[\sigma(x,s)-\sigma(y,s)\bigr]^2 \,\dd s
	    = \frac{4}{2\alpha+1} \left[ x+y - 2\frac{(x\wedge y)^{2\alpha+1}}{x^\alpha y^\alpha} \right]
	    \ge \frac{4}{2\alpha+1} |x-y|,
	$$
	for any $x,y\in \r_+$, so that condition~\eqref{ineq:ass_sig} holds with $L = \frac{4}{2\alpha+1}$. 
	The desired result then follows directly from Lemma~\ref{lm:appen_tau}.
\end{proof}

\section{Results on discrepancy urn processes}
\label{app:discrepancy}
In this appendix we provide auxiliary estimates for the discrepancy urn processes in two regimes of \eqref{eq:weight_func}: the polynomially self-repelling case ($\alpha>0$) and the asymptotically free case ($\alpha=0$).  The asymptotically free estimates support the discussion of an alternative approach to the BMPE limit in Remark~\ref{rem:AF}.

Recall the generalized P\'olya urn scheme introduced in Section~\ref{sec:joint_RK}. The weights specified in \eqref{weights_on_Z} naturally divide the sites on $\z$ into three classes. We denote the urn processes at sites to the left of the origin by $\{(\fB_n^-, \fR_n^-)\}_{n \ge 0}$, those at the origin by $\{(\fB_n^0, \fR_n^0)\}_{n \ge 0}$, and those to the right by $\{(\fB_n^+, \fR_n^+)\}_{n \ge 0}$. The corresponding discrepancy urn processes are denoted by $\fD^-$, $\fD^0$, and $\fD^+$, and the stopping times by $\tau_{m}^{*^{\dag}}$ for $\dag\in \{-,0,+\}$ and $* \in \{\fB,\fR\}$. For notational simplicity, we abbreviate $\fD^{\dag}_{\tau_{m}^{*^{\dag}}}$ by $\fD^{\dag}_{\tau_{m}^{*}}$.
By symmetry, the following distributional identities hold: 
\begin{align}\label{identi_law}
    \fD^+_{n} \stackrel{d}{=} -\fD^-_{n}, \quad \fD^+_{\tau_n^\fR} \stackrel{d}{=} -\fD^-_{\tau_n^\fB}, \quad \fD^-_{\tau_n^\fR} \stackrel{d}{=} -\fD^+_{\tau_n^\fB}, \quad \fD^0_{\tau_n^\fR} \stackrel{d}{=} -\fD^0_{\tau_n^\fB}.
\end{align}

The following lemma gives a concentration inequality that holds in both regimes.

\begin{lemma}[\cite{KMMP26}, Lemma~3.1; \cite{Elena23}, Lemma~4.1]\label{lm:4_1}
Let $w$ be as in \eqref{eq:weight_func} with $\alpha\ge 0$.
There exist constants $C, c>0$ such that for any integers $n, m \ge 1$ and $\dag \in\{-,0,+\}$,
$$
\P\left(\left|\fD_{\tau_n^{\fB}}^\dag\right| \ge m\right) \leq C e^{\frac{-c m^2}{m \vee n}} \quad \text{and} \quad \P\left(\left|\fD_{\tau_n^{\fR}}^\dag\right| \ge m\right)\leq C e^{\frac{-c m^2}{m \vee n}}.
$$
\end{lemma}

\subsection{Polynomially self-repelling case}\label{appsub:PSR}

Fix $\alpha>0$. We list below the results from \cite{KMMP26}, extending the generalized P\'olya urn results of \cite{Elena23}, that are needed in the proof of Theorem~\ref{thm:meas}.

\begin{lemma}[\cite{KMMP26}, Lemma~3.6]\label{lm:4_7}
    There exists a constant $C>0$ such that for any integer $n \ge 1$ and $\dag \in\{-,0,+\}$,
    \begin{align}\label{diff_tau_n}
        \e\left[\left(\fD^\dag_{\tau_n^\fB}-\fD^\dag_{2 n}\right)^2\right] \leq C \sqrt{n} \quad \text{and} \quad \e\left[\left(\fD^\dag_{\tau_n^\fR}-\fD^\dag_{2 n}\right)^2\right] \leq C \sqrt{n}.
    \end{align}
\end{lemma}

\begin{proposition}[\cite{KMMP26}, Proposition~3.9]\label{prop:exp_discre}
\begin{equation*} 
\begin{aligned}
    \lim_{n \rightarrow \infty} \e\!\left[ \fD^+_{\tau_n^\fB} \right]
        &= \lim_{n \rightarrow \infty} \e\!\left[ \fD^-_{\tau_n^\fB} +1\right] = \frac{1}{2(2\alpha+1)},\\
    \lim_{n \rightarrow \infty} \e\!\left[ -\fD^-_{\tau_n^\fR} \right] 
        &= \lim_{n \rightarrow \infty} \e\!\left[ -\fD^+_{\tau_n^\fR}+1 \right] = \frac{1}{2(2\alpha+1)},\\
    \lim_{n \rightarrow \infty} \e\!\left[ \fD^0_{\tau_n^\fR} \right] 
        &= \lim_{n \rightarrow \infty} \e\!\left[ -\fD^0_{\tau_n^\fB} \right] = \frac{\alpha}{2\alpha+1}.
\end{aligned}
\end{equation*}
\end{proposition}

\begin{proposition}[\cite{KMMP26}, Corollary~3.3 \& Proposition~3.7]\label{prop:var_discre}
    For any $\dag \in\{-,0,+\}$,
    \begin{align}\label{eq:var_fD_+}
        \lim_{n\to \infty}\frac{\operatorname{Var}\!\left(\fD_{n}^\dag\right)}{n}
        =\frac{1}{2\alpha+1},
    \end{align}
    and
    \begin{align}\label{eq:var_fD_tau_+}
        \lim_{n\to \infty}\frac{\operatorname{Var}\!\left(\fD_{\tau^\fB_n}^\dag\right)}{2n}=\lim_{n\to \infty}\frac{\operatorname{Var}\!\left(\fD_{\tau^\fR_n}^\dag\right)}{2n}
        =\frac{1}{2\alpha+1}.
    \end{align}
\end{proposition}

We next derive estimates for the covariance between $\fD_{\tau_n^*}$ and $\fD_{\tau_m^*}$ as $m,n \to \infty$, where $* \in \{\fB, \fR\}$, using results from \cite{KMMP26}. In the following proofs, we fix $\dag\in\{-,0,+\}$ and suppress the superscript.

\begin{lemma}
For any $k_1, k_2 > 0$ and $\dag \in\{-,0,+\}$,
\begin{align}\label{eq:cov_0}
\lim_{n\to\infty}
\frac{\operatorname{Cov}\!\left(
\fD^\dag_{\lfloor n k_1\rfloor},
\fD^\dag_{\lfloor n k_2\rfloor}
\right)}{n}
=
\frac{(k_1\wedge k_2)^{2\alpha+1}}
{(2\alpha+1)k_1^\alpha k_2^\alpha}.
\end{align}
\end{lemma}

\begin{proof}
	By symmetry, assume $k_1\le k_2$. Choose an integer $L>k_2/2$, and set $N=Ln$, $m_i=\lfloor nk_i\rfloor$, and $t_i=m_i/N$ for $i=1,2$. For all sufficiently large $n$, $\delta\le t_1\le t_2\le2$, where $\delta=k_1/(2L)>0$. By~\cite[Lemma~3.2]{KMMP26},
	$$
	\frac1N\operatorname{Var}\!\left(t_2^\alpha\fD_{m_2}-t_1^\alpha\fD_{m_1}\right)
	=\int_{t_1}^{t_2}u^{2\alpha}\,\dd u+O\!\left(N^{-1/2}\log^2 N\right).
	$$
	Expanding the variance and using~\eqref{eq:var_fD_+}, together with $t_i\to k_i/L$ and $N/n=L$, gives
	$$
	\lim_{n\to\infty}\frac{\operatorname{Cov}(\fD_{m_1},\fD_{m_2})}{n}
	=\frac{L(k_1/L)^{2\alpha+1}}{(2\alpha+1)(k_1/L)^\alpha(k_2/L)^\alpha}
	=\frac{k_1^{2\alpha+1}}{(2\alpha+1)k_1^\alpha k_2^\alpha}.
	$$
	This proves~\eqref{eq:cov_0}.
\end{proof}

\begin{proposition}\label{prop:cov}
For $* \in \{\fB, \fR\}$, $\dag \in\{-,0,+\}$, and any $k_1, k_2 > 0$,
\begin{align*}
\lim_{n\to\infty}
\frac{\operatorname{Cov}\!\left(\fD^\dag_{\tau^{*}_{\lfloor n k_1\rfloor}}, \fD^\dag_{\tau^{*}_{\lfloor n k_2\rfloor}}\right)}{2n}
=\frac{(k_1 \wedge k_2)^{2\alpha+1}}{(2\alpha+1) k_1^\alpha k_2^\alpha}.
\end{align*}
\end{proposition}

\begin{proof}
	By symmetry, it suffices to consider $*=\fB$. Set $m_i=\lfloor nk_i\rfloor$ for $i=1,2$. By Cauchy--Schwarz and~\eqref{diff_tau_n}, \eqref{eq:var_fD_+}, and~\eqref{eq:var_fD_tau_+},
	\begin{align*}
		&\qquad\left|\operatorname{Cov}\!\left(\fD_{\tau_{m_1}^\fB},\fD_{\tau_{m_2}^\fB}\right)-\operatorname{Cov}\!\left(\fD_{2m_1},\fD_{2m_2}\right)\right|\\
		&\le
		\sqrt{\e\!\left[(\fD_{\tau_{m_1}^\fB}-\fD_{2m_1})^2\right]\operatorname{Var}(\fD_{\tau_{m_2}^\fB})}+
		\sqrt{\operatorname{Var}(\fD_{2m_1})\e\!\left[(\fD_{\tau_{m_2}^\fB}-\fD_{2m_2})^2\right]}
		=O(n^{3/4}).
	\end{align*}
	Since $|\fD_{2m_i}-\fD_{\lfloor2nk_i\rfloor}|\le1$, the result follows from~\eqref{eq:cov_0} after dividing by $2n$.
\end{proof}

\subsection{Asymptotically free case}\label{appsub:AF}

Suppose that $w$ is positive and monotone, satisfying~\eqref{eq:weight_func} with $\alpha=0$. Define
\begin{equation}\label{def:gamma}
	\gamma
	:=\lim_{n\to\infty}
	\left(
	\sum_{j=0}^{n-1}\frac{1}{w(2j+1)}
	-\sum_{j=0}^{n-1}\frac{1}{w(2j)}
	\right).
\end{equation}
By monotonicity, the limit exists and $\gamma<1$. Moreover, $\gamma\in[0,1)$ when $w$ is nonincreasing, whereas $\gamma\le0$ when $w$ is nondecreasing.
The following expectation limits follow from the argument in~\cite[Lemma~4.3]{LW25} and the discussion following its proof.

\begin{proposition}\label{prop:AF_expe}
	For $*\in\{\fB,\fR\}$ and $\dag\in\{-,0,+\}$,
	$$
	\lim_{n\to\infty}\e\!\left[\fD^\dag_{\tau_n^*}\right]
	=
	\begin{cases}
		\gamma, & \dag=+,\\
		0, & \dag=0,\\
		-\gamma, & \dag=-.
	\end{cases}
	$$
\end{proposition}

\begin{proposition}\label{prop:cov_AF}
	For $\dag\in\{-,0,+\}$, $*_1,*_2\in\{\fB,\fR\}$,
	and $k_1,k_2>0$,
	$$
	\lim_{n\to\infty}
	\frac{1}{2n}\operatorname{Cov}\!\left(
	\fD^\dag_{\tau_{\lfloor nk_1\rfloor}^{*_1}},
	\fD^\dag_{\tau_{\lfloor nk_2\rfloor}^{*_2}}
	\right)
	=k_1\wedge k_2.
	$$
\end{proposition}

\begin{proof}[Proof sketch]
	Fix $\dag$ and take $*=\fB$, suppressing these superscripts. Let $\cF^{\fB,\fR}$ be the natural filtration of the urn process, and set $\cF_k:=\cF_{\tau_k}^{\fB,\fR}$. Define
	\begin{align*}
		\zeta_k
		&:=\e\!\left[\fD_{\tau_k}-\fD_{\tau_{k-1}}\mid\cF_{k-1}\right]
		=\e\!\left[\fD_{\tau_k}-\fD_{\tau_{k-1}}\mid\fR_{\tau_{k-1}}\right],\\
		\eta_k
		&:=\e\!\left[(\fD_{\tau_k}-\fD_{\tau_{k-1}})^2\mid\cF_{k-1}\right]
		=\e\!\left[(\fD_{\tau_k}-\fD_{\tau_{k-1}})^2\mid\fR_{\tau_{k-1}}\right].
	\end{align*}
	Then $\xi_k:=\fD_{\tau_k}-\fD_{\tau_{k-1}}-\zeta_k$ is a martingale difference sequence with respect to $(\cF_k)_{k\ge0}$, with conditional variance
	$$
	\sigma_k^2:=\e[\xi_k^2\mid\cF_{k-1}]=\eta_k-\zeta_k^2.
	$$
	The weights are bounded above and away from zero, so the conditional interarrival times have uniformly geometric tails. A geometric comparison using~\eqref{eq:weight_func} and Lemma~\ref{lm:4_1} gives
	$$
	\|\zeta_k\|_2+\e|\eta_k-2|\le C k^{-1},
	\qquad
	\sup_{k\ge1}\e[\xi_k^4]\le C.
	$$
	In particular, $\e|\sigma_k^2-2|\le C/k$, and hence $\frac1n\sum_{k=1}^{\lfloor nt\rfloor}\sigma_k^2\rightarrow 2t$ locally uniformly in probability, while the fourth-moment bound gives the Lindeberg condition. The drift is negligible on the $\sqrt n$ scale, since
	$\big\|\max_{0\le j\le\lfloor nT\rfloor}\big|\sum_{k=1}^j\zeta_k\big|\big\|_2\le C_T\log(n+1)$, for every $T<\infty$. Since $\fD_{\tau_j}=\sum_{k=1}^j\xi_k+\sum_{k=1}^j\zeta_k$, the martingale functional central limit theorem \cite[Theorem~8.2.8]{Durrett2019} yields
	\begin{equation}\label{eq:AF_urn_CLT}
		\left(\frac{\fD_{\tau_{\lfloor nt\rfloor}}}{\sqrt{2n}},
		\,t\ge0\right)
		\Rightarrow (B_t,\,t\ge0),
	\end{equation}
	in $D([0,\infty),\r)$ with the $J_1$ topology,
	where $B$ is a standard Brownian motion.
	
	To compare the two clocks, set $q_n:=\fB_{\tau_n^\fR}=n-\fD_{\tau_n^\fR}$. By Lemma~\ref{lm:4_1}, $q_n/n\to1$ in probability. Monotonicity places $\fD_{\tau_n^\fR}$ between
	$\fD_{\tau_{q_n}^\fB}$ and $\fD_{\tau_{q_n+1}^\fB}+1$. Thus~\eqref{eq:AF_urn_CLT} and continuity of its limit give $(\fD_{\tau_n^\fR}-\fD_{\tau_n^\fB})/\sqrt n\to0$ in probability. Consequently, for either choice of $*_1,*_2$, the normalized discrepancies converge jointly to $(B_{k_1},B_{k_2})$ for the same Brownian motion $B$.
	Finally, Lemma~\ref{lm:4_1} gives $\e|\fD_{\tau_m^*}|^4\le Cm^2$ for $m\ge1$. The normalized products are therefore uniformly integrable, and the joint convergence yields the stated covariance limit.
\end{proof}

\section{Tightness of the discrete local time flow}
\label{app:tight}
This appendix discusses the tightness of the collection of rescaled upcrossing profiles, a result of independent interest. Throughout this appendix, we linearly interpolate the profiles in~\eqref{eq:def_S_N} on the common grid $N^{-1}\z$ and continue to denote them by $S^N_{r,\cdot}(a)$. The interpolated paths retain the non-crossing property of Lemma~\ref{lem:noncrossing} and the convergence in
Theorem~\ref{thm:1B_toth}. 

Following \cite[Section~3]{BW04}, let $(\cH,d_\cH)$ be the space of compact sets of forward paths, with Borel $\sigma$-field $\cF_\cH$. Define $(\cH^b,d_{\cH^b})$ and $\cF_{\cH^b}$ by reversing the time coordinate. Equip $\cH^D:=\cH\times\cH^b$ with the maximum of the two Hausdorff metrics and its Borel $\sigma$-field $\cF_{\cH^D}=\cF_\cH\otimes\cF_{\cH^b}$.
For each $N$, let $S^N$ be the ordered pair consisting of the closures of
$$
\left\{S^N_{k/N,\cdot}(m/N)\big|_{[k/N,\infty)}:m\in\N_0,\, k\in\z\right\},\quad
\left\{S^N_{k/N,\cdot}(m/N)\big|_{(-\infty,k/N]}:m\in\N_0,\, k\in\z\right\},
$$
respectively. The two closures are readily checked to be compact and measurable, so $S^N$ is an $(\cH^D,\cF_{\cH^D})$-valued random variable. We adapt the argument of~\cite[Proposition~B.2 and Remark~B.2]{BW04} to prove the following proposition.

\begin{proposition}
The laws of $\{S^N\}_{N\in\N}$ form a tight family.
\end{proposition}

\begin{proof}
We prove tightness of the forward component, whose law is denoted by $\mu_N$. Fix $L,T,\eta>0$. For $\delta>0$, set
$$
L_D=\left\{k\eta/2:k=0,\ldots,\lceil2L/\eta\rceil\right\},
\quad
T_D(\delta)=\left\{-T+m\delta:m=0,\ldots,\lceil2T/\delta\rceil\right\}.
$$
For $(a,r)\in\h$, define
$$
B_{\delta,\eta}(a,r):=
\left\{
K\in\cH\;\middle|\;
\begin{array}{l}
	\exists\,(f,r')\in K,\quad \exists\,y\ge x\ge r',
	\\[4pt]
	\text{s.t.}
	\begin{cases}
		(f(x),x)\in R(a,r;\eta/2,\delta),\\
		y\in[x,x+\delta],\\
		|f(y)-f(x)|\ge\eta
	\end{cases}
\end{array}
\right\},
$$
where $R(a,r;\eta,\delta):=[a-\eta/2,a+\eta/2]\times[r,r+\delta]\cap\h$.
By the argument in~\cite[Proposition~B.1]{BW04}, as used in the proof of
Proposition~B.2 there, it suffices to show that
\begin{equation}\label{eq:tight_goal}
	\lim_{\delta\downarrow0}\limsup_{N\to\infty}
	\mu_N\!\left(
	\bigcup_{(a,r)\in L_D\times T_D(\delta)}
	B_{\delta,\eta}(a,r)
	\right)=0.
\end{equation}

Fix $(a,r)\in L_D\times T_D(\delta)$ and set $r_1=r_2=r-3\delta/8$. For $a>0$, take $a_1=a-3\eta/8$, $a_2=a+3\eta/8$ and define, for $i=1,2$,
$$
 B_i:=\left\{\sup_{x\in[r_i,r+2\delta]}
 |S^N_{r_i,x}(a_i)-a_i|<\eta/16\right\},
 \qquad D^N_{\delta,\eta}(a,r):=B_1\cap B_2.
$$
For $a=0$, take only $a_2=3\eta/8$ and set $D^N_{\delta,\eta}(0,r)=B_2$. In both cases, non-crossing gives
$$
 \mu_N(B_{\delta,\eta}(a,r))\le\p\bigl((D^N_{\delta,\eta}(a,r))^c\bigr).
$$
Let $Z^i$ be a BESQ$^1$ process started from $2(2\alpha+1)a_i$. Theorem~\ref{thm:1B_toth} and the Portmanteau theorem therefore give
\begin{align*}
 \liminf_{N\to\infty}\p(B_i)\ge\p\left(\sup_{0\le x\le r+2\delta-r_i}\left|\frac{Z^i(x)}{2(2\alpha+1)}-a_i\right|<\frac{\eta}{16}\right)\ge1-Ce^{-c/\delta}.
\end{align*}
The last inequality follows from the Brownian maximal inequality, using the representation of $Z^i$ as the square of a Brownian motion started from $\sqrt{2(2\alpha+1)a_i}$. Since $a_i\le L+\eta$ and $r+2\delta-r_i<5\delta/2$,
the constants $C,c>0$ depend only on $L,\eta$ and $\alpha$.

Consequently, $\liminf_{N\to\infty}\p(D^N_{\delta,\eta}(a,r))\ge1-2Ce^{-c/\delta}$. A union bound over the finite grid gives, for $0<\delta\le1$,
$$
 \limsup_{N\to\infty}
 \mu_N\!\left(\bigcup_{(a,r)\in L_D\times T_D(\delta)}B_{\delta,\eta}(a,r)\right)
 \le C\delta^{-1}e^{-c/\delta}\longrightarrow0,
$$
where $C$ may also depend on $T$. This proves~\eqref{eq:tight_goal}. The same argument applies to the backward component, giving tightness in $\cH^D$.
\end{proof}

\bibliographystyle{abbrvurl}
\bibliography{PSRM}

\end{document}